\documentclass{amsart}
\usepackage{amssymb,stmaryrd}
\usepackage{amsfonts}
\usepackage{amstext}
\usepackage{algorithmic}
\usepackage{algorithm}
\usepackage{young}
\usepackage{graphicx}
\usepackage{epstopdf}
\usepackage[all]{xy}
\usepackage{enumerate}

\usepackage{color}
\usepackage{epic}

\numberwithin{equation}{section}
\newtheorem{theorem}{Theorem}[section]
\newtheorem{lemma}[theorem]{Lemma}
\newtheorem{proposition}[theorem]{Proposition}
\newtheorem{corollary}[theorem]{Corollary}

\theoremstyle{definition}
\newtheorem{definition}[theorem]{Definition}
\newtheorem{example}[theorem]{Example}

\theoremstyle{remark}

\newcommand{\R}{{\mathbb{R}}}

\newcommand{\C}{{\mathbb{C}}}
\newcommand{\Z}{{\mathbb{Z}}}

\newcommand{\N}{{\mathbb{N}}}

\newcommand{\<}{{\langle}}
\renewcommand{\>}{{\rangle}}

\newcommand{\cg}{{\mathfrak{g}}}

\newcommand{\CC}{{\mathcal{C}}}
\newcommand{\CL}{{\mathcal{L}}}

\newcommand{\isom}{{\cong}}
\newcommand{\Ad}{{\rm Ad}}
\newcommand{\End}{{\rm End}}
\newcommand{\Sh}{{\rm Sh}}

\renewcommand{\ker}{{\rm{ker}}}

\newcommand{\tens}{\otimes}
\newcommand{\id}{{\rm id}}

\renewcommand{\o}{{}_{(1)}}
\renewcommand{\t}{{}_{(2)}}
\renewcommand{\th}{{}_{(3)}}
\newcommand{\four}{{}_{(4)}}
\newcommand{\extd}{{\rm d}}

\newcommand{\eps}{\epsilon}

\newcommand{\und}{\underline}

\newcommand{\la}{{\triangleright}}
\newcommand{\ra}{{\triangleleft}}

\newcommand{\lbiprod}{{>\!\!\!\triangleleft\kern-.33em\cdot}}
\newcommand{\rbiprod}{{\cdot\kern-.33em\triangleright\!\!\!<}}

\newcommand{\rcross}{{\triangleright\!\!\!<}}
\newcommand{\rcocross}{{\blacktriangleright\!\!<}}

\begin{document}

\title[Duality for Generalised Differentials on Quantum Groups]{Duality for Generalised Differentials on Quantum Groups and Hopf quivers}
\keywords{noncommutative geometry, quantum group, Hopf algebra, differential calculus, bicovariant, quiver, Hopf quiver, crossed module, duality}
\subjclass[2000]{Primary 81R50, 58B32, 20D05}

\author{Shahn Majid \& Wen-qing Tao}
\address{Queen Mary University of London\\
School of Mathematical Sciences, Mile End Rd, London E1 4NS, UK}
\email{s.majid@qmul.ac.uk, w.tao@qmul.ac.uk}
\thanks{The second author is supported by the China Scholarship Council}

\date{Version 3: May 2013}

\begin{abstract}
We study generalised differential structures $\Omega^1,\extd$ on an algebra $A$, where $A\tens A\to \Omega^1$ given by $a\tens b\to a\extd b$ need not be surjective.
The finite set case corresponds to quivers with embedded digraphs,  the Hopf algebra left covariant case to pairs $(\Lambda^1,\omega)$ where $\Lambda^1$ is a right module and $\omega$ a right module map, and the Hopf algebra bicovariant case corresponds to morphisms  $\omega:A^+\to \Lambda^1$ in the category of right crossed (or Drinfeld-Radford-Yetter) modules over $A$.  When $A=U(\cg)$ the generalised left-covariant differential structures are classified by cocycles $\omega\in Z^1(g,\Lambda^1)$.  We then  introduce and study the dual notion of a codifferential structure $(\Omega^1,i)$ on a coalgebra and for Hopf algebras the  self-dual notion of a strongly bicovariant differential graded algebra $(\Omega,\extd)$  augmented by a codifferential $i$ of degree $-1$.  Here  $\Omega$ is a graded super-Hopf algebra extending the Hopf algebra $\Omega^0=A$ and, where applicable, the dual super-Hopf algebra gives the same structure on the dual Hopf algebra.   We show how to construct such objects from first order data, with both a minimal construction using braided-antisymmetrizes and a maximal one using braided tensor algebras and with dual given via braided-shuffle algebras. The theory is applied to quantum groups with $\Omega^1(C_q(G))$  dually paired to $\Omega^1(U_q(\cg))$, and to finite groups in relation to (super) Hopf quivers.   \end{abstract}

\maketitle

\section{Introduction}

We recall that in noncommutative geometry a `space' is replaced by a `coordinate' algebra and we define the differential structure algebraically.  However, whereas on $\R^n$ and other Lie groups there is a unique translation-invariant calculus and this tends to be transferred throughout geometry, uniqueness is not the case in noncommutative geometry and this leads to a genuine degree of freedom. This can be formulated as a differential algebra, meaning an algebra $A$ equipped with an $A-A$ bimodule $\Omega^1$ and an `exterior derivative' $\extd: A\to \Omega^1$ obeying the Leibniz rule
\[ \extd (ab)= (\extd a)b+a\extd b,\quad \forall a,b\in A\]
along with a `surjectivity axiom' that $\phi:A\tens A\to \Omega^1$, $a\tens b\mapsto a\extd b$ is surjective. One says that the calculus is connected if $\ker\extd=k1$ where $k$ is the ground field. The surjectivity axiom ensures that any calculus is a quotient of the universal one $\Omega^1_{univ}=\ker(m:A\tens A\to A)$ by a sub-bimodule. This is because the above map remains surjective when restricted to $\Omega^1_{univ}$ and then becomes a bimodule map. One is also interested in extending $\Omega^1$ to a differential graded algebra $\Omega=\oplus_{i\ge 0}\Omega^i$ with $\Omega^0=A$ and $\extd$ extending to a degree 1 super-derivation such that $\extd^2=0$. The cohomology of this complex could be viewed as the `noncommutative de Rham cohomology' of the differential algebra $A$ and its extension (although this term is also used for other more specific constructions). The use of differential graded algebras goes back to Quillen and others in the 1970s and is now common to most approaches to noncommutative geometry.

In spite of its successes, this theory is unnecessarily restricted and we now consider a natural generalisation where the surjectivity is dropped. There turn out to be many natural situations where this occurs and where we still have a differential complex of interest. One still has a standard differential calculus $\bar\Omega^1\subseteq\Omega^1$ defined as the image of $\phi$ and our interest is in what happens to the rest of $\Omega^1$ particularly when we extend to $\Omega$. Section~2 covers this general theory at first order level including the example of $A$ the algebra of functions on a finite set $X$. Here standard calculi correspond to digraphs on $X$ while our generalised ones are given by quivers containing digraphs. Section~2 also generalises the theory of bicovariant differential calculi  on Hopf algebras $A$ which has been around for more than 20 years now\cite{Wor}. We recall that standard bicovariant  calculi correspond to right Ad-coaction stable right ideals in  the augmentation ideal $A^+$ (the latter is the kernel of the counit of the Hopf algebra). The  generalised bicovariant first order differential calculi case is similar but consists in the richer data $(\Lambda^1,\omega)$, where $\Lambda^1$ is a right  $A$-crossed module (or Drinfeld-Radford-Yetter-module, basically a module of Drinfeld's celebrated quantum double of $A$), and $\omega$ is a crossed-module morphism $\omega:A^+\to \Lambda^1$. Section~2.3  characterises when the generalised calculus is inner and the theory is illustrated by the case $A=U(\cg)$ where $\cg$ is a Lie algebra and $U(\cg)$ is its enveloping algebra regarded as a noncommutative space (quantising $\cg^*$ with its canonical Poisson structure). Here left-covariant calculi correspond to $\Lambda^1$ a right $\cg$-module and $\omega\in Z^1(\cg,\Lambda^1)$ (Proposition~\ref{Ug}) and are automatically bicovariant for a trivial coaction. The inner case is the case where $\omega$ is exact, which by the Whitehead lemma is always the case if $\cg$ is complex semisimple and $\Lambda^1$ is finite-dimensional.  Standard differential structures would require $\omega$ surjective, which would not be very natural on this context.

Section~3 studies the extension to a differential complex $\Omega$, in terms of extending $\Lambda^1$ to a graded algebra $\Lambda$ of left-invariant forms equipped with a degree $1$ super-derivation (Proposition~\ref{extalg}) obeying certain properties. For a full theory we are led to introduce (Section~3.2) a notion of a {\em strongly bicovariant} differential exterior algebra $(\Omega,\extd)$ on a Hopf algebra $A$, defined as a graded super-Hopf algebra equipped with $\extd$ such that $\extd^2=0$ and $\extd$ a degree 1 super-derivation and super-coderivation. It is known for standard calculi that the Woronowicz exterior algebra is a super-Hopf algebra\cite{Brz} and indeed this appears as an example in our new more general approach to bicovariant differential exterior algebras on Hopf algebras. Theorem~\ref{theorem-bi}  gives an equivalence with $(\Lambda,\delta)$ a  braided-super Hopf algebra in the braided category of crossed modules equipped with $\delta$ a degree $1$ super-derivation with certain properties. We also consider when the calculus is inner in the sense that there exists $\theta\in \Omega^1$ so that $\theta$ generates $\extd$ by graded commutator. Proposition~\ref{genworon} gives  sufficient data $(\Lambda^1,\theta)$ to generate such a generalised exterior super-Hopf algebra $\Omega(A,\Lambda^1)$. Here  $\Lambda^1$ as a crossed module and $\Omega=A\rbiprod B_-(\Lambda^1)$ where $B_\pm(\Lambda^1)$ is the respectively symmetric or antisymmetric braided(-super) Hopf algebra associated to an object in an abelian braided category cf\cite{Ma:fre,Ma:book,Ma:dcalc,Ma:dbos} (also called a Nichols-Worononowicz algebra in some contexts\cite{Baz}). Section~3.3 puts this into a wider context of a `universal inner calculus' $\Omega_\theta$ of which the generalised inner Woronowicz one is a quotient. We also find, remarkably, that the braided-shuffle algebra gives a canonical strongly bicovariant differential exterior algebra $A\rbiprod \Sh_-(\Lambda^1)$ for any right $A$-crossed module $\Lambda^1$ and any crossed module map $\omega:A^+\to\Lambda^1$ (Proposition~\ref{shuffle-diff}).

As a rather novel application of generalised bicovariant differentials on Hopf algebras we consider in Section~4 a theory of  triples $(\Omega, \extd, i)$ where $(\Omega,\extd)$ is a strongly bicovariant exterior algebra and $i$ is a super-derivation and super-coderivation of degree $-1$ with $i^2=0$, i.e. what is geometrically interior product by a vector field. When all components are finite-dimensional we have a dual such triple $(\Omega^*,i^*,\extd^*)$ on $A^*$, i.e. a super-Hopf algebra duality for generalised differential structures whereby  $\extd$ on $A$ is nothing but a `vector field' on $A^*$.  This duality of differential structures is a new construction in noncommutative geometry which can also apply even in the standard surjective setting (as dual to $i$ injective in a certain sense). Among the more unexpected results, Corollary~\ref{deltawor} shows that if an augmented first order calculus extends to higher order on $A\rbiprod B_-(\Lambda^1)$ then it does so uniquely  (this is surprising because we do not assume surjectivity of the first order $\extd$). Moreover, this super-Hopf algebra is dually paired with $H\rbiprod B_-(\Lambda^1{}^*)$ if $H$ is dually paired with $A$, a result which we also show for the symmetric version $B_+(\Lambda^1)$. Here $B_\pm(\Lambda^1)$ is dually paired with $B_\pm(\Lambda^1{}^*)$ as a version of the duality in \cite{Ma:book,Ma:dcalc,Ma:dbos}. We also find a `maximal' strongly bicovariant coinner codifferential exterior algebra $A\rbiprod B_\theta(\Lambda^1)$ (Proposition~\ref{subshuffle-codiff}) and give conditions for both this and $\Omega_\theta$ to be augmented. They turn out to be dual constructions to each other, Corollary~\ref{univdual}.

We believe that this duality provides a new point of view even in classical differential geometry. For example a 1-cocycle on a classical Lie group $G$ (expressed algebraically) provides a natural augmentation or strongly bicovariant codifferential on the classical exterior algebra $\Omega(G)$. The dual of this is a natural strongly bicovariant differential exterior algebra $\Omega(U(\cg))$ on the enveloping algebra. Conversely the classical differential calculus of $G$ corresponds to a particularly simple codifferential structure on $U(\cg)$, see Proposition~\ref{Ugclass} and Proposition~\ref{Gclass}.

In Section~5 we similarly apply this duality in the noncommutative case. Notably, in Corollary~\ref{Galmost} the almost commutative noncommutative differential calculus associated to a Laplacian\cite{Ma:bh},  in the case $\C[G]$ of the algebraic form of a semisimple Lie group $G$ with Lie algebra $\cg$ is shown to be augmented and  dual to a certain augmented differential calculus on the enveloping algebra $U(\cg)$ regarded as a noncommutative space. The latter deforms Proposition~\ref{Ugclass} in the inner case and its natural augmentation $i: \cg\oplus \C c\hookrightarrow U(\cg)^+$, where $c$ is the quadratic Casimir, giving a clear picture of the origin of the differential calculus on the $\C[G]$ side. Similarly, Corollary~\ref{coquasi} gives an understanding of the known construction\cite{Ma:cla,BauSch,Ma:dcalc} of a standard bicovariant differential calculus $\Omega^1(C_q(G))$ on quantum group coordinate algebras $C_q(G)$ from a matrix representation $\rho$ of $U_q(\cg)$, now as dual of an augmentation $i$ on $\Omega^1(U_q(\cg))$. Also, the standard construction depends on the quantum group enveloping algebras $U_q(\cg)$ being essentially factorisable, which is only true in a formal deformation theory setting. We find that the construction is much more natural now as a generalised differential calculus that does not depend on this.  We also find that $\Omega^1(C_q(G))$ is augmented in the formal deformation-theory setting.

Finally, Section~6 specialises this theory to the case of $A=k(G)$, $G$ a finite group and its dual $kG$, the group algebra. The data $(\Lambda^1,\omega)$ becomes a Hopf quiver datum $Q(G,R)$ in the sense of Cibils and Rosso\cite{cr2} containing a Cayley digraph  $\bar Q$ together with further data. Here $\Omega$ in the strongly bicovariant case becomes a quotient of the path super-Hopf algebra of \cite{Hua}. Section~6.2 gives the dual version with $A=kG$ and we illustrate our duality on the augmented bicovariant exterior algebras.

\section{First order generalised differentials}

We work over a general field $k$ of characteristic not 2 (the restriction here is for convenience and in most places is not necessary). We define a generalised differential algebra as an algebra $A$ equipped with an $A-A$-bimodule $\Omega^1$ and a linear map $\extd:A\to \Omega^1$ obeying the Leibniz rule.
This is a standard differential algebra if the induced map $\phi:A\tens A \to \Omega^1$ is surjective.  Note that $\bar\Omega^1=\phi(A\tens A)$ and $\extd$ provide a standard differential algebra contained in any generalised one. We say that a generalised differential calculus is `inner' if there exists an element $\theta\in \Omega^1$ such that
\[ [\theta, a]=\extd a,\quad\forall a\in A,\]
where $[\theta,a]=\theta a-a \theta.$ This is the same definition as in the standard case but can be weaker than saying that $(\bar\Omega^1,\extd)$ is inner as we do not require that $\theta\in \bar\Omega^1$.

\subsection{Quivers case}

We start for orientation purposes with the finite set case.

\begin{proposition}\label{setinner} Let $X$ be a finite set. The generalised first order differential calculi on $A=k(X)$ are in 1-1 correspondence with the following data:
\begin{enumerate}
\item $\Omega^1=\bigoplus_{x,y\in X}\Omega^1_{x,y}$ is a bigraded vector space with components labelled by $X\times X$
\item $\theta=\sum_{x,y\in X}\theta_{x,y}\in \Omega^1$ with $\theta_{x,x}=0$ for all $x\in X.$
\end{enumerate}
The calculus is necessarily inner. The standard sub-calculus has $\bar\Omega^1=\oplus_{x,y\in X} k\theta_{x,y}$, i.e. $\dim\bar\Omega_{x,y}=1$ whenever $\theta_{x,y}\ne 0$, and $\extd=[\theta,\ ]$.
\end{proposition}

\proof By considering the Kronecker $\delta_x$ it is immediate the bimodules are of the bigraded form stated with $f\omega=f(x)\omega$ and $\omega f=f(y)\omega$ for all $f\in k(X)$ and $\omega\in \Omega^1_{x,y}$.  So the only issue is  the data that gives $\extd$. We let $\extd \delta_x=\theta(x)=\sum_{y,z}\theta_{y,z}(x)$ for some collection of 1-forms $\{\theta(x)\}$. The Leibniz rule is
\[ \delta_x\extd \delta_x=\extd\delta_x-(\extd\delta_x)\delta_x.\]
This equation $(\ )\delta_x$ gives $\delta_x(\extd \delta_x)\delta_x=0$ hence
\[ \theta_{x,x}(x)=0.\]
If $y\ne x$, the same equation $\delta_y(\ )$ gives $\delta_y\extd \delta_x=\delta_y(\extd\delta_x)\delta_x$ hence
\[ \sum_{z\ne x}\theta_{y,z}(x)=0,\quad\forall y\ne x\]
As these terms are all in different spaces we conclude that
\[ \theta_{y,z}(x)=0,\quad \forall y,z\ne x.\]
Next, if $y\ne x$ we have $\delta_y\extd \delta_x+(\extd \delta_y)\delta_x=0$. Given the result already obtained, this implies
\[ \theta_{y,x}(x)+\theta_{y,x}(y)=0, \quad\forall y\ne x.\]
Putting these results together we conclude that
\begin{eqnarray*} \extd f&=&\sum_{y\ne x} f(x)(\theta_{x,y}(x)+ \theta_{y,x}(x))=\sum_{y\ne x}( f(x)\theta_{x,y}(x)+ f(y) \theta_{x,y}(y))\\
&=&\sum_{y\ne x}( f(y)-f(x))\theta_{x,y}(y)=[\theta,f]\end{eqnarray*}
where
\[ \theta=\sum_{y\ne x}\theta_{x,y}(y).\]
Conversely, given any $\theta\in \Omega^1$ we define $\extd f=[\theta,f]$ or equivalently we define $\extd\delta_y$ according to $\theta_{x,y}(y)=\theta_{x,y}$ and $\theta_{x,y}(x)=-\theta_{x,y}$. If follows from the formula  for $\extd$ that the standard sub-calculus $\bar\Omega^1$ has a digraph form (as it must do), where $x\to y$ if $\theta_{x,y}(y)\ne 0$.

In fact, the data $\theta$ making $(\Omega^1,\extd)$ inner here is not unique. For any two data $\theta$ and $\theta'$ in $\Omega^1$, $\extd f=[\theta,f ]=[\theta',f]$ for all $f\in k(X)$ iff $\theta-\theta'\in\bigoplus_{x\in X}\Omega^1_{x,x}.$ So we can always assume $\theta_{x,x}=0,\,\forall\,x\in X$ for a given derivation map $\extd.$\endproof

\begin{corollary}\label{calcset} Up to isomorphism, the generalised first order differential calculi $(\Omega^1,\extd)$ on $k(X)$ are in 1-1 correspondence with pairs $(\bar Q,R)$ where $\bar Q$ is a digraph on $X$ and $R=(R_{x,y})$ is an assignment $R_{x,y}\in\N_0$ for all $x,y\in X$ with $R_{x,y}\ge 1$ if $x\to y$ is an arrow in $\bar Q$. Here $\bar Q$ corresponds to the standard sub-calculus $(\bar\Omega,\extd)$.  \end{corollary}
\proof
Given the corresponding data $(\Omega^1,\theta)$ in Proposition~\ref{setinner}, we construct data $R=(R_{x,y})_{x,y\in X}$ and $r=(r_{x,y})_{x,y\in X},$ where $R_{x,y}:=\dim_k\Omega^1_{x,y}\in\mathbb{N}_0$, $r_{x,y}:=1$ if $\theta_{x,y}\neq 0$, and $r_{x,y}:=0$ if $\theta_{x,y}=0$ for all $x,y\in X.$ Note that $r_{x,y}\le R_{x,y},\,r_{x,y}\in\{0,1\}$ and $r_{x,x}=0$ for all $x,y\in X.$  The datum $r$ is the same thing as specifying the digraph $\bar Q$, i.e. a quiver on $X$ with no self-loops and at most one arrow between any two vertices.

We claim that any two calculi $(\Omega^1,\theta)$ and $({\Omega^1}',\theta')$ are isomorphic iff the corresponding data $(R,r)$ and $(R',r')$ are equal. Note that $\theta_{x,y}=\delta_x\extd\delta_y$ for any $x,y\in X,\,x\neq y.$ Assume $\psi:\Omega^1\to{\Omega^1}'$ is a $k(X)$-bimodule isomorphism such that $\psi(\extd \delta_x)=\extd'\delta_x.$ Then
$\psi(\Omega^1_{x,y})=\Omega^1{}'_{x,y}$ and $\psi(\theta_{x,y})=\theta'_{x,y}$ for all $x,y\in X.$ Clearly, $R_{x,y}=R'_{x,y}$ and $\theta_{x,y}\neq0$ iff $\theta'_{x,y}\neq0,$ for any $x,y\in X.$ Conversely, suppose the data $(R_{x,y},r_{x,y})=(R'_{x,y},r'_{x,y})$ for any two calculi. When $r_{x,y}=0$ we can freely choose vector space isomorphism $\psi:\Omega^1_{x,y}\to \Omega^1{}'_{x,y}$ and when $r_{x,y}=1$ we can find a vector space isomorphism $\psi:\Omega^1_{x,y}\to \Omega^1{}'_{x,y}$ such that $\psi(\theta_{x,y})=\theta'_{x,y}$ for any choices of nonzero vectors $\theta_{x,y}$ and $\theta'_{x,y}$ in their respective spaces, these being nonzero precisely when $r_{x,y}=r'_{x,y}=1$. Then $\psi$ is a $k(X)$-bimodule isomorphism and $\psi(\extd\delta_y)=\sum_{x\in X}\psi(\delta_x\extd\delta_y)=\sum_{x\in X}\psi(\theta_{x,y})=\sum_{x\in X}\theta'_{x,y}=\extd'(\delta_y)$ for any $y\in X.$

Given the data $(R,r)$ such that $r_{x,y}\in\{0,1\},\,r_{x,x}=0,\,r_{x,y}\le R_{x,y}$ we can certainly construct a quiver pair $\bar Q\subseteq Q$ where $\bar Q$ is $r$ regarded as defining a digraph  and $Q$ has $R_{x,y}$ arrows from $x$ to $y$. Then the next corollary provides for the existence of $(\Omega,\extd)$ from the data $\bar Q\subseteq Q$. \endproof

In the following we will consider $\bar Q$ a digraph contained in a quiver $Q$ with the same base $Q_0=\bar Q_0=X$. One can represent this by marking some of the arrows of $Q$ with a $*$, namely those in $\bar Q$. For a quiver $Q$, we denote $kQ_1$ the space spanned by all the arrows of $Q,$ and ${}^xkQ_1{}^y$ is the subspace of $kQ_1$ spanned by all the arrows from $x$ to $y.$

\begin{corollary}\label{canonicalform}
Associated to a digraph-quiver pair  $\bar Q\subseteq Q$ on a finite set $X$ is a generalised differential calculus on $k(X)$ given by $\Omega^1=kQ_1=\bigoplus_{x,y\in X}{}^xkQ_1{}^y$ and $\extd=[\theta,\ ]:k(X)=kQ_0\to kQ_1$ where $\theta$ is the sum of all arrows in $\bar{Q}$. Every generalised differential calculus on $k(X)$ is isomorphic to such a  `quiver  calculus' canonical form. \end{corollary}
\proof This is immediate from Proposition~\ref{setinner} where $\Omega_{x,y}={}^xkQ_1{}^y$,  $\theta_{x,x}=0$ for any $x\in X,$ and $\theta_{x,y}=x{\buildrel *\over\to}y$ (the distinguished arrow from $\bar Q$). Clearly any other calculus is isomorphic to one of this form by  Corollary~\ref{calcset}. \endproof

Note also that a morphism of quiver generalised calculi in Corollary~\ref{canonicalform}, from one associated to $\bar Q\subseteq Q$ to one associated to $\bar Q'\subseteq Q'$ say, means a linear map ${}^xkQ_1{}^y\to {}^xkQ'_1{}^y$ for every $x,y\in X$ sending, where present, the distinguished arrow on one side to the distinguished arrow on the other side. This entails that $\bar Q\isom \bar Q'$ as digraphs. As calculi on $k(X)$ can be taken in this form,  isomorphism classes of calculi are therefore given by the choice of $\bar Q$ and the number of arrows $|{}^xQ_1{}^y|$ for each $x,y\in X$, which is the data in Corollary~\ref{calcset}. In particular, the different embeddings of $\bar Q\subseteq Q$ all give isomorphic calculi.

\begin{example} Let $X$ be a finite set, $\Omega^1(X)$ a symmetric digraph calculus, and
\[ (\Delta f)(x)=2\sum_{y:x\to y}(f(x)-f(y))g_{y\to x}\]
be the graph Laplacian for any nonzero weights edge $g_{y\to x}$. These coefficients have a geometrical interpretation as a metric with
\[ (\ ,\ ):\Omega^1\tens_{k(X)}\Omega^1\to k(X),\quad (\omega_{x\to y},\omega_{y'\to x'})=\delta_{x,x'}\delta_{y,y'} g_{y\to x}\delta_x\]
where $\omega_{x\to y}$ are the basis elements over $k$ of $\Omega^1$, labelled by directed edges. The Laplacian obeys
\[ \Delta(f g)=( \Delta f)g+f\Delta g+2(\extd f,\extd g).\]
Given such a second order operator one has a `quantisation' of this standard calculus to a generalised one $(\tilde\Omega^1,\tilde\extd)$ with $\tilde\Omega^1(X)=k(X)\theta'\oplus\Omega^1(X)$ with new bimodule structure \cite{Ma:bh}
\[ f\bullet \omega=f\omega,\quad \omega\bullet f= \omega f+ \lambda(\omega,\extd f)\theta',\quad \tilde\extd f= \extd f+ {\lambda\over 2}(\Delta f)\theta'\]
where $\lambda\in k$ is a parameter. According to the above, this is isomorphic to a quiver calculus. We show that this quiver consists of the original graph with the addition of all the identity loops $x\to x$.

Thus, in our example
\[ \tilde\extd \delta_x= \sum_{y\to x}\omega_{y\to x}-\sum_{x\to y}\omega_{x\to y}+\lambda g(x)\delta_x\theta'-\lambda\sum_{y\to x}g_{x\to y}\delta_y\theta'\]
where $g(x)=\sum_{x\to y} g_{y\to x}$. From this we compute $\delta_x(\tilde\extd \delta_y)\bullet\delta_y$ and arrive at
\[ \theta_{x,y}(y)=\omega_{x\to y}-\lambda g_{y\to x}\delta_x\theta',\quad \theta=\sum_{x\to y}\omega_{x\to y}- \lambda g \theta'\]
with zero if $x,y$ are  not adjacent. When $x\to y$ we have 1-dimensional $\tilde\Omega^1_{x,y}\subset\tilde\Omega^1$ spanned by $\theta_{x,y}(y)$, which is deformed from $\Omega^1_{x,y}$ (which was spanned by $\omega_{x\to y}$) by the $\lambda$ term. In addition we have $\tilde\Omega_{x,x}\subset\tilde\Omega^1$ which are
1-dimensional with basis $\delta_x\theta'$. These subspaces need to be in $\tilde\Omega^1$ but together we obtain a decomposition of all of it.  We see that it has the quiver form where we add the self-loops. Finally, the standard sub-calculus $\bar\Omega^1=\oplus_{x\to y}\tilde\Omega^1_{x, y}$ is clearly isomorphic to the original calculus $\Omega^1(X)$.  \end{example}

\subsection{Left covariant and bicovariant Hopf algebra case}

When $A$ is a Hopf algebra we can ask for $\Omega^1$ to be a bicomodule, i.e. there are commuting coactions
\[ \Delta_L:\Omega^1\to A\tens\Omega^1,\quad \Delta_R:\Omega^1\to \Omega^1\tens A\]
and we require these to be bimodule maps, where $A$ acts by the tensor product of the actions on  $\Omega^1$ and on $A$ by multiplication.
In addition we ask $\extd$ to be a bicomodule map. We then say that the generalised calculus is {\em bicovariant}.  If we are given only (say)
$\Delta_L$ then we say that the calculus is {\em left covariant}. Note that unlike the standard case, in the generalised theory covariance is additional structure not
a property as these coactions, if they exist, need not be unique.

We recall~\cite{Ma:book, Rad:pro} that a right $A$-crossed module (also called Drinfeld-Radford-Yetter or quantum double module) is a vector space $V$ which is both an $A$-right module, denoted $\ra$, and an $A$-right comodule, denoted $\Delta_R$, such that
\[ \Delta_R(v\ra a)=v_0\ra a\t\tens Sa\o v_1 a\th,\quad \forall a\in A,\ v\in V; \ \Delta_Rv=v_0\tens v_1\]
 Morphisms between crossed modules are maps which commute with both the action and coaction.  If $A$ has bijective antipode then the category of right $A$-crossed modules is braided with braiding
\[ \Psi:V\tens W\cong W\tens V,\quad \Psi(v\tens w)= w_0\tens v\ra w_1,\quad \forall v\in V,\ w\in W\]
between any two crossed modules. We will not need the braiding until later.

We let $A^+$ denote the augmentation ideal, defined as the kernel of the counit. This forms an $A$-crossed module with
\begin{equation}\label{Acrossed} a\ra b=ab,\quad \Delta_R(a)=a\t\tens Sa\o a\th ,\quad\forall a\in A^+,\ b\in A\end{equation}
i.e., the right regular action and adjoint coaction respectively. We let $\pi:A\to A^+$,  $\pi a=a-\eps(a)$ be the counit projection.

\begin{theorem}\label{Hopfcalc} Let $A$ be a Hopf algebra. Generalised left covariant differential calculi on $A$ are isomorphic to ones of the form $\Omega^1=A\tens\Lambda^1$ and $\extd a=a\o\tens\omega\circ\pi (a\t)$, given by any data $(\Lambda^1,\omega)$ where $\Lambda^1$ is a right $A$-module and $\omega:A^+\to \Lambda^1$ a right module map. Generalised bicovariant differential calculi are given by $(\Lambda^1,\omega)$ where $\Lambda^1$ is a right $A$-crossed module and $\omega$ is a morphism in the category of right crossed modules. The image $\bar\Lambda^1=\omega(A^+)$ and $\omega$ give the standard sub-calculus. \end{theorem}
\proof The first part of the proof is routine; A left covariant $\Omega^1$ is a left Hopf module (i.e. a left module, a left comodule with left coaction a left module map). By the Hopf module lemma, a left Hopf module $\Omega^1$ is isomorphic to $A\tens \Lambda^1$ where $\Lambda^1$ is given by the space of left invariant elements of $\Omega^1$. The map from $A\tens\Lambda^1$ is given by the left action and its inverse is $m\mapsto m_{-2}\tens (Sm_{-1}).m_0$ where we use a usual notation for the left coaction and $.$ is the left action.  In the other direction, given any vector space $\Lambda^1$ we have a left Hopf module on $A\tens \Lambda^1$ by left multiplication of $A$ and the left coaction of $A$. Next, and also well-known, under this isomorphism the right Hopf module structure on $\Omega^1$ transfers to a crossed module structure on $\Lambda^1$. Thus given a crossed module we give $A\tens\Lambda^1$ a left Hopf module and also a right Hopf module structure by
\[ (a\tens v).b=a b\o \tens v\ra b\t,\quad \Delta_R(a\tens v)=a\o\tens v_0\tens a\t v_1\]
in terms of the crossed module structure. The new part of the proof is to more carefully analyse the content of $\extd:A\to\Omega^1$. Under our isomorphism this transfers to a map $\extd: A\to A\tens\Lambda^1$ necessarily of the form $\extd a=a\o\tens \tilde\omega(a\t)$ for some map $\tilde\omega:A\to \Lambda^1$ defined by $\tilde\omega(a)=Sa\o \extd a\t$, the properties of which can then be deduced.

Equivalently and more explicitly, let $\Lambda^1$ be a crossed module and let $\extd:A\to A\tens \Lambda^1$  be a linear map, which we write as $\extd a=a^1\tens a^2$. We define $\tilde\omega=(\eps\tens\id)\extd:A\to \Lambda^1$ and conversely left covariance of $\extd$ in the form $a\o\tens a\t{}^1\tens a\t{}^2=a^1\o\tens a^1\t\tens a^2$ implies by applying $\eps$ in the middle factor that $a\o\tens\tilde\omega(a\t)=\extd a$, so left covariant $\extd$ is equivalent to a linear map $\tilde\omega$. That $\extd$ obeys the product rule is $(ab)^1\tens (ab)^2=a^1 b\o\tens a^2\ra b\t + a b^1\tens b^2$ which implies
\[ \tilde\omega(ab)=\tilde\omega(a)\ra b+\eps(a)\tilde\omega(b).\]
That $\extd$ is right covariant is $a\o{}^1\tens a\o{}^2\tens a\t=a^1\o\tens a^2{}_0\tens a^1\t a^2{}_1$. Applying the counit to the first factor gives $\tilde\omega(a\o)\tens a\t=\tilde\omega(a\t)_0\tens a\o \tilde\omega(a\t)_1$ which is equivalent to $\tilde\omega:A\to \Lambda^1$ being equivariant where $A$ has the right adjoint coaction.  Conversely, one can check that 	 these properties for $\tilde\omega$ imply that $\extd$ is a differential for $A\tens\Lambda^1$. Clearly the image $\bar\Lambda^1={\rm image}(\tilde\omega)$ is a sub-crossed module of $\Lambda^1$ and one can check that $A\tens\bar\Lambda^1$ is the standard sub-calculus inside $A\tens\Lambda^1$.

Finally, it is convenient (but not necessary) to note that $\tilde\omega(1)=0$ (due to $\extd(1)=0$ and hence $\tilde\omega=\omega\circ\pi$  and the two conditions on $\tilde\omega$ in terms of $\omega:A^+\to \Lambda^1$ become that it is a morphism in the category of right modules respectively crossed modules for the two cases.    \endproof

Note that $I=\ker\,\omega$ will be a right ideal in $A^+$ (ad-invariant in the bicovariant case) but this information determines only $\bar\Omega^1$ not all of $\Omega^1$ in the generalised case. This is the main difference in the generalised setting compared to the Woronowicz theory in \cite{Wor}.

\begin{corollary}\label{Hopfcorol} Up to isomorphism generalised (bi)covariant differential calculi in Theorem~\ref{Hopfcalc} are classified by isomorphism classes of pairs $(\Lambda^1,\omega)$ (two such pairs are isomorphic if they are as objects and the isomorphism forms a commutative triangle with $A^+$ in the category of $A$-(crossed) modules.) \end{corollary}
\proof
Use the previous notations, we show that $(\Omega^1,\extd)$ and $({\Omega^1}',\extd')$ are isomorphic as generalised (bi)covariant differential calculi if and only if the corresponding data $(\Lambda^1,\omega)$ and $({\Lambda^1}',\omega')$ are isomorphic in the category of (crossed) $A$-modules. On one hand, let $\varphi:\Omega^1\to {\Omega^1}'$ be the Hopf bimodule isomorphism, then $\varphi:\Lambda^1\to {\Lambda^1}'$ is an isomorphism of (crossed) modules, where $\varphi(\Lambda^1)={\Lambda^1}'$ since $\varphi$ is a left module map. Then $\varphi\circ\tilde\omega(a)
=\varphi(Sa_{(1)}\extd(a_{(2)}))=Sa_{(1)}\varphi(\extd(a_{(2)}))
=Sa_{(1)}\extd'(a_{(2)})=\tilde\omega'(a)$ implies $(\Lambda^1,\omega)$ and $({\Lambda^1}',\omega')$ are isomorphic pairs. On the other hand, if $\varphi:\Lambda^1\to {\Lambda^1}'$ is a (crossed) module isomorphism compatible with $\omega,\omega',$ one can define $\Phi: A\tens \Lambda^1\to A\tens {\Lambda^1}'$ maps $a\tens v$ to $a\tens \varphi(v).$ Then $\Phi$ is an $A$-bimdoule and left $A$-comodule map obviously. When $\varphi$ is right $A$-comodule map, then $\varphi(v)_0\tens\varphi(v)_1=\varphi(v_0)\tens v_1$ implies $a_{(1)}\tens \varphi(v)_0\tens a_{(2)}\varphi(v)_1=a_{(1)}\tens \varphi(v_0)\tens a\t  v_1.$ This shows $\Phi$ is right $A$-comodule map as request. Lastly, from $\Phi\circ\extd(a)=\Phi(a_{(1)}\tens \omega\circ\pi(a_{(2)}))=a_{(1)}\tens \varphi\circ\omega\circ\pi(a_{(2)})=a_{(1)}\tens\omega'\circ\pi(a_{(2)})=\extd'a$. We proved that $\Phi$ is a (bi)covariant differential calculus isomorphism.
\endproof

\subsection{Innerness of generalised bicovariant differential calculi}

For the next results we will refer to the invariant subspace under the right action,
\[ \Lambda^1_A=\{\eta\in\Lambda^1\ |\ \eta\ra a=\eta\eps(a),\quad\forall a\in A\}=\{\eta\in \Lambda^1\ |\ \eta\ra A^+=0\}.\]
This is the subspace in $\Omega^1$ which is left-invariant and central for the bimodule structure. In the bicovariant case the crossed module condition ensures that $\Delta_R(\Lambda^1_A)\subseteq \Lambda^1\square {}_{Ad}A$ where $A$ has the left adjoint action. We also have $\Lambda^1_A\tens 1\subseteq\Lambda^1\square A$. Here if $V_R,{}_LV$ are right, left $A$-modules the $\square$ is  defined as
\[ V_R\square {}_LV=\{\sum v_i\tens w_i\ |\ \sum v_i\ra a\tens w_i=\sum v_i\tens a\la w_i,\ \forall a\in A\}\subseteq V_R\tens {}_LV\]
(in the bimodule case this gives a new bimodule but we are not using that here). If the antipode of $A$ is bijective then one can also think of $\Lambda^1\square A=(\Lambda^1\tens A)_A$ the invariants for the tensor product action where $A$ acts on $A$ in our case by $b\ra a= S^{-1}a\t b a\o$ (this is the left adjoint action converted to right via the inverse antipode).

\begin{lemma}\label{propinner} A generalised left covariant first order differential calculus in Theorem~\ref{Hopfcalc} is inner if and only if there exists $\theta\in \Lambda^1$ such that $\omega(a)=\theta\ra\,a$ for any $a\in A^+$. This inner calculus is bicovariant iff  we have $\Delta_R$ making $\Lambda^1$ a crossed module with $\Delta_R\theta-\theta\tens 1\in \Lambda^1\square A$.
\end{lemma}
\proof
If $(A\tens\Lambda^1,\extd)$ given by pair $(\Lambda^1,\omega)$ is inner, then there exists $\tilde{\theta}\in A\tens\Lambda^1$ such that $\extd=[\tilde{\theta},\ ].$ Set $\theta=\eps\tens\id (\tilde{\theta})\in \Lambda^1.$ Then
$\tilde{\omega}(a)=\eps\tens\id$ $(\tilde{\theta}.a-a.\tilde{\theta})$ $
=\theta\ra \,a-\eps(a)\theta=\theta\ra\,\pi(a).$ Hence $\omega(a)=\theta\ra\,a$ for $a\in A^+.$ Conversely, given such an element $\theta$ then clearly $\extd a=a\o\tens\tilde{\omega}a\t=a\o\tens\theta\ra \pi(a\t)=(1\tens\theta).a-a.(1\tens\theta)=[\theta,a]$ as required. Moreover,  if $\omega(a)=\theta\ra\,a$ and if we have a crossed module then the condition that $\omega$ is a right $A$-comodule map, which is is equivalent to $\tilde\omega:A\to \Lambda^1$ a right $A$-comodule map,  is $\theta\ra\pi(a\t)\tens (Sa\o) a\th= \theta_0\ra \pi(a)\t\tens S\pi(a)\o\theta_1 \pi(a)\th$  for all $a\in A$. Explicitly, this is $\theta_0\ra a\t\tens Sa\o \theta_1 a\th-\theta\ra a\t\tens Sa\o a\th= \eps(a)(\Delta_R\theta-\theta\tens 1)$ for all $a\in A$, i.e.  $\Delta_R\theta-\theta\tens1\in \Lambda^1\square A.$ \endproof

\begin{proposition}\label{propinner-iso} Inner generalised left covariant (resp. bicovariant) differential calculi on $A$ are classified by isomorphism classes of pairs $(\Lambda^1,[\theta])$ where $\Lambda^1$ is a right (resp. right crossed) $A$-module and $[\theta]\in \Lambda^1/\Lambda^1_A$ (with $\Delta_R[\theta]=[\theta]\tens 1$ in $(\Lambda^1\tens A)/(\Lambda^1\square A)$ for the bicovariant case). \end{proposition}
\proof We chose representatives for $[\theta]$ and $[\theta']$ and obtain equivalence $(\Lambda^1,\theta)\sim (\Lambda^1{}',\theta')$ if there is a right module isomorphism $\varphi:\Lambda^1\to \Lambda^1{}'$ such that $\varphi(\theta)-\theta'\in\Lambda^1_A$. Similarly, we show that inner generalised bicovariant calculi up to isomorphism correspond to pairs $(\Lambda^1,\theta)$  where $\Lambda^1$ is an $A$-crossed module and $\Delta_R\theta-\theta\tens 1\in \Lambda^1\square A$ and equivalence requires in addition that $\varphi$ is a comodule map. One direction of the proof here is covered by Lemma~\ref{propinner}. Conversely, given $\Lambda^1$ and $\theta\in\Lambda^1$ we define $\omega:A^+\to \Lambda^1$ by $\omega(a)=\theta\ra\,a.$ It is obvious that $\omega$ is a right $A$-module map. If $\Lambda^1$ is a crossed module then $\Delta_R\theta-\theta\tens1\in \Lambda^1\square A$ implies that  $\omega$ is a morphism in the category of right $A$-crossed modules as required in the bicovariant case. The isomorphism classes of $(\Lambda^1,\omega)$ in the two cases reduce to the equivalences claimed. Note that by the remarks above the condition on $\theta$ for bicovariance depends only on $[\theta]$ as both $\Delta_R$ and $\id\tens 1$ descend. One can then interpret the result as stated, where an isomorphism class means a morphism $\varphi$ such that $\varphi([\theta])=[\theta']$ on the relevant quotient spaces. \endproof

Next we note that as $\bar\Omega^1\subseteq\Omega^1$, the latter being inner is a weaker statement than $\bar\Omega^1$ being inner as it may be that $\theta\notin\bar\Lambda^1$.

\begin{proposition}\label{quasint} For a generalised left covariant calculus on $A$ in Theorem~\ref{Hopfcalc}, the standard sub-calculus $\bar\Omega^1$  is inner  iff there exists an element $\mu\in A$ such that  $\mu A^+\subseteq \ker\omega$ and $\eps(\mu)=1$.  Then $\theta=\omega(1-\mu)$.
\end{proposition}
\proof By Lemma~\ref{propinner} applied to the standard sub-calculus, being inner is equivalent to the existence of $\theta\in \bar\Lambda^1$ with $\omega(a)=\theta\ra a$ for all $a\in A^+$. Then $\theta=\omega(1-\mu)$ for some $1-\mu\in A^+$, such that $\omega(1-\mu)\ra a=\omega(a)$ for all $a\in A^+$. But $\omega$ is a right module map so this is equivalent to $\omega(\mu a)=0$ for all $a\in A^+$. \endproof

This observation appears to be new even for standard differential calculi. A corollary of it is, however, well-known. Namely, any finite-dimensional semisimple $A$ has a normalised integral $\mu$ so that $\mu a=0$ for all $a\in A^+$ and $\eps(\mu)=1$, hence any left-invariant calculus on such an $A$ is inner. Geometrically, such $\mu$  corresponds  to right-invariant integration $\int: A^*\to k$ and Proposition~\ref{quasint} says that more generally what we need is a `partially right invariant' integration, namely when restricted to $I^\perp$.

We conclude with an elementary example.

\begin{proposition}\label{Ug} Let $\cg$ be a Lie algebra and $A=U(\cg)$.  Generalised left covariant differential structures $\Omega^1(U(\cg))$ correspond to Lie algebra cocycles $\omega\in Z^1(\cg,\Lambda^1)$ where $\Lambda^1$ is a right $\cg$-module. Coboundaries correspond to inner calculi. The bimodule relations and exterior derivative are
\[ v\xi-\xi v=v\ra\xi,\quad \extd \xi=\omega(\xi),\quad\forall v\in \Lambda^1,\ \xi\in\cg.\]
\end{proposition}
\proof Let $\Lambda^1$ be a right $\cg$-module. A right-module map $\omega:U(\cg)^+\to \Lambda^1$ is fully determined by its restriction to $\cg\subset U(\cg)^+$ since any other element is a sum of products of the form $\xi x$ where $\xi\in\cg$ and $x\in U(\cg)^+$. The restriction obeys $\omega([\xi,\eta])=\omega(\xi)\ra\eta-\omega(\eta)\ra\xi$ for all $\xi,\eta\in\cg)$, i.e. a 1-cocycle. Conversely, given a 1-cocycle $\omega:\cg\to \Lambda^1$ we extend this map as $\omega(x\xi)=\omega(x)\ra\xi$ for all $x\in U(\cg)^+$ and $\xi\in\cg$. Suppose this $\omega$ is well defined on degree $\le n$. Then if $x$ has degree $\le n-1$ we have $\omega$ defined on degree $\le n+1$ because $\omega(x(\xi\eta-\eta\xi)):=\omega(x\xi)\ra\eta-\omega(x\eta)\ra\xi=(\omega(x)\ra\xi)\ra\eta-(\omega(x)\ra\eta)\ra\xi=\omega(x)\ra [\xi,\eta]$ by definition of $\omega$ on degree $\le n+1$ for the first equality,  by the right module property already established at lower degree for the second equality, and by right $\cg$-module property of $\Lambda^1$ for the last equality. The result is consistent with the right module property at degree $\le n$.  The case of a coboundary is $\omega(\xi)=\theta\ra\xi$ for some $\theta\in \Lambda^1$. This implies that $\omega(\xi x)=\omega(\xi)\ra x=(\theta\ra\xi)\ra x=\theta\ra (\xi x)$ where we use the right action of $U(\cg)$ induced by the Lie algebra action. Hence by induction $\omega(x)=\theta\ra x$ for all $x\in U(\cg)^+$ and the calculus is inner. The converse is immediate. \endproof

In this context it is not particularly natural to require that $\omega$ is surjective, i.e our notion of a generalised differential calculus is more appropriate. One can also ask for $\Lambda^1$ to be a right $U(\cg)$-crossed module and $\omega$ to intertwine so as to give a bicovariant calculus. Thus we suppose a right coaction $\Delta_R:\Lambda^1\to \Lambda^1\tens U(\cg)$ which is a right module map where $\cg$ acts on $U(\cg)$ by right commutator (this is the content of the crossed-module condition in the cocommutative case). Then as the adjoint coaction on $U(\cg)^+$ is trivial, the condition for a bicovariant calculus is that $\Delta_R\circ\omega=\omega\tens 1$.  In the inner case this amounts to $\Delta_R\theta-\theta\tens 1 \in \Lambda^1\tens U(\cg)$ being killed under the right action of $\cg$. This illustrates many of the results above.

\section{Exterior algebra of a generalised bicovariant calculus}

A differential graded algebra or differential exterior algebra $(\Omega,\extd)$ (extending a given generalised differential calculus $(\Omega^1,\extd)$ over an algebra $A$) means a graded algebra $\Omega=\oplus_{n\ge 0}\Omega^n$ with given $\Omega^1$ and $\Omega^0=A,$ and a degree $1$ `super-derivation' $\extd:\Omega\to \Omega$ in the sense \[\extd(uv)=(\extd u)v+(-1)^nu\extd v\] for all $u\in\Omega^n,v\in\Omega$ extending $\extd:A\to\Omega^1$  such that $\extd^2=0$. The differential exterior algebra is inner if $\extd=[\theta, \ \}$ for some $\theta\in\Omega^1$, where $[\theta,u\}=\theta u-(-1)^n u\theta$ for any $u\in\Omega^n$. The standard case is with the surjectivity at first order and $\Omega$ generated by $A,\Omega^1$. The following is immediate:

\begin{lemma} Suppose that $\Omega^1$ extends to a differential exterior algebra $\Omega$ generated by $A,\Omega^1$. Then the graded subalgebra $\bar\Omega\subseteq \Omega$ generated by $A,\bar\Omega^1$ is a standard differential exterior algebra. 
\end{lemma}

Our main results will be in the Hopf algebra case, where we formulate the correct notion of `strongly bicovariant differential exterior algebra' applicable to the generalised case. Given $\Lambda^1$, we give in particular a canonical `universal' inner construction, a canonical shuffle algebra construction and an analogue of the standard `Woronowicz'  construction for exterior algebras.

\subsection{Left covariant and bicovariant differential exterior algebras}

When $A$ is a Hopf algebra,
we first consider the left covariant case where we suppose that $\Omega$ is a left comodule algebra with $\Delta_L$ graded and restricting to the coproduct on $A$ (and to $\Delta_L$ on $\Omega^1$ if we are extending a given left covariant $\Omega^1$).

\begin{proposition}\label{extalg} Let  $(\Lambda^1,\omega)$ give a generalised left covariant differential calculus in the setting of Theorem~\ref{Hopfcalc}. Left covariant differential exterior algebra $(\Omega,\extd)$ extending $\Omega^1$ are in correspondence with pairs $(\Lambda,\delta)$ where $\Lambda$ extends $\Lambda^1$ as a graded right $A$-module algebra with $\Lambda^0=k$, and $\delta:\Lambda\to \Lambda$ is a degree $1$ super-derivation such that
\[ \delta^2=0,\quad (\delta\eta)\ra a-\delta(\eta\ra a)=\omega(\pi(a\o))(\eta\ra a\t)-(-1)^{|\eta|}\eta\ra a\o \omega (\pi(a\t)),\quad \forall \,a\in A\]
\[ \delta\omega(a)+\omega(\pi(a\o))\omega(\pi(a\t))=0,\quad \forall a\in A^+.\]
\end{proposition}
\proof  By the Hopf module lemma, the left $A$-Hopf module $\Omega\cong A\rcross \Lambda$ for the graded algebra of left-invariant differential forms. To be specific, here $\Lambda=\oplus_{n\ge 0}\Lambda^n$ with $\Lambda^n:={}^{\mathrm{co} A}\Omega^n$ the left invariant subspaces of $\Omega^n$ for all $n.$ Note that $\Lambda^n$'s are right $A$-modules by $v\ra a=Sa\o.v.a\t.$ Because $\Delta_L(vw)=1\otimes vw$ for all $v\in\Lambda^n $ and $w\in\Lambda^m,$ so $vw\in\Lambda^{n+m}.$ Then $(vw)\ra a=Sa\o.vw.a\t=Sa\o.v.a\t Sa\th.w.a\four=(v\ra a\o)(w\ra a\t)$ for all $v,w\in\Lambda,\,a\in A$ implies that $\Lambda$ is an $\mathbb{N}$-graded right $A$-module algebra. In fact, the left $A$-Hopf modules as well as right $A$-modules isomorphism $\beta:\Omega\to A\rcross \Lambda$ is an algebra isomorphism. This follows from
$\beta(v)\beta(w)=(v_{-2}\otimes Sv_{-1}v_0)\cdot (w_{-2}\otimes Sw_{-1}w_0)
=v_{-2}w_{-3}\otimes (Sv_{-1}v_0)\ra w_{-2} Sw_{-1}w_0
=v_{-2}w_{-4}\otimes Sw_{-3}Sv_{-1}v_0w_{-2}Sw_{-1}w_0=v_{-2}w_{-2}\otimes Sv_{-1}Sw_{-1}v_0w_0=\beta(vw),$ where $\cdot$ denotes the multiplication in the smash product algebra $A\rcross \Lambda.$  Under this isomorphism, the super-derivation on $\Omega$ transfer to a super-derivation $\extd$ on $A\rcross\Lambda,$ which is also a left comodule map.

The super-derivation $\extd$ with $\extd^2=0$ on $A\rcross \Lambda$ necessarily has the form $\extd (a\tens \eta)=a\o \tens \tilde\omega(a\t)\eta+a\tens\delta\eta$ for the linear map $\delta:\Lambda\to \Lambda$ obtained as the restriction of $\extd$ to left-invariant forms. Obviously, $\delta^2=0$ as $\extd^2$ does. Since $\extd^2 a=\extd (a\o\otimes\tilde{\omega}(a\t))=a\o\otimes\tilde{\omega}(a\t)\tilde{\omega}(a\th)+a\o\otimes\delta\tilde{\omega}(a\t),$ so $\extd^2 a=0$ is equivalent to $\delta\tilde{\omega}(a)+\tilde{\omega}(a\o)\tilde{\omega}(a\t)=0$ for any $a\in A,$ which is also equivalent to $\delta\omega(a)+\tilde{\omega}(a\o)\tilde{\omega}(a\t)=0$ for any $a\in A^+.$

For any $\eta=1\otimes\eta$ and $a=a\otimes1$ in $A\rcross \Lambda,$ $\extd(\eta\cdot a)=(\delta\eta)\cdot a+(-1)^{|\eta|}\eta\cdot\extd a=a\o\otimes(\delta\eta)\ra a\t+(-1)^{|\eta|}a\o\otimes(\eta\ra a\t)\tilde\omega(a\th),$ while $\extd(a\o\otimes\eta\ra a\t)=a\o\otimes\tilde\omega(a\t)(\eta\ra a\th)+a\o\otimes\delta(\eta\ra a\t).$ Then $\extd(\eta\cdot a)=\extd(a\o\otimes\eta\ra a\t),$ from $\eta\cdot a=a\o\otimes\eta\ra a\t$ in $A\rcross \Lambda$, is equivalent to $(\delta\eta)\ra a-\delta(\eta\ra a)=\tilde\omega(a\o)\eta\ra a-(-1)^{|\eta|}(\eta\ra a\o)\tilde\omega(a\t).$

The `if part' is also true. For any pair $(\Lambda,\delta)$ in the setting, one can define the super-derivation on $A\rcross \Lambda$ by $\extd (a\tens \eta)=a\o \tens\tilde\omega(a\t)\eta+a\tens \delta\eta$ for all $a\in A$ and $\eta\in\Lambda.$ The only need to show $\extd^2(a\otimes\eta)=0$ for all $a\in A$ and $\eta\in\Lambda,$ which follows from $\extd^2(a\otimes\eta)=a\o\otimes\tilde\omega(a\t)\tilde\omega(a\th)\eta
+a\o\otimes\delta(\tilde\omega(a\t)\eta)+a\o\otimes\tilde\omega(a\t)\delta\eta
+a\tens\delta^2\eta=a\o\otimes\tilde\omega(a\t)\tilde\omega(a\th)\eta+a\o\otimes\delta\tilde\omega(a\t)\eta=0.$
\endproof

Next, we study the bicovariant case. We say the differential graded algebra $(\Omega,\extd)$ on a Hopf algebra $A$ is bicovariant if there exist $\Delta_L,\Delta_R$ on $\Omega$ making it a bicomodule algebra, and $\extd$ is a bicomodule map. Here $\Delta_{L,R}$ are required to be graded and restrict to the given coproduct of $A$ and to coactions on $\Omega^1$ if a given one is being extended.

\begin{proposition}\label{extalg-bi}
In the setting of Proposition~\ref{extalg} with $(\Lambda^1,\omega)$ bicovariant, the left covariant differential exterior algebra $(\Omega,\extd)$ is bicovariant iff $\Lambda$ is an algebra in the category of right $A$-crossed modules, and the super-derivation $\delta:\Lambda\to \Lambda$ satisfying the conditions in Proposition~\ref{extalg} is a right $A$-comodule map.
\end{proposition}
\begin{proof}
The first part is routine, under the isomorphism $\Omega\cong A\rcross\Lambda$ in Proposition~\ref{extalg}, the right Hopf module structure on $\Omega$ is equivalent to a right crossed module structure on $\Lambda.$ If the product of $\Omega$ is a right $A$-comodule map, then so is the product of $\Lambda,$ and vice verse. Noting $\delta$ is a restriction of $\extd$, the super-derivation $\delta$ is a right $A$-comodule map as $\extd$ does. Conversely, if $\delta$ is a right $A$-comodule map, then $\Delta_R(\extd(a\otimes\eta))=a\o\otimes\tilde\omega(a\th)_0\eta_0\otimes a\t\tilde\omega(a\th)_1\eta_1+a\o\otimes(\delta\eta)_0\otimes a\t\otimes(\delta\eta)_1=a\o\otimes\tilde\omega(a\four)\eta_0\otimes a\t S(a\th) a_{(5)}\eta_1+a\o\otimes\delta\eta_0\otimes a\t\eta_1=a\o\otimes\tilde\omega(a\t)\eta_0\otimes a\th\eta_1+a\o\otimes\delta\eta_0\otimes a\t\eta_1=\extd(a\o\otimes\eta_0)\otimes a\t\eta_1=((\extd\otimes\mathrm{id})\circ\Delta_R)(a\otimes\eta),$ which means $\extd$ a right $A$-comodule. Recall that $\tilde\omega:A\to\Lambda^1$ is a right $A$-comodule map.
\end{proof}

We now study when the differential exterior algebra is inner. We have
\begin{proposition}\label{ext-inner}  Let $(\Lambda^1,\theta)$ define an inner generalised left covariant differential calculus and suppose that $\Lambda^1$ extends to an algebra $\Lambda$ in the category of right $A$-modules. Then $\delta\eta=[\theta,\eta\}$ meets the conditions in Proposition~\ref{extalg} and we have a left covariant differential exterior algebra $(\Omega,\extd)$ {\em iff}  $ \theta^2\in \Lambda_A\cap Z(\Lambda)$ (i.e. $\theta^2\ra A^+=0$ and $\theta^2$ central in $\Lambda$). This $(\Omega,\extd)$ is bicovariant {\em iff}
 $(\Lambda^1,\theta)$ defines an inner generalised bicovariant differential calculus,  $\Lambda$ is an algebra in the category of right $A$-crossed modules and $[\Delta_R\theta-\theta\otimes1,\Delta_R\eta\}=0$ for all $\eta\in\Lambda$.
\end{proposition}
\proof
We show that $\delta$ defined by $\theta$ meets all the conditions of Proposition~\ref{extalg} and \ref{extalg-bi}. Clearly $\delta^2\eta=\delta(\theta\eta-(-1)^{|\eta|}\eta\theta)=\theta^2\eta-\eta\theta^2=0$ for all $\eta\in\Lambda$ requires $\theta^2$ to be central in $\Lambda$.
Next, note that $\Lambda$ is right $A$-module algebra, we have
$(\delta\eta)\ra a-\delta(\eta\ra a)
=(\theta\eta)\ra a-(-1)^{|\eta|}(\eta\theta)\ra a-\theta(\eta\ra a)+(-1)^{|\eta|}(\eta\ra a)\theta
=(\theta\ra a\o)(\eta\ra a\t)-(-1)^{|\eta|}(\eta\ra a\o)(\theta\ra a\t)+(-1)^{|\eta|}(\eta\ra a)\theta
=\tilde\omega(a\o)(\eta\ra a\t)-(-1)^{|\eta|}$ $(\eta\ra a\o)\tilde\omega(a\t)$ for all $a\in A.$ Now, for any $a\in A^+,$ $\tilde\omega(a\o)\tilde\omega(a\t)=(\theta\ra\pi(a\o))(\theta\ra\pi(a\t))
$ $=(\theta\ra(a\o-\epsilon(a\o))(\theta\ra(a\t-\epsilon(a\t))
=(\theta\ra a\o)(\theta\ra a\t)+(\theta\ra \epsilon(a\o))(\theta\ra \epsilon(a\t))-(\theta\ra a)\theta-\theta(\theta\ra a)=\theta^2\ra a-(\theta\ra a)\theta-\theta(\theta\ra a)$  while $\delta\omega(a)=\delta(\theta\ra a)=\theta(\theta\ra a)+(\theta\ra a)\theta$ for any $a\in A^+,$ so $\delta\omega(a)+\tilde\omega(a\o)\tilde\omega(a\t)=0$ holds for all $a\in A^+$ precisely when $\theta^2\in \Lambda_A$ (this is also immediate from $\extd^2 a=0$).
For the bicovariant case, the equation $[\Delta_R\theta-\theta\otimes1,\Delta_R\eta\}=0$ is equivalent to $\delta=[\theta,\ \}$ being a right $A$-comodule map.
\endproof

We will give a construction shortly, but it will have a further `strongly bicovariant' property in the next subsection.

\subsection{Strongly bicovariant differential exterior algebras}

It is known that standard first order bicovariant differential calculi have a `minimal' extension to a bicovariant differential exterior algebra, due to Woronowicz\cite{Wor}, and which is known\cite{Brz} to be a super-Hopf algebra.  This motivates the following definition for generalised calculi:

\begin{definition}\label{extd} We say that a differential exterior algebra $(\Omega,\extd)$ over a Hopf algebra $A=\Omega^0$ is {\em  strongly bicovariant} if $\Omega$ is a graded super-Hopf algebra with odd/even part given by the parity of the grading and the super-derivation $\extd$ is also a `super-coderivation' in the sense
\begin{equation*}
\Delta\circ\extd(w)=(\extd\otimes\mathrm{id}+(-1)^{|\ |}\otimes\extd)\Delta(w),\quad \forall w\in\Omega,
\end{equation*} where $\Delta=(\ )^1\otimes(\ )^2$ is the graded-super coproduct of $\Omega$ and $(-1)^{|\ |}w=(-1)^{|w|}w$ according to the degree.
\end{definition}
By assumption the coproduct respects the grading so that $\Delta(\Omega^n)\subseteq\oplus_{i,j=0}^{i+j=n}\Omega^i\otimes \Omega^j.$ The super-coderivation condition in Definition~\ref{extd} appears to be a new observation even in the standard case and is key to what follows.  Our terminology is justified by the following lemma.

\begin{lemma}
Any strongly bicovariant differential exterior algebra is bicovariant.
\end{lemma}
\proof
Denote the projection from $\Omega=\oplus_{n\ge0}\Omega^n$ to $\Omega^0=A$ by $\Pi.$ Then $\Delta_L:=(\Pi\otimes\mathrm{id})\Delta$ and $\Delta_R:=(\mathrm{id}\otimes\Pi)\Delta$ make $\Omega$ a graded $A$-bicomodule, from the coassociativity of the graded coproduct $\Delta$. The Hopf bimodule and $A$-bicomodule algebra structure easily follow from the fact that $\Delta$ is algebra map. Apply $\Pi\otimes\mathrm{id}$ (resp. $\mathrm{id}\otimes\Pi$) to the both sides of $\Delta\extd=(\extd\otimes\mathrm{id}\pm\mathrm{id}\otimes\extd)\Delta,$ we have $\Delta_L\extd=(\mathrm{id}\otimes\extd)\Delta_L$ (resp. $\Delta_R\extd=(\extd\otimes\mathrm{id})\Delta_R$). Thus, $\extd$ is an $A$-bicomodule map.
\endproof

We now turn to the construction of examples.

\begin{theorem}\label{theorem-bi} Let $A$ have bijective antipode. 
In the setting of Proposition~\ref{extalg-bi}, the bicovariant differential exterior algebra $(\Omega,\extd)$ is strongly bicovariant iff $\Lambda$ is a graded braided-super Hopf algebra in the category of right $A$-crossed modules  and $\delta$ obeys in addition
\begin{equation}\label{delta}
(\delta\eta)^1\otimes(\delta\eta)^2=\delta\eta^1\otimes\eta^2+(-1)^{|\eta^1|}\eta^1\otimes\delta\eta^2+(-1)^{|\eta^1|}(\eta^1)_0\otimes\omega(\pi((\eta^1)_1))\eta^2
\end{equation}
for all $\eta\in\Lambda$. Here $\underline{\Delta}=(\ )^1\otimes(\ )^2$ denotes the braided coproduct of $\Lambda$ while $\Delta_R=(\ )_0\otimes(\ )_1$ denotes the right coaction on it.
\end{theorem}
\proof
The correspondence between super-Hopf algebra structures is a super version  of Radford's Theorem~\cite{Rad:pro} in the braided-category interpretation due to the first author\cite[Appendix]{Ma:skl}.  Under this super-Hopf algebra isomorphism $\Omega\cong A\rbiprod\Lambda,$ the coproduct on $A\rbiprod\Lambda$ is $\Delta(a\otimes\eta)=a\o\otimes(\eta^1)_0\otimes a\t(\eta^1)_1\otimes\eta^2$ for any $a\in A$ and $\eta\in\Lambda.$ Then light computation shows the bicoactions constructed by $(\Pi\otimes\mathrm{id})\Delta$ and $(\mathrm{id}\otimes\Pi)\Delta$ are the same ones induced from the crossed module structure of $\Lambda$ if and only if  $\Lambda$ as a braided-super Hopf algebra is correspondingly graded. Note that the coproduct necessarily has form: $\underline{\Delta}\eta=1\otimes\eta+\cdots+\eta\otimes1$ for all $\eta\in\Lambda$ by the counity axiom of a coalgebra due to $\eps=0$ except on degree 0, in order to respect the grading. This means in particular that $\underline{\Delta}\eta=\eta\tens 1+1\tens\eta$ for $\eta\in \Lambda^1$.

Note that $\Delta(\eta)=1\otimes(\eta^1)_0\otimes(\eta^1)_1\otimes\eta^2$ in $A\rbiprod\Lambda,$ then the super-coderivation property in Definition~\ref{extd} implies (\ref{delta}) by direct computation. Conversely, recall $\extd(a\otimes\eta)=a\o\otimes \tilde\omega(a\t)\eta+a\otimes\delta\eta,$ and note that $\underline{\Delta}(\tilde\omega(a))=\tilde\omega(a)\otimes1+1\otimes\tilde\omega(a)$ and $\underline{\Delta}(vw)=(-1)^{|v^2||w^1|}v^1(w^1)_0\otimes (v^2\ra (w^1)_1)w^2$ in $\Lambda$ for any $a\in A$ and $\eta,v,w\in \Lambda.$ then the left hand side of the super-coderivation property $\Delta\extd(a\otimes\eta)=a\o\otimes(\tilde\omega(a\t)\eta)^1{}_0\otimes a\o(\tilde\omega(a\t)\eta)^1{}_1\otimes(\tilde\omega(a\t)\eta)^2+a\o\otimes(\delta\eta)^1{}_0\otimes a\t(\delta\eta)^1{}_1\otimes(\delta\eta)^2=a\o\otimes\tilde\omega(a\t)(\eta^1)_0\otimes a\th(\eta^1)_1\otimes\eta^2+(-1)^{|\eta^1|}a\o\otimes(\eta^1)_0\otimes a\t(\eta^1)_1\otimes(\tilde\omega(a\th)\ra (\eta^1)_2)\eta^2+a\o\otimes\delta((\eta^1)_0)\otimes a\t(\eta^1)_1\otimes\eta^2+(-1)^{|\eta^1|}a\o\otimes(\eta^1)_0\otimes a\t(\eta^1)_1\otimes \delta(\eta^2)+(-1)^{|\eta^1|}a\o\otimes(\eta^1)_0\otimes a\t(\eta^1)_1\otimes\tilde\omega((\eta^1)_2)\eta^2,$ which meets the right hand side of super-coderivation property, recalling again that $\tilde\omega(ab)=\tilde\omega(a)\ra b+\epsilon(a)\tilde\omega(b)$ for any $a,b\in A$. \endproof

\begin{proposition} \label{sbic-inner} Let $A$ have bijective antipode. The inner bicovariant differential exterior algebra $(\Omega,\extd)$ in Proposition~\ref{ext-inner} is strongly bicovariant iff $\Lambda$ is a graded braided-super Hopf algebra in the category of right $A$-crossed modules  and $\eta^1\theta_0\tens\eta^2\ra\theta_1=\eta^1\theta\tens\eta^2$ for all $\eta\in\Lambda$.
\end{proposition}
\proof We verify that (\ref{delta}) holds automatically in the inner case. The left hand side $\underline{\Delta}\delta(\eta)=\underline{\Delta}(\theta\eta-(-1)^{|\eta|}\eta\theta)=(-1)^{|\eta^1|}(\eta^1)_0\otimes(\theta\ra (\eta^1)_1)\eta^2+\theta\eta^1\otimes\eta^2-(-1)^{|\eta|}\eta^1\otimes\eta^2\theta-(-1)^{|\eta^1|}\eta^1\theta_0\otimes\eta^2\ra\theta_1.$
And the right hand side $\delta\eta^1\otimes\eta^2+(-1)^{|\eta^1|}\eta^1\otimes\delta\eta^2+(-1)^{|\eta^1|}(\eta^1)_0\otimes\tilde\omega((\eta^1)_1)\eta^2=\theta\eta^1\otimes\eta^2-(-1)^{|\eta^1|}\eta^1\theta\otimes\eta^2+(-1)^{|\eta^1|}\eta^1\otimes\theta\eta^2-(-1)^{|\eta|}\eta^1\otimes\eta^2\theta+(-1)^{|\eta^1|}(\eta^1)_0\otimes(\theta\ra (\eta^1)_1)\eta^2-(-1)^{|\eta^1|}\eta^1\otimes\theta\eta^2.$ This meets the left hand side after cancelling the third and last terms in it and provided the condition on $\theta$ holds. Note that the stated condition on $\theta$ means in particular that  $\theta_0\tens\eta\ra\theta_1=\theta\tens\eta$, i.e. $\Psi(\eta\tens\theta)=\theta\tens\eta$ for all $\eta\in\Lambda^1$ where we recall that the braiding in the right crossed-module case is $\Psi(\eta\tens\zeta)=\zeta_0\tens\eta\ra\zeta_1$. \endproof

It remains to construct $\Lambda$ and to do this we start with the tensor algebra $TV$ associated to any object $V$ in a braided abelian category. As an algebra the product is $\tens$ itself. Moreover, there is known to be a coproduct making $TV$ a Hopf algebra/super-Hopf algebra in the braided category, which we denote respectively as $T_\pm V$, with\cite{Ma:book,Ma:fre}
\[ {\und\Delta}(v_1\tens v_2\tens\cdots\tens v_n)=\sum_{r=0}^n\gamma_{r,n-r}\circ\left[ {n\atop r}; \pm\Psi\right](v_1\tens v_2\tens\cdots\tens v_n),\quad \und{\epsilon}(v_1\tens v_2\tens\cdots\tens v_n)=0\]
where $\Psi:V\tens V\to V\tens V$ is the braiding and we use the natural isomorphism
\[ \gamma_{r,n-r}:{V}^{\tens n}\to {V}^{\tens r}\tens{V}^{\tens (n-r)},\quad v_1\tens\cdots \tens v_n\mapsto (v_1\tens\cdots \tens v_r)\tens(v_{r+1}\tens\cdots \tens v_n)\]
 for any $0\le r\le n$. The braided binomials here are defined recursively by\cite{Ma:book,Ma:fre}
\begin{gather*}
\left[ {n\atop r}; \Psi\right] =\Psi_r\Psi_{r+1}\cdots\Psi_{n-1}(\left[ {n-1\atop r-1};\Psi\right] \tens\id)+\left[ {n-1\atop r}; \Psi\right]\tens\id,\\
\left[ {n\atop 0}; \Psi\right]=\id,\quad \left[ {n\atop r}; \Psi\right]=0\text{ if }r>n,
\end{gather*}
for all $n,r\in\mathbb{N}_0$ where $\Psi_i$ denotes $\Psi$ acting in the $i,i+1$ tensor factors. This particularly defines $\left[ {n\atop n}; \Psi\right]=\id$ for any $n\in\mathbb{N}_0$ and $\left[ {n\atop 1}; \Psi\right]=\Psi_1\Psi_2\cdots\Psi_{n-1}+\left[ {n-1\atop 1}; \Psi\right]\tens\id=\id+\Psi_1+\Psi_1\Psi_2+\cdots+\Psi_1\Psi_2\cdots\Psi_{n-1}=\left[n; \Psi\right]$ the `braided integers'. We have given the structure of $T_\pm V$ concretely but what us actually being specified by the braided binomials is a morphism (one does not need elements.) Also note that $\und\Delta v=v\tens 1+1\tens v$ in degree 1, so these are examples of additive braided(-super ) Hopf algebras. One may make quadratic and other quotients\cite{Ma:book}, notably
\[ B_\pm(V)=T_{\pm}V/\oplus_n \ker [n,\pm\Psi]!,\quad [n,\Psi]!=(\id\tens [n-1,\Psi]!)[n,\Psi]:\Lambda^1{}^{\tens n}\to \Lambda^1{}^{\tens n}\]
The symmetric version $B_+(V)$ in the special case of the category of $A$-crossed modules is sometimes called the Nichols-Woronowicz algebra\cite{Baz} associated to $V$, while the above approach based on braided Hopf algebras and braided factorials is due to the first author. Also, see \cite{Ma:dcalc} for the relationship with the work of Woronowicz\cite{Wor}.

Clearly, given bicovariant $(\Lambda^1,\omega)$ on $A$ and in the case where $\Lambda$ is generated by $\Lambda^1$ we can reduce Theorem~\ref{theorem-bi} to data $\delta_1$ on degree 1 obeying various properties such that this extends as a super-derivation with the required properties. For $\Lambda=B_-(\Lambda^1)$ we can do better and show that $\delta$ if it exists is uniquely determined by the first order calculus, a result which we defer to Section 4.  Here we limit ourselves to the important inner case where $\delta=[\theta,\ \}$, meaning super-commutator. Note that the bosonisation\cite{Ma:book} of  $B_+(\Lambda^1)$ is an ordinary Hopf algebra $A\rbiprod B_+(\Lambda^1)$. The parallel super-bosonization of
  $B_-(\Lambda^1)$ is necessarily a usual super-Hopf algebra which in our case we interpret by an extension of the isomorphism in Theorem~\ref{Hopfcalc}  as the exterior algebra $ \Omega(A)=A\rbiprod B_-(\Lambda^1)$.

\begin{proposition}\label{genworon} Let $A$ have bijective antipode and 
let $\Lambda^1$ be an object in the category of right $A$-crossed modules and $\theta\in\Lambda^1$ be such that $\Delta_R\theta-\theta\otimes1\in \Lambda^1\square A$ as in Lemma~\ref{propinner}. Suppose that
\[  \Psi(\eta\tens\theta)=\theta\tens\eta,\quad \{\Delta_R\theta-\theta\otimes1,\Delta_R(\eta)\}=0\]
 for all $\eta\in\Lambda^1$. Then  $\omega(a)=\theta\ra a$ for all $a\in A^+$ and $\delta=[\theta,\ \}$ provides an inner strongly bicovariant differential exterior algebra $\Omega=A\rbiprod B_-(\Lambda^1)$ according to Proposition~\ref{sbic-inner}.
\end{proposition}
\proof  Note that  $B_-(\Lambda^1)$ is generated by $\Lambda^1$. Suppose $\theta_0\tens\xi\ra\theta_1=\theta\tens\xi$ for any $\xi\in\Lambda^1$, then $\theta_0\tens\theta_1\tens \xi\ra\theta_2=\theta_0\tens\theta_1\tens\xi$
implies $\theta_0\tens\eta\ra\theta_1\tens\xi\ra\theta_2=\theta_0\tens\eta\ra\theta_1\tens\xi=\theta\tens\eta\tens\xi,$
i.e. $\theta_0\tens \eta\ra\theta_1=\theta\tens \eta,\ \forall \eta\in{\Lambda^1}^{\tens 2}.$ Similarly (or by induction) for all  $\eta\in{\Lambda^1}^{\tens k}$ for any power and hence for $\eta\in \Lambda$.  In fact this is just functoriality of the braiding with respect to the product of $\Lambda$.  Then $\eta^1\theta_0\tens\eta^2\ra\theta_1=\eta^1\theta\tens\eta^2$ for any $\eta\in\Lambda.$

Likewise, suppose that  $[\Delta_R\theta-\theta\tens1,\Delta_R\xi\}=0$ is true for  any $\xi\in{\Lambda^1}^{\tens k}.$ Then, for any $\eta'=\eta\tens\xi\in{\Lambda^1}^{\tens (k+1)},$ we have $[\Delta_R\theta-\theta\tens1,\Delta_R(\eta\tens\xi)\}=\{\Delta_R \theta-\theta\tens 1, \Delta_R\eta\}\Delta_R\xi-\Delta_R\eta[\Delta_R\theta-\theta\tens 1,\Delta_R\xi\}=0.$ Hence $[\Delta_R\theta-\theta\tens1,\Delta_R\eta\}=0$ is valid for any $\eta\in\Lambda.$

 The first displayed condition ensures in particular that $\Psi(\theta\tens\theta)=0$ and hence  that $\theta^2=0$ in $B_-(\Lambda^1)$. The other requirements are from the analysis above.\endproof

In particular, all these conditions hold if $\Delta_R\theta=\theta\tens 1$. So any right-invariant element of $\Lambda^1$ gives a strongly bicovariant differential exterior algebra $\Omega=A\rbiprod  B_-(\Lambda^1)$. This  reworks and generalises the Woronowicz approach for standard calculi. The same proof as for Proposition~\ref{genworon}  also applies to any $\Lambda$ generated by $\Lambda^1$ where $\theta^2\ra A^+=0$ and $\theta^2$ central, for example it applies to $B_-^{quad}(\Lambda^1)$ where we just take the degree 2 relations, i.e. we  quotient by $\<\ker(\id-\Psi)\>$ and still have $\theta^2=0$ as in Proposition~\ref{genworon}.

\subsection{Universal and shuffle differential exterior algebras}

Although we have emphasised the `minimal' choice $\Lambda=B_-(\Lambda^1)$, at the other extreme one can also take the following `universal' choice, which we cover in the inner case:

\begin{proposition}\label{univ} Let $A$ have bijective antipode and let $(\Lambda^1,\theta)$ define an inner generalised first order bicovariant differential calculus with $\theta\in \Lambda^1$ such that $\Delta_R\theta=\theta\tens 1$. We take $\Lambda_\theta(\Lambda^1)=T_-\Lambda^1/\<\theta^2\ra a, [\theta^2,\eta]\ |\ a\in A^+,\ \eta\in \Lambda^1\>$. Then $\Omega_\theta(A)=A\rbiprod\Lambda_\theta(\Lambda^1)$ with $\extd=[\theta,\ \}$ is an inner strongly bicovariant calculus. Conversely, any inner strongly bicovariant differential exterior algebra on $A$ with $\theta$ right-invariant and which is generated by its degrees $0,1$ is isomorphic to a quotient of $\Omega_\theta(A)$ for some crossed module $\Lambda^1$.
\end{proposition}
\proof We quotient the braided-super Hopf algebra $T_-\Lambda^1$ by the relations $\theta^2\ra a=0$ for all $a\in A^+$ and $[\theta^2,\eta]=0$ for all $\eta\in \Lambda^1$. Working in the tensor algebra we have $\und\Delta(\theta^2\ra a)=\und\Delta(\theta\ra a\o\tens \theta\ra a\t)=(\theta\ra a\o\tens 1+1\tens\theta\ra a\t)\und\cdot(\theta\ra a\t\tens 1+1\tens \theta\ra a\t)=\theta^2\ra a\tens 1+1\tens\theta^2\ra a$ where the crossed terms cancel in the braided-super tensor product. Here $\Psi(\theta\ra a\o\tens\theta\ra a\t)= \theta\ra a\o\tens\theta\ra a\t$ since $\Psi$  (the braiding in the category of $A$-crossed modules) is a morphism and $\Psi(\theta\tens\theta)=\theta\tens\theta$ by our assumption on $\theta$.  Similarly $\und\Delta(\theta^2 \eta-\eta\theta^2)=(\eta^1\tens\eta^2)\und\cdot(\theta^2\tens 1+1\tens\theta^2)-(\theta^2\tens 1+1\tens\theta^2)\und\cdot(\eta^1\tens\eta^2)=[\eta^1,\theta^2]\tens \eta^2+\eta^1\tens[\eta^2,\theta]-\eta^1{}_0\tens(\theta^2\ra\pi(\eta^1{}_1))\eta^2$ for all $\eta\in T_-\Lambda^1$ from the form of the braiding $\Psi$ and invariance of $\theta$. Hence we  see that the ideal $J$ generated by the relations has braided coproduct in $J \tens T_-\Lambda^1+T_-\Lambda^1\tens J$. Hence we obtain a braided-super Hopf algebra $\Lambda_\theta(\Lambda^1)$ and we then use Proposition~\ref{sbic-inner} to obtain an inner   bicovariant differential exterior algebra $\Omega=A\rbiprod\Lambda_\theta(\Lambda^1)$. \endproof

That $\theta^2\ra a=0$ for all $a\in A^+$ is equivalent to $\theta^2$ commuting with $A=\Omega^0$, so the relations are that $\theta^2$ is central in $\Omega$. Also note that isomorphic inner first order generalised differential calculi have corresponding $\Omega_\theta$ isomorphic as super-Hopf algebras since they are isomorphic as degree $0,1$ and have the corresponding relations coming from $\theta^2$. Clearly Proposition~\ref{genworon} in the case $\Delta_R\theta=\theta\tens 1$ is a quotient of this `universal' one. Another choice in between the two is  $A\rbiprod B_-^{quad}(\Lambda^1)$.

Next, let $V$ be an object of a braided abelian category. We use again the braided binomials but in a different convention, namely
  $\left( {n\atop r}; \Psi\right):{V}^{\tens n}\to{V}^{\tens n}$ defined for all $n,r\in\mathbb{N}_0$ by
\begin{gather*}
\left( {n\atop r}; \Psi\right) =(\left( {n-1\atop r-1};\Psi\right) \tens\id)\Psi_{n-1}\cdots\Psi_{r+1}\Psi_r+\left( {n-1\atop r}; \Psi\right)\tens\id,\\
\left( {n\atop 0}; \Psi\right)=\id,\quad \left( {n\atop r}; \Psi\right)=0\text{ if }r>n.
\end{gather*} Then $\left( {n\atop 1}; \Psi\right)=\id+\Psi_1+\Psi_2\Psi_1+\cdots+\Psi_{n-1}\cdots\Psi_1.$
It is easy to see that if $V$ is finite-dimensional and $\Psi^*:V^*\tens V^*\to V^*\tens V^*$ is the adjoint map of $\Psi,$ then $\left[ {n\atop r}; \Psi\right]^*=\left( {n\atop r}; \Psi^*\right)$ for all $n,r\in\mathbb{N}_0.$  We then define $\Sh_\pm(V)$ to be $TV$ as a graded vector space with coalgebra structure $\underline{\Delta},\,\underline{\epsilon}$ as for the usual shuffle algebra, namely with
\[ \underline{\Delta}(v_1\tens v_2\tens\cdots\tens v_n)=1\underline{\tens}(v_1\tens v_2\tens\cdots\tens v_n)
+\sum_{i=1}^{n-1}(v_1\tens\cdots\tens v_i)\underline{\tens}(v_{i+1}\tens \tens\cdots\tens v_n)
+(v_1\tens v_2\tens\cdots\tens v_n)\underline{\tens}1,\]
\[ \underline{\epsilon}(v_1\tens v_2\tens\cdots\tens v_n)=0\]
for $n\ge 1$, and we define $\bullet_\pm:\Sh_\pm(V)\underline{\tens} \Sh_\pm(V)\to \Sh_\pm(V)$ by $$(v_1\tens\cdots\tens v_r)\bullet_\pm(v_{r+1}\tens\cdots\tens v_n)=\left( {n\atop r}; \pm\Psi\right)(v_1\tens v_2\tens\cdots\tens v_n)$$

\begin{proposition}
Let $V$ be an object in a braided abelian tensor category with braiding $\Psi,$ then $\Sh_\pm(V)$ with product $\bullet_{\pm}$ is a graded braided(-super) Hopf algebra, called the braided shuffle (super-) Hopf algebra on $V$.
\end{proposition}
The proof is omitted as it is just the arrow reversal of the proof for the tensor algebra $T_\pm V$ being a braided(-super) Hopf algebra, with the roles of product and coproduct swapped. The braided shuffle Hopf algebra $\Sh_+(V)$ was also considered by Rosso\cite{Rosso}  where $\sum_{w\in\sum_{r,n-r}}T_w$ or $\mathcal{B}_{r,n-r}$ is used in stead of $\left( {n\atop r}; \Psi\right)$ here.

\begin{proposition}\label{shuffle-diff}
Let $A$ be a Hopf algebra with bijective antipode, $(\Lambda^1,\omega)$ define a generalised bicovariant first order differential calculus on $A.$ Then $\Omega_{sh}(A)=A\rbiprod \Sh_-(\Lambda^1)$ is a strongly bicovariant differential exterior algebra on $A$ with
\[ \extd(a\tens v_1\tens\cdots\tens v_n)=a_{(1)}\tens\omega\circ\pi(a_{(2)})\bullet_{_{-}}(v_1\tens\cdots\tens v_n)+a\tens\delta(v_1\tens\cdots \tens v_n),\]
and $\delta:\Sh_-(\Lambda^1)\to \Sh_-(\Lambda^1)$ defined recursively  by
 \[ \delta(v_1\tens\cdots\tens v_n)=\delta(v_1\tens\cdots\tens v_{n-1})\tens v_n+(-1)^n(v_1)_0\tens\cdots\tens(v_n)_0\tens\omega\circ\pi((v_1)_1\cdots(v_n)_1)\]
  for any $v_1\tens\cdots\tens v_n\in \Sh_-(\Lambda^1).$  Moreover, $\Omega_{sh}(A)$ is inner by $\theta\in \Lambda^1$ if and only if $\theta$ makes $(\Lambda^1,\omega)$ inner and $\Psi(v\tens\theta)=\theta\tens v$ for any $v\in \Lambda^1.$
\end{proposition}
\proof  We set $\Lambda=\Sh_-(\Lambda^1)$ and one can check that $\delta$ defined above satisfied all the properties required in Proposition~\ref{extalg}, \ref{extalg-bi} and Theorem~\ref{theorem-bi}. We note in particular that, $\delta(v)=-v_0\tens\omega\circ\pi(v_1)$ for any $v\in \Lambda^1.$  In terms directly of $\extd$, the induction is
\begin{eqnarray*}
\extd(a\tens v_1\tens\cdots v_n))&=&\extd(a\tens v_1\tens\cdots\tens v_{n-1})\tens v_n\\
&&+(-1)^na_{(1)}\tens(v_1)_0\tens\cdots\tens(v_n)_0\tens\omega\circ\pi(a_{(2)}(v_1)_1\cdots(v_n)_1)\end{eqnarray*}
and $\extd$ must be of this form to be a super-coderivation. Also, one may use alternative formulae such as
\[ \left({n\atop r};\Psi\right)=\id\tens\left({n-1\atop r-1};\Psi\right)+(\id\tens\left({n-1\atop r};\Psi\right))\circ\Psi_1\Psi_2\cdots\Psi_r\]
 to compute the braided shuffle product $\bullet_{_{-}}$. For the innerness, one can show $\delta$ in form of $[\theta,\ \}$ if and only if $\theta_0\tens v\ra\theta_1=\theta\tens v$ for any $v\in \Lambda^1$ by induction. In fact, for any $\theta\in \Lambda^1$ such that $\Delta_R\theta-\theta\tens 1\in \Lambda^1\square A$ and $\Psi(v\tens\theta)=\theta\tens v$ for any $v\in \Lambda^1$, we have $\theta^2=\theta\bullet_-\theta=\theta\tens\theta-\theta_0\tens \theta\ra\theta_1=0$, $[\Delta_R\theta-\theta\tens1,\Delta_R\eta\}=0$ and $\eta^1\theta_0\tens\eta^2\ra\theta_1=\eta^1\theta\tens\eta^2$ for any $\eta\in \Sh_-(\Lambda^1)$ (the latter two can be shown by induction). Then from Proposition~\ref{sbic-inner}, we know $\delta=[\theta,\ \}$ provides an inner strongly bicovariant differential structure on $A\rbiprod \Sh_-(\Lambda^1)$ and this gives the same result. \endproof

\section{Augmented generalised bicovariant differentials and duality}

Let $\Omega$ be a strongly bicovariant differential exterior algebra on a Hopf algebra $A$.  If  the components $\Omega^i$ are all finite dimensional, the graded dual $\Omega^*$ becomes a strongly bicovariant codifferential exterior algebra over $A^*$ via $\extd^*$ as a super derivation and coderivation as before but being degree $-1$. We start with a short study of such `codifferential' calculi.

\subsection{Codifferential structures}

Let $C$ be a coalgebra over $k$. We define a {\em first order codifferential calculus} on $C$ to be a $C$-bicomodule $\Omega^1$ and a linear map $i:\Omega^1\to C$ such that
\[ \Delta\circ i(\eta)=(i\otimes\mathrm{id})\Delta_R\eta+(\id\otimes i)\Delta_L\eta,\quad \forall \eta\in\Omega^1.\]
One may deduce that $\eps\circ i=0$. This is the dual notion to a generalised first order differential algebra. Likewise, we define an {\em codifferential exterior   coalgebra} (or codifferential graded coalgebra) to be a graded coalgebra $\Omega=\oplus_{n\ge 0}\Omega^n$ with $\Omega^0=C$ and in the case where a first order structure is given, coalgebra $\Delta=\Delta_L+\Delta_R$ on degree 1, equipped with $i:\Omega\to \Omega$ of degree $-1$ with $i^2=0$ and obeying the super-coderivation property \[ \Delta\circ i(\eta)=(i\otimes\mathrm{id}+(-1)^{|\ |}\otimes i)\circ\Delta\eta,\quad \forall \eta\in\Omega.\]
One necessarily has $i=0$ when restricted to $C$. We say $\Omega$ is {\em coinner} if there exists an element $\theta^*\in{\Omega^1}^*$ such that
$$i(w)=\<\theta^*,w^1\>w^2+(-1)^n w^1\<\theta^*,w^2\>$$
for any $w\in\Omega^n.$ Here $\Delta=(\ )^1\tens (\ )^2$ denotes the coproduct of the underlying coalgebra of $\Omega$ and $\<\ ,\ \>$ the duality pairing.

In the case where $A$ is a Hopf algebra we have of course the notion of left, right and bi-covariant codifferential calculi with respect to left and right actions of $A$. Thus, a first order bicovariant codifferential structure on $A$ clearly means an $A$-biHopf module (or $A$-Hopf bimodule) $\Omega^1$ together with a bimodule map $i:\Omega^1\to A$ such that
\begin{equation}\label{foi} \Delta\circ i=(i\tens\id)\Delta_R+ (\id\tens i)\Delta_L.\end{equation}
We also note that for any Hopf algebra, $A$ is canonically a right $A$-crossed module in a different way from (\ref{Acrossed}), namely by the right adjoint action and
right regular coaction (given by the coproduct). This projects down to a second $A$-crossed module structure on $A^+$,
\begin{equation}\label{Acocrossed} a\ra b=Sb\o a b\t,\quad \Delta_R=\Delta-1\tens\id,\quad \forall a\in A^+,\,b\in A\end{equation}
different from the one used in Section~2.2.

\begin{lemma}\label{ilambda} A first order bicovariant codifferential calculus over $A$ is isomorphic to one of the form $\Omega^1=A\tens\Lambda^1$  where $\Lambda^1$ is a right $A$-crossed module and $i$ is determined by its restriction  $i:\Lambda^1\to A^+$ as an $A$-crossed module map. The calculus is coinner iff there exists $\theta^*\in {\Lambda^1}^*$ such that $i(\eta)=\<\theta^*,\eta_0\>\eta_1-\<\theta^*,\eta\>1_A$ for all $\eta\in\Lambda^1.$
 \end{lemma}
\proof It is or less immediate from the structure in Section~2 (i.e. by application of the Hopf-module lemma)  that $i$ is determined by its restriction to the left-invariant 1-forms, where we require $i:\Lambda^1\to A^+$ such that
\[ Sa\o i(\eta) a\t=i(\eta\ra a),\quad \Delta i(\eta)=i(\eta_0)\tens\eta_1+1\tens i(\eta),\quad\forall a\in A,\ \eta\in \Lambda^1.\]
Note by applying $\id\tens\eps$ that the second condition indeed entails that the image of $i$ is in $A^+$. We interpret this as a crossed module morphism as stated. As in Section~2 the left $A$-(co)module structure on $\Omega^1$ is the regular (co)action on $A$ and the right one is the tensor product of the regular (co)action and the given one on $\Lambda^1$. We recover $i$ on $\Omega^1$ by extension as a left $A$-module map.  The condition for this to be coinner is clear. \endproof

As to higher degrees, we define a {\em strongly bicovariant codifferential exterior algebra} on $A$ as an $\mathbb{N}_0$-graded super-Hopf algebra extending $A$ in degree 0 and equipped with a degree $-1$, square zero super-derivation and super-coderivation.

We conclude with a canonical construction. We recall that for an object $\Lambda^1$ in an abelian braided category the tensor algebra $T_\pm\Lambda^1$ is a braided(-super) Hopf algebra in the category.

\begin{proposition}\label{tensor-codiff}
Let $(\Lambda^1,i)$ define a first order bicovariant codifferential calculus on a Hopf algebra $A$ with bijective antipode. Then (1) $\Omega_{tens}=A\rbiprod T_-\Lambda^1$ is a strongly bicovariant codifferential exterior algebra on $A$ where $i$ extends to higher degrees as a super-derivation.  Moreover, it is coinner with $\theta^*\in{\Lambda^1}^*$ if and only if its first order $\Omega^1=A\rbiprod\Lambda^1$ is coinner with same $\theta^*\in{\Lambda^1}^*$ and $\<\theta^*,v\>w=w_0\<\theta^*,v\ra w_1\>$ for any $v,w\in\Lambda^1.$
(2) A strongly bicovariant codifferential exterior algebra generated by its degree $0,\,1$ is isomorphic to a quotient of $(\Omega_{tens},i)$ for some crossed module $\Lambda^1$.
\end{proposition}
\proof Extends the first order bicovariant codifferential $i:A\tens\Lambda^1\to A$ to $i:A\tens T_-\Lambda^1\to A\tens T_-\Lambda^1$ recursively by
\begin{equation}\label{univ-codiff}
i(a\tens v_1\tens\cdots\tens v_n)=i(a\tens v_1\tens\cdots\tens v_{n-1})\tens v_n-(-1)^na\tens v_1\tens v_2\tens\cdots\tens v_{n-1}\cdot i(v_n),
\end{equation}
where $\cdot$ here denotes the product of $A\rbiprod T_-\Lambda^1.$ One can show this $i$ has $i(a\tens v\cdot b\tens w)=i(ab\o\tens (v\ra b\t)\tens w)=i(a\tens v)\cdot b\tens w-a\tens v\cdot i(b\tens w)$ for any $a,b\in A,\,v,w\in\Lambda^1,$ thus provide a super-derivation of degree $-1$ by induction. Since $T_-\Lambda^1$ is a braided-super Hopf algebra and freely generated by $\Lambda^1.$  One can also check this $i$ is a super-coderivation with $i^2=0$ (define $i(A)=0$) by induction. The coinner case is shown by directly checking definition. The second part of the statement follows naturally as any such algebra is a quotient of $A\rbiprod T_-\Lambda^1$ as braided-super Hopf algebra and its codifferential must be induced from (\ref{univ-codiff}) as the super-derivation property matters here.\endproof

This is the dual construction to $\Omega_{sh}$ in Proposition~\ref{shuffle-diff}.  Another canonical example is a coinner strongly bicovariant codifferential exterior algebra  dual to $\Omega_\theta$ in Proposition~\ref{univ} constructed as follows. Let $A$ be a Hopf algebra with bijective antipode, $V$ an $A$-crossed module. We let $\theta^*\in V^*$ be such that $\<\theta^*,v\ra a\>=\epsilon_A(a)\<\theta^*,v\>$ for any $v\in V,\,a\in A$. This is a version the condition for $\theta^*$ being right invariant on the adjoint side, so we still call such $\theta^*$ `right invariant'. Then we  define
\begin{align*}
B_{\theta^*}(V)=\{b\in \Sh_-(V)|\ b^1\tens\<\theta^*\tens\theta^*,(b^2)_0\tens(b^3)_0\>(b^2)_1(b^3)_1&\tens b^4\\
=b^1\tens \<\theta^*\tens\theta^*,b^2\tens b^3\> 1_A&\tens b^4\\
=\<\theta^*\tens\theta^*,b^1\tens b^2\> b^3\tens 1_A&\tens b^4\ \}
\end{align*}
where $\underline{\Delta}=(\ )^1\tens(\ )^2$ denotes the coproduct of $\Sh_-(V)$ and $\Delta_R=(\ )_0\tens(\ )_1$ denotes the coaction.  One can check directly that $B_{\theta^*}(V)\subseteq \Sh_-(V)$ as a sub-graded-braided-super Hopf algebra.

\begin{proposition}\label{subshuffle-codiff}
Let $A$ be a Hopf algebra with bijective antipode, $\Lambda^1$ an $A$-crossed module, and $\theta^*\in \Lambda^{1*}$ right invariant.
Then $A\rbiprod B_{\theta^*}(\Lambda^1)$ is a coinner strongly bicovariant codifferential exterior algebra on $A,$ where $i$ is given by
\[ i(a\tens v_1\tens\cdots\tens v_n)=\<\theta^*,(v_1)_0\>a(v_1)_1\tens(v_2\tens\cdots \tens v_n)+(-1)^n a\tens(v_1\tens\cdots \tens v_{n-1})\<\theta^*,v_n\>.\]
Conversely, any coinner strongly bicovariant codifferential exterior algebra $(\Omega,i)$ with $\theta^*$ right invariant and which is cogenerated by its degree $0$ and $1$, is isomorphic to such a sub-codifferential exterior algebra  of $A\rbiprod B_{\theta^*}({\Lambda^1})$ for some crossed module $\Lambda^1$. 
\end{proposition}
\proof It is easy to see that $i$ is a degree $-1$ map, both a super-derivation and super-coderivation as it's in 'coinner' form. The condition $i^2=0$ asks for $\<\theta^*\tens\theta^*,w^1\tens w^2\>w^3=w^1\<\theta^*\tens\theta^*,w^2\tens w^3\>$ for any $w\in A\rbiprod B_{\theta^*}(\Lambda^1),$ which is covered by the defining conditions of $B_{\theta^*}(\Lambda^1).$

Conversely, suppose $(\Omega,i)$ is a coinner strongly bicovariant codifferential exterior algebra on a Hopf algebra A, we say it is `cogenerated by its degree $0$ and $1$' if there is a coalgebra embedding $j$ from $\Omega$ to $\mathrm{CoT}_A{\Omega}^1$ the cotensor coalgebra and $j$ is the identity map when restricted to $A$ and $\Omega^1.$ Recall that the cotensor coalgebra $\mathrm{CoT}_A\Omega^1$ defined by a coalgebra $A$ and its bicomodule $\Omega^1$ is the graded dual construction to the tensor algebra. This graded coalgebra $\mathrm{CoT}_A{\Omega}^1$ admits a graded-super Hopf algebra structure (unique up to isomorphism) induced from the $A$-Hopf bimodule structure of $\Omega^1,$ and one can show by the universal property of the cotensor coalgebra that $j$ is a graded super-Hopf algebra embedding. By the  super version of Radford's theorem cf.\cite{Rad:pro,Ma:skl}, there is a super-Hopf algebra isomorphism $\varphi$ from $\mathrm{CoT}_A{\Omega^1}$ to $A\rbiprod \Sh_-(\Lambda^1)$ for crossed module $\Lambda^1={}^{co A}\Omega^1.$ Hence $\varphi\circ j(\Omega)=A\rbiprod\Lambda$ for some sub-braided-super Hopf algebra $\Lambda$ of $\Sh_-(\Lambda^1).$ In particular, $\varphi\circ j(\Omega^1)=A\tens\Lambda^1,$ meaning $\varphi\circ j$ provides the Hopf module isomorphism.

Now translate the coinner codifferential structure $i$ of $\Omega$ to $A\rbiprod\Lambda.$ Without loss of generality we can assume $\theta^*\in\Lambda^{1*}$ and we suppose as stated that this is also right invariant. One can see that $i$ then necessarily has the same form as  displayed in the statement for $A\rbiprod B_{\theta*}(\Lambda^{1*}).$ The condition that $i^2(w)=0$ for any element $w\in\Omega$ requires that $\<\theta^*\tens\theta^*,(\eta^1)_0\tens(\eta^2)_0\>(\eta^1)_1(\eta^2)_1\tens\eta^3=1_A\tens\eta^1\<\theta^*\tens\theta^*,\eta^2\tens\eta^3\>$ for any $\eta\in\Lambda.$ Note that the coproduct $\underline{\Delta}=(\ )^1\tens(\ )^2$ and the coaction $\Delta_R(\ )=(\ )_0\tens (\ )_1$ of $\Lambda$ are those of $\Sh_-(\Lambda^1).$  Then using the formula above, one can show that any element of $\Lambda$ belongs in $ B_{\theta^*}(\Lambda^1),$ i.e. $\Lambda\subseteq B_{\theta^*}(\Lambda^1)\subseteq \Sh_-(\Lambda^1).$ This finishes the proof.\endproof

This says that the sub-braided-super Hopf algebra $B_{\theta^*}(\Lambda^1)$ is the `maximal' one contained in $\Sh_-(\Lambda^1)$ that is compatible with a coinner codifferential structure.

\subsection{Augmented bicovariant differential structures}

We now consider both differential and codifferential structures at the same time, a self-dual concept.

\begin{theorem}\label{aug} Let $(\Omega,\extd,i)$ be a strongly bicovariant differential exterior algebra equipped with a degree $-1$ super-derivation and super-coderivation $i:\Omega\to \Omega$ with $i^2=0$ (we say $(\Omega,\extd)$ is {\em augmented}). Then (1) $\CL=\extd i + i \extd$ is a degree zero derivation and coderivation of the super-Hopf algebra structure and commutes with $i,\extd$. (2) If $(\Omega,\extd)$ is inner via $\theta\in \Omega^1$ then $\CL=[i(\theta),\ ]$ is also inner.  \end{theorem}
\proof (1) For the derivation property of $\CL$ we expand $\CL(\omega\omega')=(i\extd+\extd i)(\omega\omega')$ using the derivation properties of $i,\extd$, cancel some terms and arrive as $\CL(\omega)\omega'+\omega \CL(\omega')$. The coderivation property follows by arrow-reversal of this calculation but can also be seen directly from the super-coderivation properties of $i,\extd$ (when doing it explicitly one should note that $(-1)^{i(\omega)\o}i(\omega)\o\tens i(\omega)\t=(-1)^{|i(\omega\o)|}i(\omega\o)\tens \omega\t+ (-1)^{|\omega\o|}\omega\o\tens i(\omega\t)=(-1)^{|\omega\o|}(-i(\omega\o)\tens \omega\t+ \omega\o\tens i(\omega\t))$, i.e. there is an extra minus sign from the degree of $i$ but only in one of the terms). In this way one shows $\Delta \CL=(\CL\tens\id+\id \tens \CL)\Delta$. Clearly $\CL\extd=\extd i \extd =\extd \CL$ by $\extd^2=0$, and similarly for $i$. (2) This is immediate if we assume $\extd=[\theta,\ \}$.  \endproof

We similarly define an augmented generalised first order bicovariant differential calculus as $(\Omega^1,\extd, i)$,  the two structures together defining $\CL$ on degree $0,1$. Note also that $i:\Omega^1\to A$ is a bimodule map so one should think of $i$ geometrically as interior product by a preferred vector field (and $\CL$ the Lie derivative along it). In the standard case of $\bar\Omega$ the existence of $\CL$ is equivalent to that of $i$ since given $\CL$ we can define $i$ inductively by $i(\extd\omega)=\CL(\omega)-\extd i(\omega)$. Also in this case $\CL$ if it exists is determined by $\CL$ on degree zero, i.e. by a single derivation $\CL$ of $A$ that extends to the bicovariant exterior algebra.  However, many Hopf algebras in noncommutative geometry have few derivations and even fewer that are also coderivations, so typically we may have $\CL=0$ and hence $i$ zero on $\bar\Omega$. That is why we need the more general setting of generalised differentials for this to be useful.

At the level of first order differential calculi the data for an augmented strongly bicovariant one are $(\Lambda^1,\omega,i)$ where $\omega:A^+\to\Lambda^1$ and $i:\Lambda^1\to A^+$ are morphisms for the appropriate right $A$-crossed module structures on $A^+$. This merely superposes Theorem~\ref{Hopfcalc} and Lemma~\ref{ilambda}.  We now consider the question of when the first order data extend to higher orders.

For the strongly bicovariant `generalised Woronowicz exterior algebra' in Proposition~\ref{genworon},  a given morphism $i:\Lambda^1\to A^+$ does not automatically extend to higher degrees. For example, for degree two one needs $\sum \eta\ra i(\zeta)=0$ for all $\sum \eta\tens\zeta\in \ker(\id-\Psi)$.
For the strongly bicovariant `universal calculus' in Proposition~\ref{univ} we have:

\begin{proposition}\label{univaug}
In the setting of Proposition~\ref{univ}, suppose $i:\Lambda^1\to A^+$ is an $A$-crossed module map as in Lemma~\ref{ilambda}. Then $\Omega_\theta(A)$ is augmented by $i$  if and only if $\theta\ra i(\theta)$ is graded central in $\Omega_\theta(A)$.
\end{proposition}
\proof Graded-central means $[\theta\ra i(\theta),\eta\}=0$ for all $\eta\in\Omega_\theta(A)$. Since $\Omega_\theta(A)$ is a quotient of $A\rbiprod T_-\Lambda^1$ module by some Hopf ideal $J$ of $T_-\Lambda^1.$ It suffices to show that $i(A.J)=0$ is equivalent to stated condition by Proposition~\ref{tensor-codiff}. By construction, the former is equivalent to $i(\theta^2 a)=i(a\theta^2)$ and $i(\theta^2\eta)=i(\eta\theta^2)$ for any $a\in A,\ \eta\in T_-\Lambda^1/J,$ which equivalent to $i(\theta^2)a=ai(\theta^2)$ and $i(\theta^2)\eta=(-1)^{|\eta|}\eta i(\theta^2)$ for any $a\in A$ and $\eta\in T_-\Lambda^1/J$ as $i$ defined as super-derivation. Now, note that $i(\theta)$ is primitive in $A,$ so $i(\theta^2)=i(\theta).\theta-\theta.i(\theta)=i(\theta)\theta-i(\theta)_1\theta\ra i(\theta)_2=i(\theta)\theta-i(\theta)\theta-\theta\ra i(\theta)=-\theta\ra i(\theta).$ We obtain the condition displayed. \endproof

Finally, combining the differential in Proposition~\ref{shuffle-diff} with the codifferential exterior algebra in Proposition~\ref{subshuffle-codiff}, we have
\begin{proposition}\label{subshuffle-aug} In the setting of Proposition~\ref{subshuffle-codiff}, a  right $A$-crossed module map $\omega:A^+\to \Lambda^1$ makes $A\rbiprod B_{\theta^*}(\Lambda^1)$ an augmented strongly bicovariant exterior algebra with $\extd$ inherited from $A\rbiprod \Sh_-(\Lambda^1)$ if and only if
\begin{gather*}
\<\theta^*\tens\theta^*,w_0\tens \tilde{\omega}(w_1)\>w_2=\<\theta^*\tens\theta^*,w_0\tens \tilde{\omega}(w_1)\>,\\
\<\theta^*\tens\theta^*,(v_1)_0\tens \tilde{\omega}((v_1)_1)\> v_2\tens\cdots\tens v_n=(-1)^{n-1}v_1\tens\cdots\tens v_{n-1}\<\theta^*\tens\theta^*,(v_n)_0\tens \tilde{\omega}((v_n)_1)\>
\end{gather*}
for any $w\in \Lambda^1\subseteq B_{\theta^*}(\Lambda^1)$ and $v_1\tens\cdots\tens v_n\in B_{\theta^*}(\Lambda^1).$
\end{proposition}
\proof The map $\delta$ defined in Proposition~\ref{shuffle-diff} such that $\delta(B_{\theta^*}(\Lambda^1))\subseteq B_{\theta^*}(\Lambda^1)$ if and only if the conditions displayed in the statement hold. \endproof

We conclude with an elementary example, extending Proposition~\ref{Ug}. Thus, let $\cg$ be a Lie algebra and $(\ra, \Lambda^1)$ a right representation of $\cg$ regarded as a crossed module with trivial coaction and let $\omega:\cg\to \Lambda^1$ be a 1-cocycle extended to a crossed module morphism $U(\cg)^+\to \Lambda^1$ as in Proposition~\ref{Ug}.

\begin{proposition}\label{Ugclass} The `generalised Woronowicz calculus' extending Proposition~\ref{Ug} is a strongly  bicovariant  differential exterior algebra $\Omega(U(\cg))=U(\cg)\rcross\Lambda(\Lambda^1)$ with the usual  Grassmann algebra $\Lambda$ and $\delta=0$. Augmentations on this correspond to $\cg$-module maps $i: \Lambda^1\to \cg$, where $\cg$ has the right adjoint action, obeying  $v\ra i(v)=0$ for all $v\in \Lambda^1$. The Lie derivative is given by $\CL\xi=i\circ\omega(\xi)$ and $\CL v=\omega\circ i(v)$ on $\xi\in\cg$ and $v\in \Lambda^1$.  The canonical choice is $\Lambda^1=\cg$ with augmentation $i:\cg\hookrightarrow U(\cg)^+$. \end{proposition}
\proof The exterior algebra in the inner case is clear from Proposition~\ref{genworon} as the coaction is trivial, so $\extd=[\theta,\ \}$ for any $\theta\in\Lambda^1$ and $B_-(\Lambda^1)$ is the usual exterior algebra of a vector space, i.e. the Grassmann algebra. However, it is easy to see that $\delta=0$ still meets the requirements of Proposition~\ref{extalg} and Theorem~\ref{theorem-bi} to give a strongly bicovariant differential exterior algebra even if the cocycle is not a coboundary, i.e. if we are not in the inner case.  Then $\Omega$ as super-bosonisation of this is the cross product algebra and tensor product coalgebra. For a coderivation we seek $i:\Lambda^1\to U(\cg)^+$ so that $\Delta i(v)=i(v)\tens 1+1\tens i(v)$ and $i(v\ra\xi)=[i(v),\xi]$ for all $v\in \Lambda^1, \xi\in\cg$. This dictates the form of $i$ as stated. To extend this to $\Lambda^2$ we require $v\ra i(v)=0$ for all $v\in \Lambda^1$, and one can check that $i$ then extends as a super-derivation. This means a coderivation is given as stated and the Lie derivative is from the general theory as in Theorem~\ref{aug}.  In the inner case the Lie derivative stated reduces correctly to $\CL=[i(\theta),\ ]$. \endproof

\subsection{Duality results}

We now give some general results on the duality of augmented strongly bicovariant differential calculi. We start at first order:

\begin{lemma}\label{dualcodif} In the finite-dimensional case, if $(\Lambda^1,i)$ defines a bicovariant codifferential structure on $A$ then $(\Lambda^1{}^*,i^*)$ defines a generalised first order differential structure on $A^*$. The latter being inner via $\theta^*\in \Lambda^1{}^*$ corresponds to $i$ being coinner via $\theta^*$. 
\end{lemma}
\proof This is mainly a matter of checking conventions since the content is clear from the discussion above. Thus if $\Lambda^1$ is a right $A$-crossed module then  the right action of $A$ can be viewed equivalently as a left coaction $A^*$ and this dualises to a right coaction of $A^*$ on $\Lambda^1{}^*$ with $\Delta_R(\phi)=\sum_a \<\phi,(\ )\ra e_a\>\tens f^a$ for all $\phi\in \Lambda^1{}^*$, where $\{e_a\}$ is a basis of $A$ and $\{f^a\}$ a dual basis of $A^*$. Similarly a right coaction of $A$ can be viewed equivalently as a left action of $A^*$ and this dualises to a right action on $\Lambda^1{}^*$ with $\phi\ra f=\<\phi, (\ )_0\> f((\ )_1)$ for all $\phi\in \Lambda^1{}^*$ and $f\in A^*$. It is a nice exercise to check that we obtain $\Lambda^1{}^*$ as right $A^*$-crossed module. Also note that $(A^+)^*=(A^*)^+$ under the splitting $A=k1\oplus A^+$ provided by the counit projection $\pi(a)=a-1\eps(a)$ and even more, $A^+$ as a right $A$-crossed module by the right adjoint action and right coaction $\Delta-1\tens \id$ dualises by the above to $A^*{}^+$ by the right adjoint coaction of $A^*$ and right multiplication of $A^*$. One can check finally that  $i^*: A^*{}^+\to \Lambda^1{}^*$ is a morphism of right $A^*$-crossed modules as a consequence of Lemma~\ref{ilambda}. If this is inner, i.e of the form $i^*(f)=\theta^*\ra f$ for $f\in A^*{}^+$, this is equivalent to $i$ as $i(\eta)=\<\theta^*,\eta_0\>\eta_1-\<\theta^*,\eta\>1_A$ for all $\eta\in \Lambda^1$ by reversing the above dualisation.  \endproof

In the finite dimensional case the notion of a first order bicovariant differential calculus augmented by a codifferential is then self-dual as an extension of Hopf algebra duality: if $(\Omega^1,\extd,i)$ is an augmented first order calculus over $A$ then $(\Omega^{1*},i^*,\extd^*)$ is one over $A^*$.

We now consider higher degrees in the same way. If $\Omega$ is an augmented strongly bicovariant exterior algebra on $\Omega^0=A$ and has finite-dimensional components then the super-Hopf algebra graded dual $(\Omega^*,i^*,\extd^*)$ is an augmented strongly bicovariant exterior algebra on $A^*$. The dualisation is clear from the self-duality of the axioms of an augmented strongly bicovariant exterior algebra, but again may be explicitly verified by constructing the adjoints. Here $A^*$ is an algebra adjoint to the coalgebra structure of $A$ and so on for the exterior super-Hopf algebra. The wedge product of the differential forms on $A$ becomes the coproduct expressing the covariance of $\Omega^*(A^*)$ and so forth. Clearly $\extd^*:(\Omega^n)^*\to (\Omega^{n-1})^*$ is adjoint to $\extd$ and provides the codifferential in $\Omega^*$. For avoidance of doubt we adopt the conventional duality of super-Hopf algebras not the duality of braided-Hopf algebras\cite{Ma:book} which avoids unnecessary braidings.  It is clear from the discussion above that if $\Omega,\Omega^*$ are strongly bicovariant differential exterior algebras on $A,A^*$ respectively then they are necessarily both augmented by mutual dualisation of $\extd$.

Next, we don't want to confine duality to the case of finite-dimensional components. More generally, let $H,A$ be dually paired Hopf algebras. We say similarly that an object $V$ in the category of $A$-crossed modules is {\em mutually adjoint} to $V'$ in the category of $H$-crossed modules if $V',V$ are dually paired as vector spaces so that the action on one is adjoint to the coaction of the other in the sense
\[\< \phi\ra h,v\>=\<\phi,v_0\>\<v_1,h\>,\quad \<\phi, v\ra a\>=\<\phi_0,v\>\<a,\phi_1\>\quad\forall \phi\in V',\ v\in V,\ h\in H,\ a\in A.\]
As an example, $A^+$ as a right $A$-crossed module by right multiplication is mutually adjoint to $H^+$ as a right $H$-crossed module by the adjoint action.  Then we can say that a first order generalised bicovariant differential calculus $(\Omega^1(A),\extd)$ is dually paired to a bicovariant first order codifferential calculus $(\Omega^1(H),i)$ if the associated $\Lambda^1$ in the two cases (let us denote the latter one by $\Lambda^1{}^*$) are mutually adjoint as crossed modules and $\omega:A^+\to \Lambda^1$ defining $\extd$ is adjoint to $i: \Lambda^1{}^*\to H^+$.  Similarly when both structures are present for a dual pairing of augmented strongly bicovariant first order differential calculi. The same applies to all orders, with mutual duality being expressed in terms of the maps on the associated left invariant braided Hopf algebras $\Lambda,\Lambda^*$. We are typically interested in $A,H$ infinite-dimensional but $\Lambda,\Lambda^*$ with finite-dimensional graded components.

\begin{proposition}\label{dualext} Let $A,H$ be non-degenerately dually paired Hopf algebras with bijective antipode, $\Omega^1(A)=A\rbiprod \Lambda^1$ a generalised first order bicovariant differential calculus dually paired to a first order bicovariant codifferential structure $\Omega^1(H)=H\rbiprod\Lambda^1{}^*$ and suppose that $\Lambda^1$ extends to a graded braided-super Hopf algebra $\Lambda$ with finite-dimensional components in the category of $A$-crossed modules and is mutually adjoint to $\Lambda^*$ as an $H$-crossed module.  In this setting $\extd$ extends to $\Omega(A)=A\rbiprod \Lambda$ in the setting of Theorem~\ref{theorem-bi} iff $i$ extends to $\Omega(H)=H\rbiprod \Lambda^*$ as a strongly bicovariant codifferential structure.  \end{proposition}
\proof Here the assumption is that the action and coaction of $A$ extend to a braided-super Hopf algebra $\Lambda$ and that this is mutually adjoint to $\Lambda^*$  in the category of $H$-crossed modules. Also by assumption $\omega:A^+\to \Lambda^1$ is mutually adjoint to $i:\Lambda^*{}^1\to H^+$. According to Theorem~\ref{theorem-bi} the only additional data we need is $\delta:\Lambda\to \Lambda$ with various properties, equivalent to $\extd(a\tens \eta)=a\o\tens\omega(\pi a\t)\eta+a\tens\delta \eta$ for all $a\in A$, $\eta\in \Lambda$ providing a strongly bicovariant differential calculus. However, this data dualises to $\mathfrak i:\Lambda^*\to \Lambda^*$ of degree -1 obeying a dual set of properties such that
\[ i(\phi)= i_1(\phi^1)\phi^2+ {\mathfrak i}(\phi)\]
gives a bicovariant codifferential structure. Here $\und\Delta=(\ )^1\tens(\ )^2$ is the braided-super coproduct on $\Lambda^*$,  $i_1:\Lambda^*{}^1\to H^+$ is the given $i$ adjoint to $\omega$ and zero in other degrees, and $\mathfrak i$ is zero on degree 1. This is sufficient to complete the proof but for completeness we note these dual properties. Thus, the dual of the super-coderivation property of $\delta$ in Theorem~\ref{theorem-bi} is
\begin{equation}\label{ideriv}{\mathfrak i}(\phi\psi)={\mathfrak i}(\phi)\psi+(-1)^{|\phi|} \phi\, {\mathfrak i}(\psi) +(-1)^{|\phi|}(\phi\ra i_1(\psi^1))\psi^2,\quad\forall \phi,\psi\in \Lambda^*\end{equation}
and corresponds to $i$ a super-derivation when extended as a bimodule map to $\Omega$. Similarly the super-derivation and nilpotence  properties of $\delta$ dualise to
\[ \und\Delta {\mathfrak i}={(\mathfrak i}\tens \id + (-1)^{|\ |}\tens{\mathfrak i})\und\Delta, \quad {\mathfrak i}^2=0,\quad i_1{\mathfrak i}(\phi)+i_1(\phi^1)i_1(\phi^2)=0,\quad\forall \phi\in \Lambda^2{}^*\]
and correspond to $i$ is a super-coderivation and $i^2=0$. Finally, the property under the aciton of $A$ in Proposition~\ref{extalg} dualises to
\[ ({\mathfrak i}\tens\id)\Delta_R\phi-\Delta_R{\mathfrak i}\phi=\phi^2_0\tens i_1(\phi^1)\phi^2_1-(-1)^{|\phi^1|}\phi^1_0\tens\phi^1_1 i_1(\phi^2)\]
for all $\phi\in\Lambda^*$.  Now, if $\extd$ on degree 1 extends to $\Omega(A)$ we have the $\delta$ data well defined on $\Lambda$ and is compatible with $\omega$, and hence equivalently the data $\mathfrak i$ extends to $\Lambda^*$ and is compatible with $i_1$. Nondegeneracy of the pairing ensures that if an adjoint map exists it is uniquely determined and this is needed to deduce the axioms involving $H$ form the dual axioms involving $A$ and vice-versa.  \endproof

Clearly the roles of $A,H$ here are symmetric so a codifferential on $\Omega(A)$ corresponds in this context to a differential on $\Omega(H)$ and hence an augmented strongly bicovariant calculus on one side corresponds to the same on the other side. Also, an inner differential calculus via $\theta\in\Lambda^1$ corresponds  on the dual side to
\[ {\mathfrak i}(\phi)=\<\phi^1,\theta\>\phi^2-(-1)^{|\phi^1|}\phi^1\<\phi^2,\theta\>,\quad\forall \phi\in \Lambda^*.\]
We now apply these duality ideas to our general constructions, starting with the generalised Woronowicz one in Proposition~\ref{genworon}.

\begin{lemma}\label{wordual} Let $A,H$ be nondegenerately dually paired Hopf algebras with bijective antipode and $\Lambda^1$ a finite-dimensional $A$-crossed module mutually adjoint to $\Lambda^1{}^*$ an $H$-crossed module. Then $A\rbiprod B_\pm(\Lambda^1)$ and $H\rbiprod B_\pm(\Lambda^1{}^*)$ are nondegenerately dually paired as (super-)Hopf algebras.
\end{lemma}
\proof The `super' applies in the $B_-$ case, which is the case we need but the proof for both versions is the same. Here $T_\pm\Lambda^1$ is a braided(-super) Hopf algebra in the category of right $A$-crossed modules and $T_\pm\Lambda^1{}^*$ a braided-(super) Hopf algebra in the category of right $H$-crossed modules. The latter has coproduct given by braided-binomial matrices using braided-integers $[n,\pm\Psi^*]$ where $\Psi^*:\Lambda^1{}^*\tens\Lambda^1{}^*\to \Lambda^1{}^*\tens\Lambda^1{}^*$ is the braiding on $\Lambda^1{}^*$. We first check that the latter is indeed adjoint to $\Psi$ the braiding on $\Lambda^1$ in its category. Thus $\<\Psi^*(\phi\tens\psi),\eta\tens \zeta\>=\<\psi_0\tens \phi\ra \psi_1,\eta\tens \zeta\>=\<\psi_0\tens \phi,\eta\tens \zeta_0\>\<\psi_1,\zeta_1\>=\<\psi\tens\phi,\eta\ra \zeta_1\tens \zeta_0\>=\<\phi\tens\psi,\Psi(\eta\tens \zeta)\>$ where $\<\ ,\ \>$ is the evaluation pairing on degree 1. Next, we define the pairing between $T_\pm\Lambda^1{}^*$ and $T_\pm\Lambda^1$ to be cf.\cite{Ma:fre,Ma:book,Ma:dcalc}
\[ \<\<\phi_1\tens\phi_2\tens\cdots\tens\phi_n, \eta_1\tens \eta_2\tens \cdots\tens \eta_m\>\>=\delta_{n,m}\<\phi_1\tens\phi_2\tens\cdots\tens\phi_n, [n,\pm\Psi]!(\eta_1\tens \eta_2\tens \cdots\tens \eta_m)\>\]
where the right hand side is the usual evaluation pairing of tensor products of dual spaces. One can also show that
\[ [n,\Psi]!^*=[n,\Psi]^*(\id\tens [n-1,\Psi]!^*)=(\id\tens [n-1,\Psi^*]!)[n,\Psi^*]=[n,\Psi^*]!\]
where
\[ [n,\Psi]^*=\id+ \Psi^*_1+\Psi^*_2\Psi^*_1+\cdots+\Psi^*_{n-1}\cdots\Psi^*_1.\]
One way to see this is that $[n,\Psi]!$ can be written as a sum over $S_n$ with reduced expressions replaced by $\Psi_i$. Reversing the order of compositions of $\Psi_i$ then corresponds to inversion in $S_n$ and hence gives the same sum after a change of variables. Similarly $[n,-\Psi]!^*=[n,-\Psi^*]!$. This means that the pairing $\<\<\ ,\ \>\>$ can also be written with $[n,\pm\Psi^*]!$ on the $\phi_i$. As a result $\<\<\ ,\ \>\>$ descends to a pairing of $B_\pm(\Lambda^1{}^*)$ with $B_\pm(\Lambda^1)$. As these are defined exactly by setting to zero the kernel of the pairing on each side, they are now nondegenerately paired. Finally, the tensor product $A$-crossed module structure on each degree of $T_\pm\Lambda^1$ descends to $B_\pm(\Lambda^1)$ and the tensor product $H$-crossed module structure on each degree of $T_\pm\Lambda^1{}^*$ descends to $B_\pm(\Lambda^1{}^*)$ and these two remain adjoint with respect to $\<\<\ ,\ \>\>$ because $[n,\pm\Psi]!$ is a morphism of $A$-crossed modules (or similarly on the dual side). It follows that $A\rbiprod B_\pm(\Lambda^1)$ and $H\rbiprod B_\pm(\Lambda^1{}^*)$ are dually paired as (super-)Hopf algebras by $\<h\tens \phi,a\tens \eta\>=\<h,a\>\<\<\phi,\eta\>\>$ for all $a\in A$, $h\in H$, $\eta\in B_\pm(\Lambda^1)$ and $\phi\in B_\pm(\Lambda^1{}^*)$ since the semidirect product and coproduct structures are by mutually adjoint actions and coactions respectively. This can easily be verified from the definitions of the cross product and coproduct, using the conventions in \cite[Ch. 1]{Ma:book}. \endproof

This applies to Proposition~\ref{genworon} with dual given by Lemma~\ref{wordual} so that we have a `coinner' codifferential structure on $\Omega(H)=H\rbiprod B_-(\Lambda^1{}^*)$ using the pairing $\<\<\ ,\ \>\>$ for the duality. Note also that if $\Lambda$ is generated by $\Lambda^1$ then clearly $\delta$ is uniquely determined by its degree 1 part $\delta_1$ as a super-derivation. Similarly if $\Lambda^*$ is generated by $\Lambda^*{}^1$ then $\mathfrak i$ is uniquely determined  by $i_1$ via (\ref{ideriv}) and hence in this case all of $\delta$ is uniquely determined. Again, this applies to $B_-(\Lambda^1)$ by Lemma~\ref{wordual}.

\begin{corollary}\label{deltawor} In the setting of Lemma~\ref{wordual}, if a generalised first order bicovariant differential calculus $(\Lambda^1,\omega)$ extends to a strongly bicovariant differential exterior algebra on $\Omega(A)=A\rbiprod B_-(\Lambda^1)$ then it does so uniquely with $\delta$  given by
\[ \delta \eta=-\cdot [2,-\Psi]^{-1}(\eta_0\tens\omega(\pi \eta_1)),\quad\forall \eta\in \Lambda^1\]
in degree 1. Here $[2,-\Psi]$ is not assumed to be invertible but the choice of inverse element does not change the result in $\Lambda^2$.
\end{corollary}
\proof Whenever $\Lambda^1{}^*$ generates, $\delta$ is determined inductively by
\[ \<\<\delta(w),\phi\psi\>\>=\<\<\delta w^1,\phi\>\>\<w^2,\psi\>+ (-1)^{|w|}\<w_0\tens\omega(\pi w_1),\phi\tens\psi\>\]
for all $w\in \Lambda^{n},\ \phi\in \Lambda^*{}^n,\ \psi\in\Lambda^*{}^1$. This is obtained either by extending $\mathfrak i$ by (\ref{ideriv}) as  explained above or directly from the super-coderivation property in Theorem~\ref{theorem-bi}. In the case of $\Lambda=B_-(\Lambda^1)$ the formula in Lemma~\ref{wordual} gives the result for $\delta$ on degree 1. One can similarly write down a formula for general degree. \endproof

Note that this entails that $\eta_0\tens\omega(\pi \eta_1)\in {\rm Image}[2,-\Psi]$ for all $\eta\in\Lambda^1$. For example, if $(\Lambda^1,\theta)$ defines an inner differential calculus and $\Psi(\eta\tens\theta)=\theta\tens \eta$ then $-\eta_0\tens \omega(\pi \eta_1)=-\eta_0\tens\theta\ra \eta_1+\eta\tens\theta)=\eta\tens\theta-\Psi(\theta\tens \eta)=[2,-\Psi](\theta\tens \eta+\eta\tens \theta)$. This gives $\delta \eta=\{\theta,\eta\}$ as expected.

We conclude with the universal and shuffle constructions of Section~3.3.

\begin{lemma}\label{univpair}
Let $A, H$ be non-degenerately dually paired Hopf algebras with bijective antipode and $\Lambda^1$ a finite-dimensional $A$-crossed module mutually adjoint to ${\Lambda^1}^*$ an $H$-crossed module with $\theta\in\Lambda^1$ right invariant, then $A\rbiprod \Lambda_\theta(\Lambda)$ and $H\rbiprod B_{\theta}({\Lambda^1}^*)$ are non-degenerately dually paired as super-Hopf algebra.
\end{lemma}
\proof First, $\Sh_\pm({\Lambda^1}^*)$ is the graded dual of braided tensor (super-)Hopf algebra $T_\pm\Lambda^1$ under the pairing
$\<v_1\tens\cdots\tens v_n,\phi_1\tens\cdots\tens \phi_m\>=\delta_{n,m}\<v_1,\phi_1\>\cdots\<v_n,\phi_n\>.$
By construction,
$\<(v_1\tens\cdots\tens v_r)\bullet_\pm(v_{r+1}\tens\cdots\tens v_n),(\phi_1\tens\cdots\phi_n)\>
=\<\left( {n\atop r}; \pm\Psi\right)(v_1\tens\cdots\tens v_r)\underline{\tens}(v_{r+1}\tens\cdots\tens v_n),(\phi_1\tens\cdots\phi_n)\>
=\<(v_1\tens\cdots\tens v_r)\underline{\tens}(v_{r+1}\tens\cdots\tens v_n),\gamma_{r,n-r}\circ\left[ {n\atop r}; \pm\Psi^*\right](\phi_1\tens\cdots\phi_n)\>$ for any $0\le r\le n,$ where we use $\left[ {n\atop r}; \pm\Psi^*\right]=\left({n\atop r}; \pm\Psi\right)^*.$ Thus
$\<(v_1\tens\cdots\tens v_r)\bullet_\pm(v_{r+1}\tens\cdots\tens v_n),(\phi_1\tens\cdots\phi_n)\>=\<(v_1\tens\cdots\tens v_r)\underline{\tens}(v_{r+1}\tens\cdots\tens v_n),{\Delta}(\phi_1\tens\cdots\phi_n)\>$.The rest of proving they are dually paired is obviously.

Then one can check $B_{\theta}({\Lambda^1}^*)\subseteq \Sh_-({\Lambda^1}^*),$ by construction, is exactly the annihilator of $J_{\theta}\subseteq T_-\Lambda^1,$ where $J_\theta=\{\xi({\theta}^2\ra a-\epsilon_A(a){\theta}^2)\eta)\,, [{\theta}^2,\xi]\eta\ |\,a\in A,\,\xi,\eta\in T_-\Lambda^1\}$ is the ideal that $\Lambda_\theta(\Lambda^1)$ quotient out. Thus $\Lambda_{\theta}(\Lambda^1)$ and $B_{\theta}({\Lambda^1}^*)$ are non-degenerately dually paired under the same pairing. \endproof

\begin{corollary}\label{univdual}
In the setting of Lemma~\ref{univpair},  $A\rbiprod \Lambda_\theta(\Lambda)$ in Proposition~\ref{univaug} is augmented if and only if $H\rbiprod B_{\theta}({\Lambda^1}^*)$ in Proposition~\ref{subshuffle-aug} is augmented. If so, they are mutually dual with the differential on one side dual to the codifferential on the other side.
\end{corollary}
\proof The proof is dualising formula from one side to the other.
One can show the condition for one being augmented is dual to one for another. Then inner differential calculus $d=[\theta,\ \}$ of $A\rbiprod \Lambda_\theta(\Lambda)$ is dual to the coinner codifferential calculus of $H\rbiprod B_{\theta}({\Lambda^1}^*)$ given in Proposition~\ref{subshuffle-codiff}, Also, the codifferential calculus of $A\rbiprod \Lambda_\theta(\Lambda)$ is dual to the differential calculus of $H\rbiprod B_{\theta}({\Lambda^1}^*)$ by comparing the formula (\ref{univ-codiff}) and the formula in Proposition~\ref{shuffle-diff}.
\endproof

As an elementary application,  let $k[G]$ be a linear algebraic group in the form of commutative Hopf algebra and let $\Lambda^1$ be a finite-dimensional right $k[G]$-comodule with coaction $\Delta_R$. Equivalently, $\Lambda^{1*}$ is a left $k[G]$-comodule with coaction $\Delta_L$. We view 1-cocycles $\zeta\in Z^1(G,\Lambda^{1*})$ algebraically as $\zeta\in k[G]\tens \Lambda^{1*}$ obeying
\[ (\Delta\tens\id)\zeta = 1\tens\zeta+(\id\tens\Delta_L)\zeta.\]
By applying $\id\tens\eps\tens\id$ one sees that in this case $\zeta\in k[G]^+\tens \Lambda^{1*}$. We can equally view $\zeta:\Lambda^{1}\to k[G]^+$. The coboundary case is $\zeta=\Delta_L\theta^*-1\tens\theta^*$ for some $\theta\in \Lambda^{1*}$. We define $\cg^*:=k[G]^+/(k[G]^+)^2$ as the classical dual Lie algebra and $\pi_{\cg^*}$ the associated projection from $k[G]^+$. The right adjoint coaction descends to $\Ad_R:\cg^*\to\cg^*\tens k[G]$. We view $\Lambda^1$ as a $k[G]$-crossed module with the trivial action given by the counit of $k[G]$.

\begin{proposition}\label{Gclass} For any 1-cocycle $\zeta\in Z^1(G,\Lambda^{1*})$,  $\Omega=k[G]\rcocross \Lambda(\Lambda^1)$ is a strongly bicovariant codifferential algebra over $k[G]$ with $i=\zeta$ viewed as a map $\Lambda^1\to k[G]^+$ and extended to $\Omega$ as a $k[G]$-module map. Moreover, any intertwiner $\iota:\cg^*\to \Lambda^1$  makes $(\Omega,\extd,i)$ an augmented strongly bicovariant exterior algebra on $k[G]$ by
\[ \extd a= a\o \iota(\pi_{\cg^*}\pi(a\t)),\quad \delta \eta=-{1\over 2}\eta_0\wedge\iota(\pi_{\cg^*}\pi(\eta_1))\]
provided $(\id\tens\iota\pi_{\cg^*}\pi)\Delta_R$ is antisymmetric in its output. In this case $\CL$ on $k[G]$ is the derivation $\CL(a)=a\o \zeta(\pi_{\cg^*}\pi(a\t))$ for all $a\in k[G]$. The classical calculus is given by $\Lambda^1=\cg^*$ and $\iota=\id$.
\end{proposition}
\proof The codifferential structure  at first order is immediate from Lemma~\ref{ilambda} applied to $A=k[G]$. Adjoint action on $k[G]^+$ is trivial and so is the action on $\Lambda^1$ and the other condition for $i:\Lambda^1\to k[G]^+$ is clearly equivalent to the cocycle condition for $\zeta$. This codifferential structure extends to $k[G]\rcocross T_-\Lambda^1$ by Proposition~\ref{tensor-codiff} and is easily seen to descend to the usual Grassmann algebra $\Lambda(\Lambda^1)=B_-(\Lambda^1)$ in our case as the braiding $\Psi$ is the trivial flip on $\Lambda^1\tens\Lambda^1$. For the differential structure at first order we need $\omega:k[G]^+\to \Lambda^1$ that intertwines $\Ad_R$ with $\Delta_R$ and which vanishes on $(k[G]^+)^2$ (this is so as to also intertwine the actions, the action on $\Lambda^1$ being trivial). This means that it factors through a map $\iota$ as stated via $\pi_{\cg^*}$. The exterior derivative and $\CL$ are then immediate from the general constructions. The extension to $k[G]\rcocross \Lambda(\Lambda^1)$ requires $\delta$ from Corollary~\ref{deltawor} and one can check that this works using Theorem~\ref{theorem-bi}, where the displayed condition reduces to  the stated antisymmetry condition involving $\iota$. \endproof

Although not necessary for the above, we can also illustrate the duality. For this, suppose that $G$ has a finite-dimensional Lie algebra $\cg$ dual to $\cg^*$ with $k[G],U(\cg)$ dually paired in a compatible way and that $V$ is a finite-dimensional right $\cg$-module (regarded with trivial coaction as a $U(\cg)$-crossed module) mutually dual to $V^*$ as a right $k[G]$-crossed module with trivial action. Thus we suppose that $\Delta_R \phi=(\ )\la \phi\in V^*\tens k[G]$ (here the Lie algebra acts from the left on $\phi\in V^*$ and we assume that this extends to the group in algebraic form).
We also require a Lie algebra 1-cocycle extended to $\zeta^*:U(\cg)^+\to V$ (say) mutually dual to a map $V^*\to k[G]^+$ which we can view as an algebraic cocycle $\zeta\in k[G]^+\tens V$. Finally, we require an intertwiner $\iota^*:V\to \cg$ (say) obeying $v\ra\iota^*(v)=0$ for all $v\in V$. Then on the one hand setting $\Lambda^1=V$ we have $U(\cg)\rcross\Lambda(V)$ an augmented strongly bicovariant differential algebra on $U(\cg)$ by Proposition~\ref{Ugclass} using the data provided.  The mutually dual data on the other hand gives $k[G]\rcocross \Lambda(V^*)$ as an augmented strongly bicovariant exterior algebra on $k[G]$ with
\[ \extd f =\sum_i(\tilde{e_i}f) \iota(f^i),\quad \delta\phi=-{1\over 2} \sum_i  e_i\la\phi\wedge \iota(f^i),\quad i(\phi)=(\id\tens\phi)(\zeta),\quad \forall f\in k[G],\ \phi\in V^*\]
where $\la$ is the left action of $\cg$ on $V^*$, $\{e_i\}$ is a basis of the Lie algebra, $\{f^i\}$ is a dual basis and $ \tilde e_i$ denotes the associated left invariant vector field on $k[G]$ defined by $\tilde e_i(f)=f\o\<e_i,f\t\>$ via the pairing. The Lie derivative on functions works out as $\CL(f)=\sum_i\tilde{e}_i(f)(\id\tens\iota(f^i))(\zeta)$. This gives the same results as Proposition~\ref{Gclass} taken with $\Lambda^1=V^*$ but in a less algebraic and more conventional geometric form.  In the coinner case $\CL$ on functions is the vector field $\widetilde{\iota^*(\theta^*)}^R-\widetilde{\iota^*(\theta^*)}$, where  $\tilde{\ }^R$ similarly denotes the construction of a right-invariant vector field.

\section{Generalised calculi on q-deformation quantum groups}

We now give quantisations of the elementary $U(\cg)$ and $k[G]$ examples, starting with a semiclassicalisation of the quantum group case to come later. To keep things simple we give these examples over $\C$ and using the geometric notations following Proposition~\ref{Gclass}. 

\begin{proposition}\label{Ugalmost} Let $\cg$ be a complex semisimple Lie algebra and $c$ the quadratic Casimir. We take $i:\cg\oplus \C c\hookrightarrow U(\cg)^+$ with the right adjoint action and coaction $\Delta_R=\Delta-1\tens\id$. Any element $\theta^*\in \cg$ defines an augmented inner strongly bicovariant calculus $\Omega(U(\cg))=U(\cg)\rbiprod B_-^{quad}(\cg\oplus\C c)$ with
\[ v\xi-\xi v=[v,\xi],\quad \extd \xi=\theta^*\ra\xi,\quad\delta v=0,\quad \delta c=[\theta^*,t_i]t^i,\quad \forall v\in\cg\subset  \cg\oplus\C c,\ \xi\in\cg\subset U(\cg).\]
 The Lie derivative is  $\CL=[ \theta^*,\ ]$.  \end{proposition}
\proof  The crossed module structure is the canonical one on $U(\cg)^+$ which restricts
 so that
\[ v\ra\xi=[v,\xi],\quad c\ra\xi=0,\quad \Delta_R \xi=\xi\tens 1,\quad \Delta_R c=c\tens 1+ 2 t,\quad \forall  \xi\in\cg,\quad v\in\cg\subset\cg\oplus\C c\]
where $t\in \cg\tens\cg$ is the split Casimir. The crossed module algebra $B_-^{quad}(\Lambda^1)$ here is generated by $\cg,c$ with relations at the quadratic level
\[ v\wedge w+w\wedge v=0,\quad c\wedge v+v\wedge c+\sum t^i\wedge [v,t_i]=0,\quad c\wedge c=0,\quad\forall v,w\in \cg\subset\cg\oplus\C c.\]
This comes from the crossed module braiding, for which the only nontrivial one is $\Psi(v\tens c)=c\tens v+ 2 t^i\tens [v,t_i]$ for $v\in \cg$. Here $B_-^{quad}(\cg\oplus\C c)$ has the structure of a super cross product of the form $(\C[c]/\<c^2\>)\rcross \Lambda(\cg)$ with cross relations as shown (one can think of this action as a certain super-vector field acting on $\Lambda(\cg)$ as a Grassmann algebra).  Next, in order to apply Proposition~\ref{genworon} (in the quadratic version) we need $\Psi(v\tens \theta^*)=\theta^*\tens v$ which requires $\theta^*\in\cg$  by the above. As the coaction is trivial on $\cg$ we have an inner strongly bicovariant differential exterior algebra $\Omega=U(\cg)\rbiprod B^{quad}_-(\cg\oplus\C c)$ by the quadratic version of Proposition~\ref{genworon}.     We then find $\delta=[\theta^*,\ \}$ as stated. Note that in $\Lambda^3(\cg)$ one has
\[ \sum_{i,j}[v,t_i]t'_j[t^i,t'{}^j]=0,\quad\forall v\in \cg\subset\cg\oplus\C c \]
where $t'$ is another copy of $t$, $[\ ,\ ]$ is the Lie bracket. This observation follows from the Jacobi identity and ad-invariance of $t$, and is needed if one wants directly to verify $\delta^2c=0$. Next, we turn to the augmentation. Clearly $i:\cg\oplus\C c\hookrightarrow U(\cg)^+$ is a morphism by construction and we check that it extends to products as a super-derivation to all of $\Omega$. As explained before Proposition~\ref{univaug} we check $c\ra i(v)+v\ra i(c)+t_i\ra i([v,t^i])= 0+v\ra c-\sum_i[v,t_i]\ra i(t^i)=v\ra c-\sum_i[[v,t_i],t^i]=0$ for the non-standard relation in degree 2, the other relations being equally clear. \endproof

Dually, we take $\cg^*\oplus \C\theta'$, say, where $\theta'$ is dual to $c$. As before we assume $\C[G]$ dually paired with $U(\cg)$ and a right coadjoint coaction $\Delta_R(\phi)={\rm Ad}^*_{(\ )}(\phi)\in\C[G]$ on $\phi\in \cg^*$ corresponding algebraically to the left coadjoint action of $G$ on $\cg^*$. As a $\C[G]$-crossed module we have
\[ \phi\ra f=\phi f(1)+ 2 (\widetilde{t(\phi)}f)(1)\theta',\quad \theta'\ra f=\theta' f(1),\quad \Delta_R\theta'=\theta'\tens 1,\quad\forall \phi\in\cg^*.\]
We regard the split Casimir or inverse Killing form here as a map $t:\cg^*\to\cg$. As in Proposition~\ref{Gclass} we denote by $\bar\extd$ the classical exterior derivative on $G$, $\{e_i\}$ a basis of $\cg$ and $\{f^i\}$ a dual basis. In principle there may be relations of  higher degree in $B_-(\cg\oplus\C c)$ beyond quadratic but these are typically hard to compute. Similarly we now work with $B_-^{quad}(\cg^*\oplus\C\theta')$ which is dually paired but may not be nondegenerately paired.

\begin{corollary}\label{Galmost} Dual to Proposition~\ref{Ugalmost}, let $\C[G]$ be a linear algebraic group with semisimple Lie algebra $\cg$ and $\Lambda^1=\cg^*\oplus\C\theta'$ as above, and $\theta^*\in\cg$. We have an augmented `coinner' strongly bicovariant calculus $\Omega(\C[G])=\C[G]\rbiprod \Lambda(\cg^*\oplus\C\theta')$ with
\[ [\phi,f]=2\theta' \widetilde{t(\phi)}(f),\quad [\theta',f]=0,\quad\forall \phi\in\cg^*,\ f\in \C[G]\]
where  $\widetilde{\ }$ denotes extension as a left-invariant vector field. The exterior derivative is
\[ \extd f=\bar\extd f+(\square f)\theta',\quad \delta \phi=-{1\over 2}\sum_i {\rm ad}^*_{e_i}(\phi)\wedge f^i,\quad \delta\theta'=0\]
where $\square$ is the Laplace-Beltrami operator given by the action of $c$ and ${\rm ad}^*$ is the left coadjoint action of $\cg$ on $\cg^*$. The augmentation and  Lie derivative have the same form as in the coinner case of Proposition~\ref{Gclass}. \end{corollary}
\proof  The dual construction to the inclusion in Proposition~\ref{Ugalmost} is a generalised first order differential calculus with, in particular, $\omega(f)(\xi)=f(\xi)$ and $\omega(f)(c)=f(c)$ for all $f\in \C[G]^+$ and $\xi\in\cg$. As a result one has $\extd f=\bar\extd f+(\square f)\theta'$ where $\square$ is the Laplacian given by the left action of $c$ on $\C[G]$. For the relations of $B^{quad}_-(\cg^*\oplus\C\theta')$, the only non-trivial braiding is $\Psi(\phi\tens\psi)=\psi\tens\phi+2{\rm ad}^*_{t(\phi)}(\psi)\tens\theta'$ but ${\rm ad}^*_{t(\phi)}(\psi)=-{\rm ad}^*_{t(\psi)}(\phi)$ by ad-invariance of $t$, so that the quadratic relations retain the classical form $B_-^{quad}(\cg^*\tens\C\theta')=\Lambda(\cg^*\oplus\C\theta')$.  Next, $\phi_0\tens\omega(\pi\phi_1)={\rm ad}^*_{e_i}(\phi)\tens f^i+{\rm ad}^*_{c}(\phi)\tens\theta'$ by a similar computation as in Proposition~\ref{Gclass}, as $\phi_0\tens \omega(\pi\phi_1)(\eta)=\phi_0\tens \phi_1(\eta)={\rm ad}^*_{\eta}(\phi)$ for all $\eta\in \cg\oplus \C c$. Now $[2,-\Psi]( {\rm ad}^*_{e_i}(\phi)\tens f^i)={\rm ad}^*_{e_i}(\phi)\tens f^i-f^i\tens{\rm ad}^*_{e_i}(\phi)-2{\rm ad}^*_{t({\rm ad}^*_{e_i}(\phi))}(f^i)\tens\theta'=2{\rm ad}^*_{e_i}(\phi)\tens f^i+2{\rm ad}^*_{t(f_i)e_i}(\phi) \tens\theta'=2({\rm ad}^*_{e_i}(\phi)\tens f^i +{\rm ad}^*_{c}(\phi)\tens\theta')$ where we used antisymmetry of ${\rm ad}^*_{e_i}(\phi)\tens f^i$ and of ${\rm ad}^*_{t(\ )}(\ )$ as noted above. We also have $\theta'_0\tens\omega(\pi\theta'_1)=0$. Hence Corollary~\ref{deltawor} gives $\delta$ as stated. This then extends as before to $\Lambda(\cg^*\oplus\C\theta')$.

Next, dualising the differential structure on $U(\cg)$, we have a codifferential $i(\theta')=0$ and $i(\phi)=\<{\rm Ad}^*_{(\ )}(\phi),\theta^*\>-\<\phi,\theta^*\>$ where we use the left coadjoint action of $G$ on $\cg^*$ and we regard the right hand side as an element of $\C[G]$ assuming all operations are algebraic. The Lie derivative is $ \CL(f)=i(\extd f)=i(\bar\extd f+(\square f)\theta')=(\widetilde{e_i}f)i(f^i)=(\widetilde{e_i}f)\<{\rm Ad}^*_{(\ )}(f^i)-f^i,\theta^*\>$ if $\{e_i\}$ is a basis of $\cg$ and $\{f^i\}$ a dual basis. This is the same $\CL_X$ as the coinner case of Proposition~\ref{Gclass} with $i=\id$ there and using the more geometric formulation.  We verify directly that $i$ as stated extends as a superderivation. Thus, $\phi\ra i(\phi)=2\theta'\<t(\phi),i(\phi)\>=2\theta'\<{\rm ad}^*_{t(\phi)}(\phi),\theta^*\>=0$ in view of the antisymmetry of ${\rm ad}^*_{t(\ )}(\ )$.  \endproof

The calculus $\Omega(\C[G])$ here is of the `almost commutative' class of calculi on Riemannian manifolds studied recently in \cite[Sec. 3]{Ma:bh} and we see how it arises from a codifferential structure on $U(\cg)$. We see moreover that it is augmented by a vector field expressed as a Lie derivative for any choice of $\cg\in\theta^*$ and indeed arises as the dual of a calculus $\Omega(U(\cg))$.

We now turn to the quantum group case, at least at first order.  Recall that if $H$ is a finite-dimensional quasitriangular Hopf algebra then its dual $A$ can be viewed as equipped with a coquasitriangular structure $R:A\tens A\to k$. We also have a quantum Killing form $Q(a,b)=R(b\o, a\o)R(a\t,b\t)$ for $a,b\in A$. Moreover, the quasitriangular structure can be viewed as maps $R_1,R_2:A\to H$, where $R_1(a)=R(a,\  )$ and  $R_2(a)=R(\ ,a)$, and similarly for $Q_1, Q_2$ \cite[Ch 2]{Ma:book}. Following Drinfeld one says that the Hopf algebra is factorizable if $Q$ as a map is an isomorphism. The notion of coquasitriangular Hopf algebras works well when $A$ is infinite dimensional while for $H$ one typically has to work over formal power series $\C[[\hbar]]$ in a deformation parameter. It is known that standard bicovariant calculi on $A$ in the factorisable case are classified by  irreducible representations of $H$\cite{Ma:cla, BauSch, Ma:dcalc}.   We now show that in fact the construction gives more  naturally a generalised differential calculus without assuming factorizability.

Note that in the form of maps the axioms for $R$ make sense for a mutually dual pair of Hopf algebras $A,H$ (a quasitriangular dual pair) without assuming finite-dimensionality but assuming that the pairing is nondegenerate. Viz, $R_1:A\to H$ is an algebra and anticoalgebra homomorphism, $R_2:A\to H$ is an coalgebra and antialgebra hom, both are convolution-invertible and\cite{Ma:dbos}
\[ R_2(a\o)h\o\<h\t,a\t\>=\<h\o,a\o\>h\t R_2(a\t)\]
\[  h\o \<h\t,a\t\>R_1(a\t)=R_1(a\o)\<h\o,a\t\> h\t\]
\[ \<R_1(a),b\>=\<a,R_2(b)\>,\quad \forall h\in H,\ a,b\in A.\]
Likewise we define $Q_i:A\to H$ by
\[ Q_1(a)=R_2(a\o)R_1(a\t),\quad Q_2(a)=R_1(a\o)R_2(a\t),\quad\forall a\in A\]
These axioms imply that $A$ is coquasitriangular but are a little stronger. In this context $A$ with
\begin{equation*}\label{Hcrossfull} a\ra h=a\t\<a\o Sa\th,h\>,\quad \Delta_R(a)=a\t\tens R_1(a\o)R_2(a\th),\quad \forall h\in H,\ a\in A\end{equation*}
becomes a right $H$-crossed module as one may check using the axioms for the $R_i$. We will see momentarily that $Q_2$ in fact becomes a crossed module morphism where $H$ has the right adjoint action/coproduct crossed module structure cf (\ref{Acocrossed}). These and similar facts also hold if  $R_1$ in the crossed module structures and in $Q_i$ comes from a second independent quasitriangular structure and accordingly the results that follow can be made slightly more general. Of relevance to us is the projection of this crossed module structure to $A^+$,
\begin{equation}\label{Hcross} a\ra h=a\t\<a\o Sa\th,h\>,\quad \Delta_R(a)=a\t\tens R_1(a\o)R_2(a\th)-1\tens Q_2(a),\quad \forall h\in H,\ a\in A^+\end{equation}

On the dual side we suppose that
the right coadjoint coaction of $A$ on $H$ is well-defined. Here $H\subseteq A^*$ via the nondegenerate pairing and  $\Delta_Rh\in A^*\tens A$ is defined by $\<\Delta_R(h),a\>=\<h,a\t\>a\o Sa\th$, so the issue is whether this lies in $H\tens A$. If so we say that the dual pair is {\em regular} and in this case $H$ becomes a right $A$-crossed module mutually adjoint to (\ref{Hcross}). Some explicit  formulae are
\begin{equation}\label{Across} \Delta_R(h)=\sum_a e_a\o h S e_a\t\tens f^a,\quad  h\ra a=  R_2(a\o) h R_1(a\t),\quad\forall a\in A,\ h\in H\end{equation}
where $\{e_a\}$ is a basis of $H$ and $\{f^a\}$ is a dual basis in the finite-dimensional case.  In the infinite-dimensional case the coaction is by our regularity assumption.  We will see momentarily that $Q_1$ becomes a crossed module morphism where $A$ has the right adjoint coaction/right multiplication crossed module structure (\ref{Acrossed}).

\begin{proposition}\label{quasi} Let  $(H,A)$ be a regular quasitriangular dual pair and $I\subseteq H^+$ a 2-sided ideal. We take  $I^\perp \subseteq A^+$ with the $H$-crossed module structure from (\ref{Hcross}). Then augmentation by $Q_2:I^\perp \to H^+$  makes $\Omega^1(H)=H\rbiprod I^\perp$ a first order bicovariant codifferential calculus.  Any element $\theta^*\in I^\perp$ obeying
\[ \theta^*\t\tens R_1(\theta^*\o)R_2(\theta^*\th)=\theta^*\tens 1+1\tens Q_2(\theta^*)\]
provides an augmented inner first order bicovariant calculus with Lie derivative $\CL=[Q_2(\theta^*),\ ]$. Here $Q_2(\theta^*)$ is primitive in $H^+$.
  \end{proposition}
\proof  We first show that $Q_2:A\to H$ is a crossed module morphism as claimed. That it intertwines the action in (\ref{Hcross}) with the right adjoint action as a variant of  \cite[Prop~2.1.14]{Ma:book}. This works the same way in the quasitriangular dual pair setting, so we omit details. The new feature is
\begin{eqnarray*} \Delta Q_2(a)&=&\Delta(R_1(a\o)R_2(a\t))=(R_1(a\t)\tens R_1(a\o))(R_2(a\th)\tens R_2(a\four))\\
&=& R_1(a\t)R_2(a\th)\tens R_1(a\o)R_2(a\four)=Q_2(a\t)\tens R_1(a\o)R_2(a\th)\end{eqnarray*}
so $Q_2$ intertwines the coaction on $A$ with the coproduct of $H$, which together with the right adjoint action is the $H$-crossed module structure on $H$. We then have also that
$Q_2:A^+\to H^+$ is a crossed module morphism for the $H$-crossed module structures (\ref{Hcross}) on $A^+$ and (\ref{Acocrossed}) on $H^+$.  The crossed module structure on $A^+$ in (\ref{Across}) clearly restricts to $I^\perp$ hence $i=Q_2: I^\perp\to H^+$ is a morphism of crossed modules and provides an augmentation on  $\Omega^1(H)$ according to Lemma~\ref{ilambda}.

The condition for $\theta^*$ merely explicates that $\Delta_R\theta^*=\theta^*\tens 1$. Since $Q_2:I^\perp\to H^+$ is a morphism of crossed modules, this implies that $Q_2(\theta^*)$ is primitive.  The Lie derivative has the stated form by Theorem~\ref{aug} but can also be computed. \endproof

So on $H$ we have a natural codifferential structure, which means that on the dual we naturally have a generalised differential structure on $A$. In some cases we may have a differential structure on $H$ as well. The nicest case is if $Q_2:I^\perp\hookrightarrow H^+$ then finding $\theta^*$ is equivalent to finding a primitive element in its image.

Note next that $(H^+)^*\supseteq A^+$ etc due to the nondegenerate pairing and that $(H^+/I)^*\supseteq I^\perp$. What this means is that technically we rework the theory in dual form rather than literally apply Lemma~\ref{dualcodif}.

\begin{corollary}\label{coquasi} Dual to Proposition~\ref{quasi}, let $(H,A)$ be a regular quasitriangular dual pair. Then $H^+/I$ is a crossed module from (\ref{Across}) and $Q_1:A^+\to H^+/I$ gives a generalised first order bicovariant differential calculus $\Omega^1(A)=A\rbiprod (H^+/I)$.  This is inner if we have a central element $\theta\in H^+, \theta\notin I$ such that $(\theta-1)H^+\subseteq I$.  If  $\theta^*\in I^\perp$ obeys the condition in Propsition~\ref{quasi}  then
\[ i: H^+/I\to A^+,\quad i(\eta)=\<\eta,\theta^*\t\>\theta^*\o S\theta^*\th -\<\eta,\theta^*\>,\quad \forall \eta\in H^+/I\]
makes $\Omega^1(A)$ augmented with Lie derivative $\CL(a)=Q(a\o,\theta^*)a\t- a\o Q(a\t,\theta^*)$ for all $a\in A$. \end{corollary}
\proof  Here $Q_1:A\to H$ intertwines the right adjoint coaction on $A$ with the coaction in (\ref{Across}) as essentially proven in \cite[Prop~2.1.14]{Ma:book}. The new part is
\begin{eqnarray*} Q_1(ab)&=&R_2(a\o b\o)R_1(a\t b\t)=R_2(b\o)R_2(a\t)R_1(a\t)R_1(b\t)\\
&=& R_2(b\o)Q_1(a)R_1(b\t)=Q_1(a)\ra b,\quad \forall a,b\in A.\end{eqnarray*}
So $Q_1$ is also a morphism of crossed modules. Clearly, the crossed module structure on $H^+$ in (\ref{Hcross}) descends to $H^+/I$ and   $\omega=Q_1:A^+\to H^+/I$ being a morphism gives us a generalised bicovariant differential calculus on $A$ by Theorem~\ref{Hopfcalc}.   To be inner we look for $\theta\in H^+/I$ so that  $\omega(a)=\theta\ra a$,  i.e. so that $Q_1(a)-R_2(a\o)\theta R_1(a\t)\in I$  for all $a\in A^+$. The simplest way to do this, which also immediately implies that $\Delta_R\theta=\theta\tens 1$ is to assume that $\theta$ is central and that $(1-\theta)H^+\subseteq I$.

For the augmentation we obtain the formula by dualising Proposition~\ref{quasi} but after that one can verify directly that $i:H^+/I\to A^+$ is a morphism of crossed modules. Indeed, $i(\eta)=\<\eta_0,\theta^*\>\eta_1-\<\eta,\theta*\>$ in terms of the coadjoint coaction on $\Lambda^1$ and $\<\eta\ra a,\theta^*\>=\<R_2(a\o),\theta^*\o\>\<R_1(a\t),\theta^*\th\>\<\eta,\theta^*\t\>=\<R_1(\theta^*\o)R_2(\theta^*\th),a\>\<\eta,\theta^*\t\>$ $= \<\eta,\theta^*\>\eps(a)$ by assumption on $\theta^*$ (this expresses that the action on $H^+/I$ is adjoint to the coaction on $\theta^*$ and the latter is trivial). With these facts, we check $i(\eta\ra a)=\<(\eta\ra a)_0,\theta^*\>(\eta\ra a)_1-\<\eta\ra a,\theta^*\>=\<\eta_0\ra a\o,\theta^*\>Sa\o\eta_1 a\th-\<\eta,\theta^*\>\eps(a)=Sa\o(\<\eta,\theta^*\>\eta_1-\<\eta,\theta^*\>)a\t=Sa\o i(\eta)a\t$ as required. Also $i(\eta_0)\tens\eta_1=\<\eta_{00},\theta^*\>\eta_{01}\tens\eta_1-\<\eta_0,\theta^*\>1\tens\eta_1=\<\eta_0,\theta^*\>\Delta\eta_1-\<\eta_0,\theta^*\>1\tens\eta_1=\Delta i(\eta)-1\tens i(\eta)$ as $\Delta_R\eta=\eta_0\tens\eta_1$ is a right coaction. Hence $i$ provides an augmentation.  One can compute the associated Lie derivative on $\Omega^1(A)$ as
\begin{eqnarray*}\CL(a)&=&i(a\o\omega(\pi a\t))=a\o \<(Q_1(\pi a\t))_0,\theta^*\>(Q_1(\pi a\t))_1-a\o\<Q_1(\pi a\t),\theta^*\>\\
&=& a\o\<Q_1(\pi a\t\t),\theta^*\>Sa\t\o a\t\th-a\o\<Q_1(\pi a\t),\theta^*\>\\ &=&\<Q_1(\pi a\o),\theta^*\>a\t-a\o \<Q_1(\pi a\t),\theta^*\>.\end{eqnarray*} We used the definition of $i$ in terms of the coadjoint coaction on $H^+/I$, the morphism property of $Q_1$, and that the counit projection $\pi$ commutes with the right adjoint coaction on $A$. We cancel $\pi$ between the two terms in the result to obtain the answer stated.  \endproof

Clearly, we have a standard calculus on $A$  precisely when $Q_1:A^+\twoheadrightarrow H^+/I$ which, due to the nondegenerate pairing, implies $Q_2:I^\perp \hookrightarrow H^+$ is injective so this also holds in the standard calculus case.  The Corollary~\ref{coquasi} applies certainly when $H$ finite-dimensional with dual $A$, for example reduced quantum groups at $q$ a root of unity. If $e$ is a central idempotent in the block decomposition of $H$, we let $I=(1-e)H^+$ so that $H^+/I \isom eH^+$. This recovers the formulae in \cite{Ma:dcalc} in this case. Again, the new feature is that we do not need factorizability, provided we work with generalised differential calculi, and this applies for example to reduced quantum groups at even roots of unity. Since $e^2=e$, the counit $\eps(e)=0,1$ and in the 0 case $\theta=e\in H^+$ is central and makes the calculus inner. We do not in general have the primitive element $Q_2(\theta^*)$ but an analogue of the exponential of $\CL$ namely $S^4$ see below.

For the usual quantum groups $A=C_q(G)$, $H=U_q(\cg)$ one can take $I=\ker\rho$ where $\rho:H^+\to \End(V)$ is the restriction to $H^+$ of an irreducible representation then $\Lambda^1\isom{\rm im}\rho\isom \End(V)$ for generic $q$. This then reproduces the standard bicovariant differential calculus on $C_q(G)$ in the known classification for generic $q$ with classical limit\cite{Ma:cla}. The calculus is inner with $\theta\in U_q(\cg)^+$ any central element that can be normalised so that $\rho(\theta)=1$. Then $(\theta-1)U_q(\cg)^+\subseteq \ker\rho$ as required. For example we can take here the $q$-deformed quadratic Casimir or in the formal deformation-theoretic setting where $q=e^{ h\over 2}$ and we work over $\C[[h]]$, we can take $\theta\propto 1-\nu$, where $\nu$ is the ribbon element.  Also in this setting  there is a canonical element $g\in U_q(\cg)$ built from the quasitriangular structure  which is group-like and implements $S^4$ by conjugation\cite{Dri, Ma:book}. In Drinfeld's formal power series setting we have a logarithm $D$ so that $g=e^{ {h \over 2}D}$. We take its inverse image under $Q_2$ as $\theta^*$. The Lie derivative in Proposition~\ref{quasi} is then $[D,\ ]$ and hence its exponential $e^{{h\over 2}\CL}=S^4$ on $U_q(\cg)$, where $S$ is the antipode. Clearly one can also take for $D$ any direction in the Cartan subalgebra of $U_q(\cg)$ as these generators are all primitive. Dually, $C_q(G)$ in a formal setting has $i$ defined by $D$ and Lie derivative $\CL(a)=\<D,a\o\>a\t- a\o \<D,a\t\>$ for all $a\in C_q(G)$. Although these constructions are formal, one can think of them as reducing to Corollary~\ref{Galmost} in leading nontrivial order.

We have focussed on $\Omega^1(H)$ and $\Omega^1(A)$. In the inner cases we see that respectively $\theta^*,\theta$ are right invariant so Proposition~\ref{genworon} holds and both $\Omega(H)=H\rbiprod B_-(I^\perp)$ and $\Omega(A)=A\rbiprod B_-(H^+/I)$ are strongly bicovariant exterior algebras and by construction dually paired. The augmentations and their extensions to the higher degrees will be looked at elsewhere.

\section{Generalised calculi on finite groups and Hopf quivers}
Here we apply previous general theory to the Hopf algebra $k(G)$, the function algebra on a (finite) group $G$, and the group algebra $kG$. This makes a link with Hopf quivers and also allows us to explore the duality in this very concrete case. We denote by $\mathcal{C}$ the set of all the conjugacy classes of $G.$ If $V$ is a left $G$-module we denote by ${}_GV$ the space of invariant elements. Similarly $V_G$ in the right module case.

\subsection{Generalised differentials on group function algebras and Hopf quivers}

Firstly, we specialise Theorem~\ref{Hopfcalc}, Corollary~\ref{Hopfcorol} and Lemma~\ref{ilambda} to $A=k(G).$

\begin{proposition}\label{w_c}
Let $A=k(G)$ on a finite group, the generalised bicovariant differential calculus data $(\Lambda^1,\omega)$ in Theorem~\ref{Hopfcalc} are equivalent to the following data:
\begin{itemize}
\item[1)] $\Lambda^1$ a $G$-graded space i.e. $\Lambda^1=\oplus_{g\in G}\Lambda^1_g;$
\item[2)] $\Lambda^1$ also a left $G$-module s.t. $h\la\Lambda^1_g=\Lambda^1_{hgh^{-1}},\,\forall g,h\in G;$
\item[3)]  a set of pairs $\{(c, \omega_c)\ |\ c\in C,\  \omega_c\in{}_{Z_c}\!\Lambda^1_c\}_{C\in\CC,C\neq\{e\}}.$
\end{itemize}
Here $Z_c$ is the centralizer of  $c$ in $G$. Bicovariant codifferential data $(\Lambda^1,i)$ in Lemma~\ref{ilambda} are given by 1),2) and
\begin{itemize}
\item[4)] $\{\iota_g\in \Lambda^1{}^*_e\}_{g\in G}$ a cocycle in the sense $\iota_{gh}=\iota_g\ra h+\iota_h$ for all $g,h\in G$
\end{itemize}
where $\Lambda^1{}^*$ is canonically a right $G$-module.  When both exist, the Lie derivative is $\CL=0$.
\end{proposition}
\proof It is well known that a vector space is a right $k(G)$-crossed module if and only if it is a left $kG$-crossed module. So the right $k(G)$-crossed module $\Lambda^1$ is equivalent to the data 1) and 2). Note here $\Lambda^1_g:=\Lambda^1\ra\,\delta_g$ and $\Delta_R(v)=\sum_{h\in G}h\la v\tens \delta_h.$

Also, the right $k(G)$-module map $\omega:k(G)^+\to \Lambda^1$ is  uniquely defined by $\{\omega_g=\omega(\delta_g)\in\Lambda^1_g|\,g\in G\setminus\{e\}\}$ since the right-module structure corresponds to the grading. We also need $h\la \omega_g=\omega_{hgh^{-1}}$ for all $h\in G$ for $\omega$ to be a right comodule map where $k(G)^+$ has the right adjoint coaction. This is equivalent to the data stated in 3).

Indeed, given the collection $\{\omega_g\in \Lambda^1_g\}_{g\ne e}$ such that $h\la \omega_g=\omega_{hgh^{-1}}$ for all $h\in G$, clearly $\omega_g\in {}_{Z_g}\Lambda^1_g$. We can choose an element $c\in C$ and its associated $\omega_c$ for each nontrivial $C\in \CC$. Conversely, suppose we are given the data 3) consisting of $c\in C$ and  $\omega_c$  for each $C$. For any $g\in G\setminus\{e\}$ write $g=hch^{-1}$ for some  $C$ and its chosen $c\in C$ and some $h\in G.$ One can set $\omega_g=h\la\omega_c$ which in $\Lambda^1_g,$. This is well-defined because  if also $g=h'ch'^{-1},$ then $h'=hu$ for some $u\in Z_c,$ so $h'\la \omega_c=h\la(u\la \omega_c)=h\la \omega_c.$

For the codifferential structure we let
\[ i(\eta)=\sum_{g\in G}\delta_g\<\iota_g,\eta\>\]
 as the most general linear map $\Lambda^1\to k(G)^+$, where $\iota_g\in \Lambda^1{}^*$ and $\iota_e=0$. That this is a module map means $i(\eta_h)=\delta_{h,e}i(\eta)$ for all $h\in G$ from which we deduce that $\iota_g\in \Lambda^1{}^*_e$. That $i$ is a comodule map means $i(h\la \eta)=i(\eta)(\  h)-1i(\eta)(h)$ for all $h\in G$, which means $\sum_{g\in G}\delta_g\<\iota_g\ra h,\eta\>=\sum_{g\in G}\delta_{gh^{-1}}\<\iota_g,\eta\>-\<\iota_h,\eta\>$ which after a change of variables is the condition stated. The cocycle condition stated entails that $\iota_e=0$. The Lie derivative is $\CL(\delta_g)=\sum_{h\ne e}(\delta_{gh^{-1}}-\delta_g)i(\omega_h)=0$ as $\omega_h\in \Lambda^1_h$ and $i$ has support only on $\Lambda^1_e$.  \endproof

Note that the data in (3) are equivalent if they define the same map $\omega$, meaning that for  each $C$ we have $(c,\omega_c)\sim (c',\omega'_{c'})$ where the equivalence is
\begin{equation}\label{comega}(c,\omega_c)\sim (c',\omega'_{c'}),\quad {\rm iff}\quad c'=kck^{-1},\quad \omega'_{c'}=k\la\omega_c\end{equation}
for some $k\in G$.

\begin{proposition}\label{clagroup}
Let $A=k(G)$, the data $(\Lambda^1,\theta)$ for an inner generalised bicovariant calculus in Lemma~\ref{propinner} amounts to $\Lambda^1$ in Proposition~\ref{w_c} and $|\CC|$ elements
$$\{\theta_e\in\Lambda^1_e\}\cup\{(c,\theta_c)\ |\ c\in C, \theta_c\in{}_{Z_c}\!\Lambda^1_c\}_{C\in\CC,C\neq\{e\}}.$$
Up to isomorphism, the data is $\{(c,\theta_c)\}$ modulo equivalence as in (\ref{comega}). The calculus in Proposition~\ref{w_c} is always inner, namely one can take $(c,\theta_c)=(c,\omega_c)$ for all $C\ne\{e\}$. If $|G|$ is invertible a codifferential structure is coinner with $\iota_g=\theta^*\ra(g-1)$ for some $\theta^*\in \Lambda^1{}^*_e$ and all $g\in G\setminus\{e\}$. \end{proposition}
\proof
 Here, because $A=k(G)$ is commutative, we have $\Lambda^1\square A=\Lambda^1_A\tens A$. The elements of $\Lambda^1_A$ are $\eta=\sum_h\eta_h$ such that $\eta_h\ra\delta_g=\epsilon(\delta_g)\eta_h=\delta_{g,e}\eta_h$ for all $h\in G,$ which implies $\Lambda^1_A=\Lambda^1_e$.  So we need $\Delta_R\theta-\theta\tens 1\in \Lambda^1_e\tens k(G)$ as the condition in Lemma~\ref{propinner} for an inner bicovariant calculus. This means  $\theta=\sum_{g\in G}\theta_g\in\Lambda^1$, where $\theta_g=\theta\ra \delta_g\in\Lambda^1_g$ such that $h\la\theta_g=\theta_{hgh^{-1}}$ for any $h\in G$ and $g\in G\setminus\{e\}.$ As in the proof of Proposition~\ref{w_c}, such $\{\theta_g\}_{g\in G\atop g\neq e}$ are uniquely determined by a set of pairs $\{(c,\theta_c)\}$ as stated, and we also have a free choice of  $\theta_e$.  From Proposition~\ref{propinner-iso}, up to isomorphism we need only the $\{(c,\theta_c)\}$ part of the data and up to the stated equivalence whereby they define the same $\theta$.  Given a bicovariant calculus in Proposition~\ref{w_c}, one can take  $\theta=\sum_{g\in G\setminus\{e\}}\omega_g$ or in terms of the data, $\theta_e=0,$ and $\theta_c=\omega_c$ for one $c$ in each nontrivial conjugacy class. For the last part, we easily check that this is a cocycle for any $\theta^*$. The resulting $i$ is $i(\eta)=\sum_{g\in G}\delta_g\<\theta^*, (g-1)\la \eta_e\>$ where $\eta_e$ is the component of $\eta$ in $\Lambda^1_e$. Given the cocycle $\{\iota_g\}$ we define $\theta^*=-|G|^{-1}\sum_{h\in G}\iota_h$.  Then $\theta^*\ra(g-1)=-|G|^{-1}\sum_{h\in G} (\iota_{gh}-\iota_h-\iota_g)=\iota_g$ using the cocycle condition. \endproof

Here $\theta^*$ just corresponds to an inner calculus on $kG$ and that only its class in $\Lambda^1{}^*_e/(\Lambda^1{}^*_e)_G$ is relevant, while for the inner calculus on $k(G)$  the component $\theta_e$ is irrelevant.

\begin{corollary}\label{clagroupOmega}
Let $A=k(G)$ and $(\Lambda^1,\theta)$ define an inner generalised bicovariant differential calculus as in Proposition~\ref{clagroup}. This extends to an inner strongly bicovariant differential exterior algebra $(\Omega(G),\extd)$ in the setting of Proposition~\ref{genworon} iff $\Delta_R\theta=\theta\tens 1$, i.e. $\theta_e\in{}_G\Lambda^1_e$. In this case $\Omega(G)=k(G)\rbiprod B_-(\Lambda^1)$ is generated by $k(G),\Lambda^1$ with relations, coproduct and exterior derivative
\[ v\delta_g=\delta_{g|v|^{-1}}v, \quad \Delta\delta_g=\sum_{h\in G}\delta_{gh^{-1}}\tens\delta_h,\quad \Delta v=1\tens v+\sum_{h\in G}h\la v\tens \delta_h,\quad \extd=[\theta,\ \}\]
 for all homogeneous $v\in\Lambda^1$ of $G$-degree $|v|$ and all $g\in G$. \end{corollary}
\proof If $\theta_e\neq 0,$ the condition $\Psi(\eta\otimes\theta)=\theta\otimes\eta$ for all $\eta$ means $\sum_{h\in G}(h\la\theta_e)\otimes(\eta\ra\delta_h)=\theta_e\otimes\eta$ for all $\eta\in\Lambda^1.$ This means $h\la\theta_e=\theta_e$ for $h$ where $\Lambda^1_h\neq0.$ The condition $\{\Delta_R\theta-\theta\otimes1,\Delta_R(\eta)\}=0$ implies $\{\Delta_R(\theta_e)-\theta_e\otimes1,\Delta_R(\eta)\}=0$ for all $\eta\in\Lambda^1.$ Choose $\eta=\theta_e,$ we have $2(h\la\theta_e)^2=\theta_e(h\la\theta_e)+(h\la\theta_e)\theta_e.$ Since $h\la\theta_e\in\Lambda^1_e,$ we can extend $\theta_e$ to a basis of $\Lambda^1_e$ to prove that $h\la\theta_e=\theta_e$ for all $h\in G,$
 which means $\Delta_R(\theta_e)=\theta_e\otimes1$ and thus $\Delta_R\theta=\theta\otimes1.$ The rest of $(\Omega(g),\extd)$ is an elaboration of the general construction of Proposition~\ref{genworon} in our case. \endproof

\begin{corollary}\label{kofGunivaug}
Let $A=k(G)$ and $(\Lambda^1,\theta,\theta^*)$ define an augmented first order inner and coinner bicovariant calculus as in Proposition~\ref{clagroup}, with $\theta\in{}_G\Lambda^1_e, \theta^*\in{\Lambda^1}^*_e$. The inner strongly bicovariant differential exterior algebra $\Omega_\theta(k(G))$ in Proposition~\ref{univ} 
is augmented with codifferential $i$ given by (\ref{univ-codiff}) by extending the first order.
\end{corollary}
\begin{proof}
From $\theta_e\in {}_G\Lambda^1_e,$ corresponding to $\Delta_R\theta=\theta\tens 1$, we know $(g-1)\la \theta_e=0.$ This means $i(\theta)=\sum_{g\in G}\delta_g\<\theta^*,(g-1)\la\theta_e\>=0,$ so $\theta\ra i(\theta)=0.$ Hence Proposition~\ref{univaug} applies.
\end{proof}

Next we classify the isomorphism classes of generalised bicovariant differential calculi on $k(G)$ by Hopf quivers. We already know from Corollary~\ref{calcset} or Corollary~\ref{canonicalform} that calculi on $k(G)$ are given by digraph-quiver pairs $\bar Q\subseteq Q$. We elaborate the bicovariant case where $Q=G$ a group and explicitly identify with the data in Proposition~\ref{w_c} and Proposition~\ref{clagroup}.

Recall that  a digraph $\bar Q$ is a {\rm Cayley digraph} if it is of  the form $\bar Q(G,\bar C)$ where $\bar Q_0=G$ is a group, $\bar C\subseteq G\setminus\{e\}$ is an ad-stable subset  and the digraph has an arrow $x\to y$ iff $x^{-1}y\in\bar C$.  The set of arrows of a Cayley digraph has canonical and mutually commuting left and right action $h\ast(x\to y)=(xh^{-1}\to y h^{-1})$ and  $(x\to y)\ast h=h^{-1}x\to h^{-1}y$ for all $h\in G$. This underlies the standard bicovariant calculus on $k(G)$ in the finite group case.

Similarly, we say that a quiver is  a {\em coloured Hopf quiver} if it is of the form $Q(G,R)$ where $Q_0=G$ is a group, $R$ (the `ramification datum') is an assignment of a natural number $R_C\in \N_0$ to each conjugacy class $C$, and the quiver has precisely $R_C$ arrows $x\to y$ if $x^{-1}y\in C$ and if these arrows are labelled by index $i=1,\cdots, R_C$.
In this case we have a canonical right action  $(x \xrightarrow[]{(i)}y)\ast h=h^{-1}x\xrightarrow[]{(i)} h^{-1}y$. We also have a canonical and commuting left action defined similarly but we don't want to be limited to it. We say that a digraph-quiver pair is $\bar Q\subseteq Q$ is {\em coloured} if the above applies and the arrows of $\bar Q$ are all one colour. Without loss of generality we shall assume that this colour is 1. Clearly, a coloured Hopf quiver in the case where $R_C\in\{0,1\}$ and $R_{\{e\}}=0$ is the same thing as a Cayley digraph and our convention is compatible with that.

Finally, we define a {\rm Hopf digraph-quiver triple} $(\bar Q\subseteq Q,*)$ to be the above data together with a left action $\ast$ of $G$ on $kQ_1$ such that
\begin{enumerate}
\item $h\ast {}^xkQ_1{}^y={}^{xh^{-1}}kQ_1{}^{yh^{-1}}$ for all $h,x,y\in G$.
\item $\ast$ restricts on $\bar Q_1$ to the canonical left action.
\item $\ast$ commutes with the canonical right action on $kQ_1$.
\end{enumerate}

Clearly we are making these definitions so that the following holds.

\begin{proposition}\label{Hopfquivercalc} Let $(\bar{Q}\subseteq Q,\ast)$ be a Hopf digraph-quiver triple on a finite group $G.$ The associated `quiver calculus' in Corollary~\ref{canonicalform} is bicovariant and every generalised bicovariant differential calculus on $k(G)$ is isomorphic one of this form.  Its structure is inner with
\begin{gather*}
\Omega^1=kQ_1,\quad x\xrightarrow[]{(i)} y.f=x\xrightarrow[]{(i)} y.f(y),\quad f.x\xrightarrow[]{(i)} y=f(x)x\xrightarrow[]{(i)} y,\quad \theta=\sum_{a\in \bar{C}}\sum_{x\in G}x\xrightarrow[]{(1)} xa\\
\Delta_L(x\xrightarrow[]{(i)} xg)=\sum_{h\in G}\delta_h\otimes (h^{-1}x\xrightarrow[]{(i)} h^{-1}xg),\quad
\Delta_R(x\xrightarrow[]{(i)} xg)=\sum_{h\in G}h\ast (x\xrightarrow[]{(i)} xg) \tens \delta_h\\
\end{gather*}
\end{proposition}
\proof
To show $(kQ_1,\extd=[\theta,\ ])$ in Corollary~\ref{canonicalform} is bicovariant, it suffices to show that $\Delta_L$ and $\Delta_R$ are $k(G)$-bimodule map and $\extd=[\theta,\ ]$ is $k(G)$-bicomodule map. Consider $\Delta_R(\delta_k.(x\xrightarrow[]{(i)} xg))=\delta_{k,x}\Delta_R(x\xrightarrow[]{(i)} xg)=\delta_{k,x}\sum_{h\in G}h\ast(x\xrightarrow[]{(i)} xg)\otimes\delta_h$, then $\Delta_R$ is left module map iff the last expression equals to $\sum_{h\in G}\delta_{kh^{-1}}.h\ast(x\xrightarrow[]{(i)} xg)\otimes\delta_h,$ i.e. $h\ast(x\xrightarrow[]{(i)} xg)$ is a linear combination of arrows in $Q$ starting from $xh^{-1}.$ Similarly, $\Delta_R$ is right module map iff $h\ast (x\xrightarrow[]{(i)} xg)$ is a linear combination of arrows in $Q$ ending with $xgh^{-1}.$ Therefore, $\Delta_R$ is bimodule map iff $h\ast{}^xkQ_1^{y}\subseteq {}^{xh^{-1}}kQ_1^{yh^{-1}}$, which is the case under our assumptions. Similarly $\Delta_L$ is a bimodule map. Both are coactions as they correspond to actions of $G$.  Next, $\Delta_R(\extd\delta_x)=\Delta_R(\theta\delta_x-\delta_x\theta)=\Delta_R\sum_{a\in\bar C}xa^{-1}\xrightarrow[]{(1)}x-x\xrightarrow[]{(1)}xa=\sum_{a,h}(xa^{-1}h^{-1}\xrightarrow[]{(1)}xh^{-1}-xh^{-1}\xrightarrow[]{(1)}xah^{-1})\tens\delta_h$. On the other hand $(\extd\tens\id)\Delta\delta_x=\sum_{h}(\theta\delta_{xh^{-1}}-\delta_{xh^{-1}}\theta)\tens\delta_h=\sum_{a,h}(xh^{-1}a^{-1}\xrightarrow[]{(1)}xh^{-1}-xh^{-1}\xrightarrow[]{(1)}xh^{-1}a)\tens\delta_h$. The two expressions agree after a change of variables $h^{-1}ah\mapsto a$, hence $\extd$ is a right comodule map.  Similarly for $\Delta_L$. So the associated `quiver calculus' is bicovariant.

Conversely, suppose $\Omega^1$ is a bicovariant calculus. As $G$ is a finite set we know that up to isomorphism $\Omega^1$ is of the `quiver form' associated to some $\bar Q\subseteq Q$ in Corollary~\ref{canonicalform} and hence without loss of generality we assume this. We also know that $\bar\Omega^1\subseteq\Omega^1$ being a standard bicovariant calculus requires $\bar Q$ to be a Cayley digraph. Hence the bimodule structure and $\theta$ have the form shown. Moreover we are given a bicomodule $\Delta_{L,R}$ structure on $kQ_1$ compatible with the bimodule structure and the extending the bicomodule structure of $k\bar Q_1$. Clearly these $\Delta_{L,R}$ are equivalent to commuting left and right actions $*$ of $G$ on $kQ_1$ restricting to the canonical ones on $\bar Q_1$. By the arguments in the preceding paragraph if these actions respect the bimodule structure then $h\ast {}^xkQ_1{}^y\subseteq {}^{xh^{-1}}kQ_1{}^{yh^{-1}}$ for all $h,x,y\in G$ and similarly for $\Delta_L$ we have ${}^xkQ_1{}^y\ast h\subseteq {}^{h^{-1}x}kQ_1{}^{h^{-1}y}$ for all $h,x,y\in G$. As the actions of $h^{-1}$ provide inclusions going the other way, both of these are isomorphisms. From ${}^xkQ_1{}^y\ast h= {}^{h^{-1}x}kQ_1{}^{h^{-1}y}$ we conclude that $\dim({}^xkQ_1^y)=\dim({}^ekQ_1^{x^{-1}y})=R_c$, say, where $c=x^{-1}y$. From ${}^xkQ_1{}^y\ast h= {}^{h^{-1}x}kQ_1{}^{h^{-1}y}$ we conclude that $R_c$ depends only on the conjugacy class of $c$, so we denote it $R_C$ where $C$ is the class of $x^{-1}y$. Finally, we colour the arrows ${}^eQ_1^x$ from $1,\cdots,R_C$ with the arrow from  $\bar Q$ numbered 1.  The vectors $(e\xrightarrow[]{(i)}x^{-1}y)\ast x^{-1}$ provide a basis of ${}^xkQ_1^y$ including $x\xrightarrow[]{(1)}y$ as the right action restricts to the canonical one. There is a linear transformation of ${}^xkQ_1^y$ sending this basis to a basis of arrows which we label correspondingly. These linear maps together constitute a bimodule map on $kQ_1$ that respects $\theta$, i.e. a map of differential calculi. Hence up to isomorphism we can suppose that $\ast$ from the right has the canonical form of a coloured Hopf quiver $Q(G,R)$. \endproof

\begin{corollary} \label{quiverisom} To a Hopf digraph-quiver triple $(\bar Q\subseteq Q,\ast)$ we can associate the data 1),2),3) in Proposition~\ref{clagroup} as follows. Let  $e_g^{(i)}=\sum_{x\in G} x \xrightarrow[]{(i)}xg$ for $i=1,\cdots,R_C$ where $g\in C$ and let $\Lambda^1_g\subset kQ_1$ be spanned by the $\{e_g^{(i)}\}$. Then $\Lambda=\oplus_g\Lambda^1_g$ is a crossed module with action $\ast$ and  $\theta=\sum_{a\in\bar C} e_a^{(1)}$.  Conversely, to the data 1),2),3) in Proposition~\ref{clagroup} we can associate a Hopf digraph-quiver triple.
\end{corollary}
\proof Because the left and right actions $\ast$ commute, we gave $h\ast e_g^{(i)}=\sum_{x\in G} h\ast(x\xrightarrow[]{(i)}xg)=\sum_{x\in G} h\ast((e\xrightarrow[]{(i)}g)\ast x^{-1})=\sum_{x\in G}(h\ast(e\xrightarrow[]{(i)}g))\ast x^{-1}=\sum_{x\in G, j}\lambda_{i,j}(h^{-1}\xrightarrow[]{(j)}gh^{-1})\ast x^{-1}=\sum_{x\in G, j}\lambda_{i,j}xh^{-1}\xrightarrow[]{(j)}xgh^{-1}=\sum_j \lambda_{i,j}e^{(j)}_{hgh^{-1}}$ after a change of variables $xh^{-1}\mapsto x$ in the sum. Here $\lambda_{ij}$ are some coefficients that depend on $h,g$ but not on $x$. Hence the left $\ast$ restricts on $\Lambda^1$ to a left action $\la$ making it a crossed module. We also see from Proposition~\ref{Hopfquivercalc} that $\theta\in \Lambda^1$ and given as stated. Then $h\la \theta=\sum_{a\in\bar C}\sum_{x\in G} h\ast (x\xrightarrow[]{(1)}xa)=\sum_{a\in\bar C}\sum_{x\in G} xh^{-1}\xrightarrow[]{(1)}xah^{-1}=\sum_{a\in\bar C}e^{(1)}_{hah^{-1}}=\theta$ because $\ast$ restricts to the canonical left action on $\bar Q$. This is equivalent to $\Delta_R\theta=\theta\tens 1$.

Going the other way we can by Proposition~\ref{Hopfquivercalc} construct a Hopf digraph-quiver triple associated to any datum $(\Lambda^1,\theta)$ in Proposition~\ref{clagroup}. We take $\bar C$ to be the union of nontrivial conjugacy classes where $\theta_c\ne 0$, which defines the Cayley digraph $\bar Q$. We take $R_C=\dim(\Lambda^1_g)$ for any $g\in C$, which defines the quiver $Q$. We take a basis $\{e^{(i)}_g\}_{i=1,\cdots,R_C}$ for each $\Lambda^1_g$ which we choose so that $\theta_a=e_a^{(1)}$ for all $a\in \bar C$. Finally, we enumerate the arrows of $kQ_1$ and define a bimodule map $\varphi: k(G)\Lambda^1\to kQ_1$ by comparing bases, i.e. so that $\varphi(\delta_x e^{(i)}_g)=x \xrightarrow[]{(i)} xg$ for all $i=1,\cdots,R_C$. This then transfers the given left action $\la$ on $\Lambda^1$ to the required left action $\ast$ for our Hopf digraph-quiver triple. One may then verify all the requirements in detail.  \endproof

The associations are not unique in either direction, but are when both sides are taken up to isomorphism of bicovariant differential calculi. Here an isomorphism  $(\Lambda^1,\{\theta_c\})\isom (\Lambda^1{}',\{\theta'_{c'}\})$ in Proposition~\ref{clagroup} means an isomorphism $\varphi:\Lambda^1\to \Lambda^1{}'$ of $G$-graded $G$-modules such that $\{\varphi(\theta_c)\}\sim \{\theta'_{c'}\}$ in the sense of defining the same element of $\Lambda^1{}'$.  Similarly, two Hopf digraph-quiver triples $(\bar{Q}\subseteq Q,\ast)$ and $(\bar{Q'}\subseteq Q',\ast')$ are isomorphic if the data $R_C$ and $\bar C$ are the same in the two cases and there exists a left $G$-module map $\varphi:kQ_1\to kQ'_1$ such that $\varphi({}^xkQ_1^y)={}^xk{Q'}_1^y$ and $\varphi(x\xrightarrow[]{(1)}xa)=x\xrightarrow[]{(1)}xa$ for any $x\in G,\,a\in\bar{C}.$

Also, we know from these identifications and Proposition~\ref{w_c} that the calculus has an augmentation. In terms of the Hopf digraph-quiver triple it is given by a matrix representation $\lambda: G\to M_{R_{\{e\}}\times R_{\{e\}}}$ where
\[ g\ast(e \xrightarrow[]{(i)}e)=\sum_j\lambda(g)_{ij}(e\xrightarrow[]{(j)}e)\ast g\]
then from the above one finds
\[ i(x\xrightarrow[]{(i)}y)=\delta_{x,y}(\sum_j\lambda_{ij}(x)\theta^*_j-\theta^*_i)\delta_x\]
for any coefficients $\theta^*_i$, $i=1,\cdots,R_{\{e\}}$. Here $\lambda$ represents the action of $G$ on $\Lambda^1_e$.

Finally, we know from these identifications that $\Omega^1$ in Proposition~\ref{Hopfquivercalc} extends to an inner strongly bicovariant differential exterior algebra $\Omega(\bar Q\subseteq Q,\ast)$ via Proposition~\ref{genworon}. Namely we can use the invariant 1-forms $\Lambda^1=\oplus_g\Lambda^1_g$ with bases $\{e^{(i)}_g\}$ and $B_-(\Lambda^1)$ defined by antisymmetization, so in degree 2 we quotient the tensor algebra on $\Lambda^1$ by the kernel of $\id-\Psi$ where $\Psi(e^{(i)}_g\otimes e^{(j)}_h)=e^{(j)}_{ghg^{-1}}\otimes e^{(i)}_g$. The element $\theta$ in Proposition~\ref{Hopfquivercalc} is right-invariant and defines $\extd$ by graded-commutator.

Meanwhile, associated to a usual Hopf quiver $Q(G,R)$ and commuting coactions $\Delta_{L,R}$ making a $k(G)$-Hopf bimodule one has  a super-Hopf algebra which we denote $T_{k(G)}kQ_1$ and defined as follows\cite{Hua}. We take the tensor algebra in the category of $k(G)$-bimodules, so an element is a formal linear combination of paths of the form $x_0\to x_1\to\cdots \to  x_d$ where we take arrows from the quiver,  which we can take to be enumerated to distinguish them. We consider $\delta_x$ as paths of length zero, so we include $k(G)$ itself. The product is the concatenation of paths or evaluation of a function at the endpoints in the case of a product of a path with a function (so the algebra is the path algebra of the quiver). The super-coproduct structure is $\Delta =\Delta_L+\Delta_R$ on $kQ_1$ and $\Delta$ of $k(G)$ on functions. By part of the arguments in the proof of Proposition~\ref{Hopfquivercalc} the coactions correspond to mutually commuting group actions on $kQ_1$ and that up to an isomorphism of Hopf bimodules we can take one of the actions, say the right one corresponding to $\Delta_L$ to be in canonical form with respect to a colouring (or $\Delta_L$ has the form stated in Proposition~\ref{Hopfquivercalc}).  At least when given in this standard form we have the left-invariant forms $\Lambda^1$ spanned by the $\{e^{(i)}_g\}$ as above and clearly the path super-Hopf algebra of \cite{Hua} becomes isomorphic to $k(G)\rbiprod T_-\Lambda^1$ where we take the tensor super-Hopf algebra $T_-\Lambda^1$ in the braided category of $k(G)$-crossed modules.

\begin{corollary}\label{pathQ} (1), Let $(\bar Q\subseteq Q,\ast)$ be a Hopf digraph-quiver triple on a finite group $G$. The path super-Hopf algebra has a quotient $\Omega_\theta$ where we impose the relation that the element $\sum_{x\in G, a,b\in\bar C}x \xrightarrow[]{(1)} xa\xrightarrow[]{(1)}xab$ is central in the path algebra. This is an augmented inner strongly bicovariant exterior algebra over $k(G)$. (2), If $|G|^{-1}$ exists, then any strongly bicovariant calculus on $k(G)$ generated by its degrees 0,1 is isomorphic to a quotient of $\Omega_\theta$ for some Hopf digraph-quiver triple.
\end{corollary}
\proof (1) We construct $\Omega_\theta$ from Proposition~\ref{univ} noting that $\theta$ in Proposition~\ref{Hopfquivercalc} is invariant under $\Delta_R$, the element stated being $\theta^2$. Here the path super-Hopf algebra is isomorphic to $k(G)\rbiprod T_-\Lambda^1$ as explained. (2) Any first order inner calculus given by $( \Lambda^1,\theta)$ can be taken with $\theta_e=0$ so that $\Delta_R\theta=\theta\tens 1$ from Proposition~\ref{clagroup}, and is isomorphic by Corollary~\ref{quiverisom}  to one given by a Hopf digraph-quiver triple $(\bar Q\subseteq Q,\ast)$. Taking the universal super-Hopf algebra in Proposition~\ref{univ} on both sides we obtain isomorphic super-Hopf algebras.  \endproof

Clearly the `minimal' quotient $k(G)\rbiprod B_-(\Lambda^1)$ associated to a Hopf digraph-quiver triple is a quotient of $\Omega_\theta$, but we can also take the more computable $k(G)\rbiprod B_-^{quad}(\Lambda^1)$ as a quotient.

\begin{example}
Let $G=\mathbb{Z}_2=\langle g \rangle,$ with $Q=Q(\Z_2,R)$ given by $R=2\{g\}$ and $\bar Q=e{\to\atop\leftarrow}g$. Consider the Hopf digraph-quiver triple $(\bar{Q}\subset Q,\ast)$ with $\ast$  the canonical left action. Denote the arrows $\alpha_i:e\xrightarrow[]{(i)}g$ and $\beta_i: g\xrightarrow[]{(i)}e,\ i=1,2.$ The path super-Hopf algebra is $kQ^a=k\langle 1,\delta_e,\alpha_i,\beta_i\rangle$ modulo the  the relations
\[ \delta_e^2=\delta_e,\quad \delta_e\alpha_i=\alpha_i,\quad \alpha_i\delta_e=\delta_e\beta_i=0,\quad \beta_i\delta_e=\beta_i,\quad \alpha_i\alpha_j= \beta_i\beta_j=0,\quad\forall i,j,\]
with grading $|\alpha_i|=|\beta_i|=1$ and super-coproduct defined on generators by
\[  \Delta\delta_e=\delta_e\tens\delta_e+\delta_g\tens \delta_g,\quad \Delta\alpha_i=\delta_e\tens\alpha_i+\delta_g\tens\beta_i+\alpha_i\tens\delta_e+\beta_i\tens\delta_g,\]
\[ \Delta\beta_i=\delta_e\tens\beta_i+\delta_g\tens\alpha_i+\beta_i\tens\delta_e+\alpha_i\tens\delta_g\]
where $\delta_g=1-\delta_e$. The counit is $\epsilon(\delta_e)=1,\ \epsilon(\alpha_i)=0,\ \epsilon(\beta_i)=0.$ The left-invariant 1-forms are $\Lambda^1=\Lambda^1_g=k\text{-}\mathrm{span}\{e^{(i)}\}$ where $e^{(i)}=\alpha_i+\beta_i,$ with (co)action given by $e^{(i)}\ra\delta_e=0$ and $\Delta_R(e^{(i)})=e^{(i)}\tens 1$. Then \[ kQ^a\isom k(\Z_2)\rcross T_-\Lambda^1=k(\Z_2)\rcross k\langle e^{(1)},e^{(2)}\rangle\] with cross relations $e^{(i)}\delta_e=\delta_g e^{(i)}$ for all $i$, and the tensor product coalgebra as the coaction is trivial. Hence $\Delta e^{(i)}=e^{(i)}\tens 1+1\tens e^{(i)}$ and $\eps(e^{(i)})=0$.

Next we compute $\Omega_\theta$ from Proposition~\ref{univ}. Here $\theta=e^{(1)}=\alpha_1+\beta_1$ so we have $\theta^2\ra\delta_g=(\theta\ra\delta_e)(\theta\ra\delta_g)+(\theta\ra\delta_g)(\theta\ra\delta_e)=0$ and $[\theta^2,e^{(1)}]=0$, so
\[ \Omega_\theta=k(\Z_2)\rbiprod k\langle e^{(1)},e^{(2)}\rangle/{\langle e^{(1)}e^{(1)}e^{(2)}-e^{(2)}e^{(1)}e^{(1)}\rangle}.\] Equivalently, as $\delta_e(e^{(1)}e^{(1)}e^{(2)}-e^{(2)}e^{(1)}e^{(1)})=\alpha_1\beta_1\alpha_2-\alpha_2\beta_1\alpha_1$ and $\delta_g(e^{(1)}e^{(1)}e^{(2)}-e^{(2)}e^{(1)}e^{(1)})=\beta_1\alpha_1\beta_2-\beta_2\alpha_1\beta_1,$ it follows that $\Omega_\theta$ is $kQ^a$ modulo the additional relations
\[ \alpha_1\beta_1\alpha_2=\alpha_2\beta_1\alpha_1,\quad  \beta_1\alpha_1\beta_2=\beta_2\alpha_1\beta_1.\]
Here $\theta^2=\alpha_1\beta_1+\beta_1\alpha_1$ is central and requiring this is equivalent to imposing these relations as in Corollary~\ref{pathQ}. The new feature not present for the path algebra is the super-derivation  $\extd=[\theta,\ \}$. Thus
\[ \extd \delta_e=\beta_1-\alpha_1,\quad \delta \theta=2\theta^2,\quad \delta e^{(2)}=e^{(1)}e^{(2)}+e^{(2)}e^{(1)}\]
or equivalently to the latter two,
\[ \extd\alpha_1=\beta_1\alpha_1-\alpha_1\beta_1,\quad \extd\alpha_2=\beta_1\alpha_2-\alpha_2\beta_1,\quad
\extd\beta_1=\alpha_1\beta_1-\beta_1\alpha_1,\quad \extd\beta_2=\alpha_1\beta_2-\beta_2\alpha_1\]
 extended as a super-derivation with $\extd^2=0$.

Finally, since the braiding is trivial on $\Lambda^1,$  $B_{-}(\Lambda^1)=B_{-}^{quad}(\Lambda^1)=\Lambda(e^{(1)},e^{(2)})$, the usual Grassmann algebra on generators $\{e^{(i)}\}$ with anticommutative relations and basis $\{1,e^{(1)},e^{(2)},e^{(1)}\wedge e^{(2)}\}.$ Thus the canonical `minimal' calculus as in Proposition~\ref{genworon} is $k(\Z_2)\rcross B_-(\Lambda^1)=k(\Z_2)\rcross\Lambda(e^{(1)},e^{(2)})$ with cross relations as above.  Equivalently, as $\delta_e e^{(i)}=\alpha_i$ and $\delta_g e^{(i)}=\beta_i,$ it follows that we have a quotient of the path algebra version of $\Omega_{{univ}}$ by the further relations \[ \alpha_2\beta_1=-\alpha_1\beta_2,\quad \beta_2\alpha_1=-\beta_1\alpha_2,\quad \alpha_i\beta_i=0,\quad \beta_i\alpha_i=0,\quad i=1,2.\]
Here $\theta^2=\alpha_1\beta_1+\beta_1\alpha_1=0$ as we know from Proposition~\ref{genworon}. In this quotient we see that $\delta e^{(i)}=0$ or equivalently \[ \extd\alpha_1=\extd\beta_1=0,\quad \extd\alpha_2=\extd\beta_2=\beta_1\alpha_2+\alpha_1\beta_2.\]\end{example}

\subsection{Generalised calculi on group algebras}

The group Hopf algebra case $A=kG$ of a group $G$ is dual to $k(G)$ already covered above. Hence the results are clear by dualisation. The only difference is that we do not necessarily assume that $G$ is finite. Note that $\Lambda^1$ in this section is dual to $\Lambda^1$ in Section~6.1.

\begin{proposition}\label{kG} Let $A=k G$. The generalised bicovariant differential calculus data  $(\Lambda^1,\omega)$ in Theorem~\ref{Hopfcalc} amount to the following data.
\begin{itemize}
\item[1)] $\Lambda^1$ a $G$-graded vector space $\Lambda^1=\sum_g \Lambda^1_g$;
\item[2)] $\Lambda^1$ is also a right $G$-module s.t. $\Lambda^1_g\ra h= \Lambda^1_{h^{-1}gh}$, \  $\forall g,h\in G$;
\item[3)] $\{\omega_g\in \Lambda^1_e\}_{g\in G}$ a cocycle in the sense $\omega_{gh}=\omega_h+\omega_g\ra h$,  $\forall g,h\in G$.
\end{itemize}
The data for a bicovariant codifferential in Lemma~\ref{ilambda} is 1),2) and
\begin{itemize}
\item[4)]   a set of pairs $\{(c, \iota_c)\ |\ c\in C,\  \iota_c\in{}_{Z_c}\!\Lambda^1{}^*_c\}_{C\in\CC,C\neq\{e\}}.$
\end{itemize}
\end{proposition}
\proof Parts (1)-(2) are a crossed module. This is a self-dual notion so the data for $\Lambda^1$ is essentially the same as in Proposition~\ref{w_c} but in different conventions. The coaction of $kG$ is given by the grading, $\Delta_R\eta=\eta_0\tens\eta_1$ for $\eta\in\Lambda^1$. For (3)  $(kG)^+$ is spanned by $\{g-1\ | g\in G\setminus\{e\}\}$ and we write $\omega:(kG)^+\to \Lambda^1$ as $\omega(g-1)=\omega_g$. The crossed module coaction on $(kG)^+$ is trivial as the Hopf algebra is cocommutative hence we need each $\omega_g\in \Lambda^1_e$ so as to have trivial coaction. The right module map property $\omega((g-1)h)=(\omega(g-1))\ra h$ for all $h\in G$ is the cocycle condition stated. Note that this entails $\omega_e=0$. For (4) write $i(\eta)=\sum_{h\ne e} i(\eta)_h(h-1)$ say.  Suppose $\eta\in \Lambda^1_g$ then the comodule map property of $i$ requires $i(\eta)\tens g=\sum_{h\ne e} i(\eta)_h (h-1)\tens h$ or $\sum_{h\ne e} i(\eta)_h(h-1)\tens (g-h)=0$ requires $i(\eta)=\<\eta, \iota_g\> (g-1)$ for some $\iota_g\in \Lambda^1{}^*_g$. Hence in general,
\[ i(\eta)=\sum_g\<\eta,\iota_g\>(g-1)\]
as $\iota_g\in \Lambda^1{}^*_g$ picks out the $\Lambda^1_g$ component of $\eta$. The module map property  $i(\eta\ra h)=h^{-1}i(\eta)h$ requires $h\la \iota_g=\iota_{hgh^{-1}}$ for all $h\in G$.  As in the proof of Proposition~\ref{w_c}, this is specified by a subset of values, one for each nontrivial conjugacy class as stated. \endproof

\begin{proposition}\label{kGdua} Let $A=kG$. The data $(\Lambda^1,\theta)$ for an inner generalised bicovariant differential calculus in Lemma~\ref{propinner} amounts to $\Lambda^1$ as in Proposition~\ref{kG} and  $|\CC|$ elements
\[ \{\theta_e\in \Lambda^1_e\}\cup \{(c,\theta_c)\ |\ c\in C, \theta_c\in (\Lambda^1_c)_{Z_c}\}_{C\in\CC, C\ne\{e\}}\]
where $Z_c$ is the centralizer of $c\in C$. Up to isomorphism the data is $[\theta]\in \Lambda^1_e/(\Lambda^1_e)_G$.  The calculus in Proposition~\ref{kG} is  inner if $G$ is finite and $|G|$ is invertible. If $G$ is finite a codifferential structure is coinner with $\iota_g=\theta^*_g$  (the component in $\Lambda^1{}^*_g$) for some $\theta^*\in \Lambda^1{}^*$.
\end{proposition}
\proof The data in Lemma~\ref{propinner} for the inner case is $\theta=\sum_{h\in G} \theta_h$ where $\theta_h\in \Lambda^1_h$ such that $\theta_h=\theta_{ghg^{-1}}\ra g$ for all $g\in G$ and all $h\in G\setminus\{e\}$ for the bicovariance condition. This means that it is given by a set of pairs $\{(c,\theta_c)\}$ as stated and a free value $\theta_e$. However, $\sum_{h\ne e}\theta_h\in\Lambda^1_G$ by a change of variables in the sum, so up to isomorphism we can take $\theta\in\Lambda^1_e$ with further equivalence as stated. As inner data, given a bicovariant calculus, we can take $\theta= -|G|^{-1}\sum_{g\in G} \omega_g$ where we consider $\omega_e=0$.  For the codifferential structure we take $\iota_g$ in the form stated. Given $\{\iota_g\}_{g\ne e}$ we set $\theta^*=\sum_{g\ne e}\iota_g$.   \endproof

The coinner codifferential calculus corresponds to the the inner differential calculus on $k(G)$ in Proposition~\ref{clagroup}. Note that the component $\theta^*_e$ is irrelevant to the codifferential structure, while only the class of  $\theta_e$ in $\Lambda^1_e/(\Lambda^1_e)_G$ is relevant to the differential structure.

\begin{corollary}\label{kGOmega} Let $A=kG$ and $(\Lambda^1,\theta)$ define an inner bicovariant calculus. The bimodule relations and exterior derivative are
\[  \eta. g= g (\eta\ra g),\quad \extd g= g(\theta\ra g-\theta),\quad\forall\eta\in\Lambda^1,\ g\in G.\]
The conditions in Proposition~\ref{genworon} for a differential exterior algebra require $\Delta_R\theta=\theta\tens 1$ if the calculus is connected. The super-Hopf algebra structure of $\Omega(kG)=kG\rbiprod B_-(\Lambda^1)$ and exterior derivative  are
\[ \Delta \eta=1\tens\eta+\eta_0\tens\eta_1,\quad \Delta g=g\tens g,\quad \forall \eta\in \Lambda^1,\ g\in G,\quad \extd=[\theta,(\ )\}.\]
\end{corollary}
\proof The condition $\Psi(\theta\tens\eta)=\eta\tens\theta$ for all $\eta$ means $\sum_g \theta_g\tens (\eta\ra g-\eta)=0$ for all $\eta$. This requires the action of $g$ to be the identity whenever $\theta_g\ne 0$. This is a strong condition and among other things requires $g$ where $\theta_g\ne 0$ to commute with all $h$ where $\Lambda^1_h\ne 0$. It also needs that such $g$ commute with all $\eta$ in $\Omega^1$ as stated. Finally, setting $\eta=\theta$ it also requires $\sum_g \theta_g\tens\extd g=0$ which for a connected calculus (where $\ker\ \extd$ is spanned by $1$) means $\theta=\theta_e$.  In this case we have an exterior super-Hopf algebra $\Omega(kG)=kG\rbiprod B_-(\Lambda^1)$,  where we extend
the above with the relations of $B_-(\Lambda^1)$, the  super homomorphism property of $\Delta$ and the graded-derivation property of $\extd$.\endproof

Note that the standard part $\bar\Lambda^1\subseteq\Lambda^1_e$ in this case and hence $\bar\Lambda$ is the usual Grassmann algebra on $\bar\Lambda^1$ in keeping with the known theory of standard bicovariant calculi on $kG$.

\begin{proposition} Both $\Omega(G)=k(G)\rbiprod B_-(\Lambda^1)$ in Corollary~\ref{clagroupOmega} and $\Omega(kG)=kG\rbiprod B_-(\Lambda^1{}^*)$ in Corollary~\ref{kGOmega} for the dual crossed module $\Lambda^{1*}$, are augmented and are mutually dual as graded super-Hopf algebras.
\end{proposition}
\proof  This is now clear from Proposition~\ref{dualext} and Lemma~\ref{wordual}. As both sides of $\Omega^1$ are inner they both extend to $\Omega$ by Proposition~\ref{genworon} and hence both sides are augmented.  \endproof

Now we analyse the sub-shuffle calculus in Proposition~\ref{subshuffle-codiff}. Let $(\Lambda^1,\theta^*)$ defines a bicovariant coinner codifferential calculus on $kG.$
Choose a basis $\{e^{(i)}_g\}_{i=1}^{R_C}$ for each $\Lambda^1_g$ with $g$ belongs to some $C$ such that $e^{(1)}_g=(\theta^*_g)^*$ whenever $\theta^*_g\neq 0$. Then
\[ B_{\theta^*}(\Lambda^1)=\{\sum_{g_1,\dots,g_n\in G \atop i_1,\dots,i_n}\lambda_{g_1,\dots,g_n}^{i_1,\dots,i_n}e^{(i_1)}_{g_1}\tens e^{(i_2)}_{g_1}\tens\cdots \tens e^{(i_n)}_{g_n}\in \Sh_-(\Lambda^1)\ |\ 
\lambda_{g_1,\dots,g_n}^{i_1,\dots,i_n} \text{obey (A), (B)}\},
\]
where the conditions (A) are \[\sum_{a,b\in\bar{C}\atop ab=g}\lambda_{a,b,g_3,\dots,g_n}^{1,1,i_3\dots,i_n}=
\sum_{a,b\in\bar{C}\atop ab=g}\lambda_{g_1,a,b,g_4,\dots,g_n}^{i_1,1,1,i_4\dots,i_n}=\dots =
\sum_{a,b\in\bar{C}\atop ab=g}\lambda_{g_1,g_2,\dots,g_{n-2},a,b}^{i_1,i_2\dots,i_{n-2},1,1}=0,\]
for any $g_1,\dots g_{n},g\in G$ with $g\neq e$, and the conditions (B) are \[
\sum_{a,b\in\bar{C}\atop ab=e}\lambda_{a,b,g_1,\dots,g_{n-2}}^{1,1,i_1\dots,i_{n-2}}=
\sum_{a,b\in\bar{C}\atop ab=e}\lambda_{g_1,a,b,g_2,\dots,g_{n-2}}^{i_1,1,1,i_2\dots,i_{n-2}}=\dots =
\sum_{a,b\in\bar{C}\atop ab=e}\lambda_{g_1,g_2,\dots,g_{n-2},a,b}^{i_1,i_2\dots,i_{n-2},1,1},\] for any fixed $g_1,\dots g_{n-2}\in G$ and their fixed indices $i_1,\dots,i_{n-2}.$
The super-Hopf algebra $\Omega(kG)=kG\rbiprod B_{\theta^*}(\Lambda^1)$ is then a strongly bicovariant coinner codifferential exterior algebra on $kG$ with $i$ given from Proposition~\ref{subshuffle-codiff} by
\[ i(h\tens v_1\tens\cdots\tens v_n)=\<\theta^*,v_1\>hg\tens(v_2\tens\cdots \tens v_n)+(-1)^n h\tens(v_1\tens\cdots \tens v_{n-1})\<\theta^*,v_n\>,\]
for all $h\in G$, $v_1\tens\cdots\tens v_n\in B_{\theta^*}(\Lambda^1)$ and $v_1\in \Lambda^1_g$ for some $g\in G$. 
From Proposition~\ref{subshuffle-aug}, if $\theta^*_e=0,$ any cocycle data $\{w_g\in\Lambda^1_e\}_{g\in G}$ makes $kG\rbiprod B_{\theta^*}(\Lambda^1)$ augmented as $\<\theta^*,\tilde{\omega}(g)\>=0.$ Otherwise, if $\theta^*_e\neq 0$, we have
\begin{proposition}\label{kGaug}
Let $A=kG$ and $(\Lambda^1,\theta,\theta^*)$ define an augmented inner first order bicovariant differential calculus with $\theta\in \Lambda^1_e,\,\theta^*\in \Lambda^{1*}.$
Suppose $\<\theta^*,v\ra g\>=\<\theta^*,v\>$ for any $v\in \Lambda^1,\,g\in G,$ i.e. $\theta^*_e\in {}_G \Lambda^{1*}_e,$ then the coinner bicovariant codifferential calculus $kG\rbiprod B_{\theta^*}(\Lambda^1)$ is augmented with $\extd=[\theta,\ \}$ given in Proposition~\ref{shuffle-diff}.
\end{proposition}
\proof It is suffices to show that the conditions in Proposition~\ref{subshuffle-aug} are satisfied.
First, for any $v\in \Lambda^1,\,g\in G$, $\<\theta^*,v\ra g\>=\<\theta^*,v\>$  if and only if $\<\theta^*_e,v\ra g\>=\<\theta^*_e,v\>$ i.e. $\<\theta^*_e,v\ra(g-1)\>=0.$ For any $v\in \Lambda^1,$ without loss of generality, we assume $v\in \Lambda^1_g$ for some $g\in G.$ Then $\<\theta^*\tens\theta^*,v_0\tens\tilde{\omega}(v_1)\>=\<\theta^*\tens\theta^*,v\tens\tilde{\omega}(g)\>=\<\theta^*\tens\theta^*,v\tens\theta\ra(g-1)\>=\<\theta^*,v\>\<\theta^*_e,\theta\ra(g-1)\>=0.$ This completes the proof.\endproof

And starting with a calculus on $k(G)$, it immediately follows from Corollary~\ref{univdual} that we have
\begin{corollary}
Let $(\Lambda^1,\theta,\theta^*)$ define an augmented first order bicovariant differential calculus on $k(G)$ for a finite group $G$ with $\theta\in\Lambda^1$ right invariant and $\theta^*\in{\Lambda^{1*}}_e.$ Then $\Omega_\theta(k(G))$ in Corollary~\ref{kofGunivaug} and $kG\rbiprod B_{\theta}({\Lambda^1}^*)$ in Proposition~\ref{kGaug} are mutually dual as super-Hopf-algebra with the differential on one side dual to the codifferential on the other side.
\end{corollary}

We conclude by interpreting some of our results in terms of quivers. We start with the codifferentials on the path coalgebra $kQ^c$, which is dual to the differentials on the path algebra on $k(G).$ Note that an arrow in $kQ^c$ here is dual to an arrow in $kQ^a$ studied in Section 6.1.

Let $G$ be a group, $R,r: \mathcal{C}\to \mathbb{N}_0$ be two class functions such that $r_C\le R_C$ for each $C\in\mathcal{C}$, $r_C\in\{0,1\}$ and $r_{\{e\}}=0.$ Denote $\mathcal{C}$ the set of all the conjugacy classes of $G$, $\bar{C}=\cup_{C\in\mathcal{C}\atop r_C=1}C,$ and $e$ the identity of group $G.$ Let $Q=Q(G,R)$ and $\bar{Q}=Q(G,r)$ be the corresponding Hopf quivers. Colour this pair as before, then we have a digraph-quiver pair $\bar{Q}\subseteq Q.$ In this case, we have canonical left $G$-action $\cdot$ (not $\ast$) on $kQ_1$ defined by $h\cdot x \xrightarrow[]{(i)} y=hx\xrightarrow[]{(i)}hy,$ and canonical right $G$-action $\cdot$ on $k\bar{Q}_1$ as $x\xrightarrow[]{(1)}y\cdot h=xh\xrightarrow[]{(1)}yh.$

Now, we define a right-handed Hopf digraph-quiver triple $(\bar{Q}\subseteq Q,\cdot)$ (dual to $(\bar{Q}\subset Q,\ast)$) to be the above data together with a right $G$-action $\cdot$ on $kQ_1$ such that
\begin{itemize}
\item[1)] ${}^xkQ_1^y\cdot h={}^{xh}kQ_1^{yh}$,
\item[2)] $\cdot$ restricts on $\bar{Q}_1$ to the canonical right action,
\item[3)] $\cdot$ commutes with the canonical left action on $kQ_1.$
\end{itemize}

\begin{proposition}\label{pathcoalgebra}
If $(\bar{Q}\subseteq Q,\cdot)$ is a right-handed Hopf digraph-quiver triple on a finite group $G$ 
then there is an associated  bicovariant coinner codifferential calculus on $kG$,  and every bicovariant codifferential calculus on $kG$ is isomorphic to one of this form. Its structure is given by
\begin{gather*}
\Omega^1=kQ_1,\quad\Delta_L(x\xrightarrow[]{(i)}y)=x\tens x\xrightarrow[]{(i)}y,\quad\Delta_R(x\xrightarrow[]{(i)}y)=x\xrightarrow[]{(i)}y\tens y\\
\theta^*=\sum_{a\in\bar{C}}\sum_{x\in G} p_{x\xrightarrow[]{(1)} xa},\quad
g\cdot x\xrightarrow[]{(i)}y=gx\xrightarrow[]{(i)} gy,\quad x\xrightarrow[]{(i)}y\cdot g= \sum_j \lambda_{ij}(x^{-1}y,g)xg\xrightarrow[]{(j)}yg,
\end{gather*}
where $\{p_{x\xrightarrow[]{(i)} y}\}$ denotes the dual basis to $\{x\xrightarrow[]{(i)}y\}$ and the coefficients $\lambda_{ij}(a,g)$ are determined by $e\xrightarrow[]{(i)}a\cdot g=\sum_j\lambda_{ij}(a,g)g\xrightarrow[]{(j)}ag.$
\end{proposition}
\proof Obviously, $kQ_1$ is a $kG$-Hopf bimodule and $\Lambda^1=\oplus_{C\in\mathcal{C}\atop R_C\neq 0}\oplus_{a\in C}\Lambda^1_a$ with $\Lambda^1_a=k\text{-span}\{e\xrightarrow[]{(i)}a\}_{i=1}^{R_C}$. Note that $\theta^*_a:=\sum_{x\in G}p_{x\xrightarrow[]{(1)} xa}$ belongs to $\Lambda_a^*.$ By construction, $(e\xrightarrow[]{(1)} hah^{-1})\ra h=e\xrightarrow[]{(1)} a,$ for any $a\in\bar{C}$ and $h\in G.$
Then $\<h\la\theta^*_a,e\xrightarrow[]{(1)} g\>=\<\theta^*_a,(e\xrightarrow[]{(1)} g)\ra h\>=\<\theta^*_a,e\xrightarrow[]{(1)} h^{-1}gh\>=\delta_{hah^{-1},g}$ implies $h\la\theta^*_a=\theta^*_{hah^{-1}}$ for any $a\in\bar{C},h\in G.$ Hence $\{\iota_g=\theta^*_g\}_{g\in\bar{C}}$ provides the data for a coinner bicovariant codifferential calculus as in Proposition~\ref{kG} and \ref{kGdua}.

Conversely, suppose $(kG\rcross\Lambda^1,i)$ is a bicovariant codifferential calculus on $kG.$ Then $i$ is always coinner with $\theta^*\in{\Lambda^1}^*$ right-invariant and $\theta^*_e=0.$ We first associated a Hopf digraph-quiver triple to it. Define $R,r:\mathcal{C}\to\mathbb{N}_0$ as follows. Set $R_C=\dim \Lambda^1_g$ for $g\in C$, set $r_C=1$ if there exists $\theta^*_g\neq 0$ for some $g\in C$ and $r_C=0$ if otherwise.  Choose base $\{e^{(i)}_g\}$ for each $\Lambda^1_g$ such that $e^{(1)}_g=(\theta^*_g)^*$ whenever $\theta^*_g\neq 0.$ Consider the matrix representation: $e_g^{(i)}\ra h=\sum_j\lambda_{ij}(g,h)e^{(j)}_{h^{-1}gh}.$ Then $(x\tens e_g^{(i)}). h=\sum_j\lambda_{ij}(g,h)xh\tens e^{(j)}_{h^{-1}gh}.$ In particular, since $h\la\theta^*_g=\theta^*_{hgh^{-1}}$ for all $h\in G,$ then $e^{(1)}_a\ra h=e^{(1)}_{h^{-1}ah}$ for all $a\in\bar{C},h\in G$ where $\bar{C}=\{g\in G|\ \theta^*_g\neq0\},$ we have $x\tens e^{(1)}_a.h=xh\tens e^{(1)}_{h^{-1}ah}.$ Define right $G$-action $\cdot$ on $kQ_1$ by viewing $x\xrightarrow[]{(i)} xg$ as $x\tens e^{(i)}_g,$ it restricts to $k\bar{Q}_1$ a canonical action. Now, clearly, the map $\varphi:kG\rcross\Lambda^1\to kQ_1$, which maps $x\tens e^{(i)}_g$ to $x\xrightarrow[]{(i)}xg,$ provide the isomorphism between these two codifferential calculi.\endproof

Meanwhile, associated to $(Q,\cdot)$ i.e. the Hopf quiver $Q=Q(G,R)$ and $kG$-Hopf bimodule structure on $kQ_1,$ one has a length-graded super-Hopf algebra on $kQ^c\cong\mathrm{CoT}_{kQ_0}kQ_1$, see \cite{Hua}. The coproduct is de-concatenation and the product is given by quantum shuffle product. In particular, the product of paths $kQ_0$ of length $0$ and the paths $kQ_1$ of length $1$ is given by the bimodule structure. The product between the arrows in $kQ_1$ can be computed by the following formula $\alpha\cdot \beta=[\alpha\cdot s(\beta)][t(\alpha)\cdot\beta]-[s(\alpha)\cdot\beta][\alpha\cdot t(\beta)],$ where $s(\alpha),t(\alpha)$ denotes the source and target vertices of each arrow $\alpha$ and $[\ ]$'s connected by concatenation. For the formulae for higher orders, see (3.1) in~\cite{Hua}.
Note that the super-Hopf algebra $kQ^c$ is isomorphic to the super-Hopf algebra $kG\rbiprod \Sh_-(\Lambda^1)$ discussed before, where the left $1$-forms of $kQ_1$, i.e. the subspace spanned by all the arrows starting at the identity $e$ corresponds to $\Lambda^1$, and the super-Hopf algebra isomorphism is given by the natural mapping $x\xrightarrow[]{(i)}xg$ to $x\tens e^{(i)}_g.$

Start with the coinner bicovariant codifferential calculus given by the right-handed triple $(\bar{Q}\subseteq Q,\cdot)$ in Proposition~\ref{pathcoalgebra}, one can interpret $kG\rbiprod B_{\theta^*}(V)$ as a sub-super-Hopf algebra of $kQ^c,$ which has similar 'universal property' as in Proposition ~\ref{subshuffle-codiff}. Any element $\theta\in {}^ekQ_1{}^e$ defines an inner strongly bicovariant differential on this sub-super Hopf algebra, thus making it augmented as $\theta^*_e=0$. 

We illustrate the construction of Proposition~\ref{subshuffle-codiff},~\ref{subshuffle-aug} in terms of quivers by the following example.

\begin{example}\label{example}
Let $G=\mathbb{Z}_2=\<g\>$ with $Q=Q(\Z_2,R)$ and $\bar{Q}=Q(\Z_2,r)$ , where ramification data are given by $R=\{e\}+2\{g\}$ and $r=\{g\}.$ Denote the arrows $\alpha_i:e\xrightarrow[]{(i)} g,\ \beta_i:g\xrightarrow[]{(i)} e,\ \gamma:e\to e$ and $\rho:g\to g,\ i=1,2.$ Consider the right-handed Hopf digraph-quiver triple $(\bar{Q}\subseteq Q,\cdot)$ with $\cdot$ given by \[\alpha_i\cdot g=(-1)^i\beta_i,\ \beta_i\cdot g=(-1)^i\alpha_i,\ \gamma\cdot g=-\rho,\text{ and }\ \rho\cdot g=-\gamma.\]The path super-Hopf algebra $kQ^c$  is the $k$-space with grading $|\alpha_i|=|\beta_i|=|\gamma|=|\rho|=1$ spanned by all the paths (e.g. $\alpha_2\rho\rho\beta_1$) of $Q$ with super-coproduct given by de-concatenation
\begin{gather*}
\Delta e=e\tens e,\quad \Delta g=g\tens g,\quad \Delta\alpha_i=e\tens\alpha_i+\alpha_i\tens g,\quad \Delta\beta_i=g\tens\beta_i+\beta_i\tens e\\
\Delta\gamma=e\tens \gamma+\gamma\tens e,\quad \Delta\rho=g\tens\rho+\rho\tens g,\quad \Delta\beta_i\gamma=g\tens\beta_i\gamma+\beta_i\tens\gamma+\beta_i\gamma\tens e,\ {\rm etc}\\
\epsilon(e)=1,\quad\epsilon(g)=1,\quad\epsilon(p)=0, \text{ for any path $p$ of length greater than zero.}
\end{gather*}
The super-product of $kQ^c$ is given by the quantum shuffle product. Between arrows in $kQ_1$, we have
\begin{gather*}
\alpha_1\cdot\alpha_1=0,\quad\alpha_1\cdot\alpha_2=\alpha_1\beta_2-\alpha_2\beta_1,
\quad\alpha_2\cdot\alpha_1=\alpha_1\beta_2+\alpha_2\beta_1,
\quad\alpha_2\cdot\alpha_2=2\alpha_2\beta_2,\\
\quad\gamma\cdot\gamma=0,\quad\gamma\cdot\alpha_i=\alpha_i\rho+\gamma\alpha_i,
\quad\alpha_i\cdot\gamma=\alpha_i\rho-\gamma\alpha_i,\\
\quad\gamma\cdot\beta_i=(-1)^i(\rho\beta_i+\beta_i\gamma),
\quad\beta_i\cdot\gamma=\beta_i\gamma-(-1)^i\rho\beta_2,\ {\rm etc}.\\
\end{gather*}
The left invariant $1$-forms are $\Lambda^1=k\textrm{-span}\{\gamma,\alpha_i\}$ where $\Lambda^1_e=k\{\gamma\}$ and $\Lambda^1_g=k\{\alpha_1,\alpha_2\}$ with the coaction $\gamma\ra g=-\gamma,\ \alpha_i\ra g=(-1)^i\alpha_i,\ i=1,2.$ Then $kQ^c\cong k\Z_2\rbiprod \Sh_-(\Lambda^1)$ has cross relation $\gamma.g=-g.\gamma$ and $\alpha_i.g=(-1)^i g.\alpha_i.$

Now we compute $\Omega=k\Z_2\rbiprod B_{\theta^*}(\Lambda^1)$ from the analysis before Proposition~\ref{kGaug}. Here $\theta^*=p_{\alpha_1}+p_{\beta_1}\in {\Lambda^1_g}^*.$ So $\bar{C}=\{g\},$ and $\bar{C}\bar{C}=\{e\},$ then
\begin{align*}
B_{\theta^*}(\Lambda^1)=\{&\sum_{g_1,\dots,g_n\in \Z_2 \atop i_1,\dots,i_r}\lambda_{g_1,\dots,g_n}^{i_1,\dots,i_n}e^{(i_1)}_{g_1}\tens e^{(i_2)}_{g_1}\tens\cdots \tens e^{(i_n)}_{g_n}\in \Sh_-(\Lambda^1)\ |\ & \\
\lambda_{g,g,g_1,\dots,g_{n-2}}^{1,1,i_1\dots,i_{n-2}} &=\lambda_{g_1,g,g,g_2,\dots,g_{n-2}}^{i_1,1,1,i_2\dots,i_{n-2}}=\dots =\lambda_{g_1,g_2,\dots,g_{n-2},g,g}^{i_1,i_2\dots,i_{n-2},1,1},\
\forall g_1,\dots g_{n-2},  i_1,\dots,i_{n-2}.\ \}
\end{align*} Here   $e^{(i)}_{g_i}\in\{\gamma,\alpha_1,\alpha_2\}.$ One can show easily that $\Omega^2=kQ_2,$ $\Omega^n\subsetneq kQ_n$ for any $n\ge 3$ and write down a specific basis for each degree. For instance, if we denote the set of paths of length three by $Q_3,$ then the set
\begin{align*}
(Q_3\setminus\{\alpha_1\beta_1\gamma,\,\gamma\alpha_1\beta_1,\,\alpha_1\beta_1\alpha_2,\,\alpha_2\beta_1\alpha_1,\linebreak \beta_1\alpha_1\rho,\,\rho\beta_1\alpha_1,\,\beta_1\alpha_1\beta_2,\,\beta_2\alpha_1\beta_2\})\\
\cup \{\alpha_1\beta_1\gamma-\gamma\alpha_1\beta_1,\,\alpha_1\beta_1\alpha_2-\alpha_2\beta_1\alpha_1,\,\beta_1\alpha_1\rho-\rho\beta_1\alpha_1,\linebreak \beta_1\alpha_1\beta_2-\beta_2\alpha_1\beta_2\}
\end{align*}
provides a basis of $\Omega^3,$ hence $\dim \Omega^3=50<54.$  Similarly, one can show
$\Omega^4$ contains combinations of paths like $\alpha_1\beta_1\gamma\alpha_2-\gamma\alpha_1\beta_1\alpha_2,\,\alpha_1\beta_1\gamma\alpha_2-\gamma\alpha_2\alpha_1\beta_1$, etc, and $\dim\Omega^4=138<162.$

Thus $\Omega=k\Z_2\rbiprod B_{\theta^*}(\Lambda^1)$ is an infinite-dimensional coinner bicovariant  codifferential exterior algebra on $k\Z_2$ with the codifferential $i$ given by
\begin{gather*}
i(\alpha_1)=g-e,\quad i(\beta_1)=e-g,\quad i(\gamma)=i(\alpha_2)=i(\beta_2)=0\\
i(\alpha_1\beta_1)=\beta_1+\alpha_1,\quad i(\alpha_1\beta_2)=\beta_2,\quad i(\alpha_1\rho)=\rho,\ {\rm etc.}\\
\end{gather*}

Suppose in addition that  $(\Lambda^1,\theta)$ defines a non-trivial inner differential calculus on $k\Z_2$, where without of loss of generality we
assume $\theta=\gamma\in\Lambda^1_e.$ By Proposition~\ref{subshuffle-aug}, the exterior derivation extends to whole $\Omega$ automatically. So $(\Omega,i)$ is augmented with
inner differential given by $\extd=[\gamma,\ \}.$ In particular,
\begin{gather*}
\extd e=0,\quad\extd g=-2\rho\\
\extd\gamma=\extd\rho=0,\quad\extd \alpha_i=2\alpha_i\rho,\quad \extd\beta_i=(-1)^i2\rho\beta_i,\ \forall\,i=1,2,\cdots
\end{gather*}

\end{example}

\end{document}